\documentclass{amsart}
\usepackage{amssymb}
\usepackage{amsmath}
\usepackage{graphicx}
\usepackage{manfnt}
\usepackage{amsthm}
\usepackage[mathscr]{euscript}
\usepackage{mathtools}
\usepackage{epstopdf}
\usepackage{bigints}
\usepackage{color}
\usepackage{dsfont}
\usepackage{euscript}

\usepackage{tikz}
\usepackage{tikz-cd}
\usepackage{bm}
\usepackage{bbm}
\usetikzlibrary{arrows}
\usetikzlibrary{positioning}
\usepackage{mathrsfs}
\numberwithin{equation}{section}
\usepackage{hyperref}
\usepackage{enumitem}

\newcommand{\IC}{\bm{IC}}

\newcommand{\cHom}{\mc{H}om}
\newcommand{\cD}{\mathscr{D}}
\newcommand{\cH}{\mc{H}}
\newcommand{\cE}{\mc{E}}
\newcommand{\cF}{\mc{F}}
\newcommand{\cG}{\mc{G}}
\newcommand{\cL}{\mc{L}}
\newcommand{\cM}{\mc{M}}
\newcommand{\cO}{\mc{O}}

\newcommand{\C}{\mathbb{C}}
\newcommand{\Z}{\mathbb{Z}}
\newcommand{\R}{\mathbb{R}}
\newcommand{\Q}{\mathbb{Q}}
\newcommand{\HH}{\mathbb{H}}
\newcommand{\mf}{\mathfrak}
\newcommand{\mc}{\mathcal}
\newcommand{\bp}{\begin{pmatrix}}
\newcommand{\ep}{\end{pmatrix}}
\newcommand{\SBim}{\mathbb{S}\text{Bim}}

\DeclareMathOperator{\pr}{pr}
\DeclareMathOperator{\pt}{pt}
\DeclareMathOperator{\isom}{\cong}
\DeclareMathOperator{\id}{id}

\DeclareMathOperator{\Loc}{Loc}
\DeclareMathOperator{\ch}{ch}
\DeclareMathOperator{\End}{End}
\DeclareMathOperator{\Hom}{Hom}

\DeclareMathOperator{\Ind}{Ind}

\DeclareMathOperator{\stab}{stab}

\DeclareMathOperator{\Lie}{Lie}
\DeclareMathOperator{\Res}{Res}
\DeclareMathOperator{\For}{For}
\DeclareMathOperator{\im}{im}
\DeclareMathOperator{\rres}{res}

\theoremstyle{plain}
\newtheorem{theorem}{Theorem}[subsection]
\theoremstyle{definition}
\newtheorem{definition}[theorem]{Definition}
\theoremstyle{plain}
\newtheorem{lemma}[theorem]{Lemma}
\theoremstyle{plain}
\newtheorem{proposition}[theorem]{Proposition}
\theoremstyle{remark}

\theoremstyle{plain}
\newtheorem{corollary}[theorem]{Corollary}
\theoremstyle{remark}
\newtheorem{remark}[theorem]{Remark}
\theoremstyle{plain}
\newtheorem{example}[theorem]{Example}
\theoremstyle{definition}
\newtheorem{notation}[theorem]{Notation}

\makeatletter
\def\l@subsection{\@tocline{2}{0pt}{2.5pc}{5pc}{}}
\makeatother

\makeatletter
\def\l@subsubsection{\@tocline{3}{
0pt}{5pc}{5pc}{}}
\makeatother

\title{A categorification of the Lusztig--Vogan module}
\author{Scott Larson and Anna Romanov}
\date{}

\begin{document}

\maketitle

\begin{abstract}
    We construct two categorifications of the Lusztig--Vogan module associated to a real reductive algebraic group. The first categorification is given by semisimple complexes in an equivariant derived category, and the second is constructed as a module category over Soergel bimodules. Our categorifications are related by taking equivariant hypercohomology.  
\end{abstract}

\tableofcontents

\section{Introduction}
\label{sec: introduction}

\subsection{Overview}
\label{sec: Overview}
Let $G$ be a connected complex reductive algebraic group, $\theta$ a holomorphic involution of $G$, and $K$ a finite-index subgroup of the fixed point group $G^\theta$. From this data, Lusztig--Vogan constructed a module over the Hecke algebra of $G$ and a collection of polynomials that arise as entries in a change-of-basis matrix for this module \cite{Vogan3, LV}. These Kazhdan--Lusztig--Vogan (KLV) polynomials have both geometric and  representation theoretic significance: geometrically, they capture subtle information about the singularity of closures of $K$-orbits in the flag variety of $G$, and representation theoretically, they give character formulas for irreducible representations of the associated real reductive group  \(G_\R\) (cf., e.g., \cite{AdamsLeeuwenTrapaVogan2020}). 

In this paper, we take a Soergel bimodule approach to the Lusztig--Vogan module. In classic work \cite{Soergel90, Soergel2000}, Soergel explained that the hypercohomology functor applied to certain categories of constructible sheaves appearing in Lie theory is fully faithful. The upshot is that one can understand categories of constructible sheaves algebraically, as certain modules or bimodules. In this paper, we explain that a piece of the Lusztig--Vogan module can also be understood in this way after taking hypercohomology of a category of constructible sheaves. 


A major motivation for Soergel's work was to establish certain Koszul duality equivalences on the derived category of category $\mc{O}$. Similarly, one motivation for giving a bimodule theoretic interpretation of the Lusztig--Vogan module might be to establish Soergel's conjectures \cite{Soe00}, which generalize Koszul duality to real reductive groups. It is this point of view that has been taken in recent work of Bezrukavnikov--Vilonen \cite{BV}, which has results close to those of this paper\footnote{We comment on these results in more detail in Section \ref{sec: relation to the results of Bezrukavnikov-Vilonen}}. 

Another motivation for our work is to understand as explicitly as possible the Lusztig--Vogan module from a Soergel bimodule point of view, as a rich new source of module categories for Soergel bimodules. 


\subsection{Geometry of the Lusztig--Vogan module}
\label{sec: the geometry of the Lusztig--Vogan module}
The remainder of the introduction describes our results in more detail. Fix a Borel subgroup $B$ of $G$. The group $K$ acts on the flag variety $X=G/B$ with finitely many orbits, and the Lusztig--Vogan module is constructed using the geometry of these orbits. Denote by $\mathscr{D}$ the set of pairs $(\mc{O}, \mc{L})$ consisting of a $K$-orbit $\mc{O}$ and an irreducible $K$-equivariant local system on $\mc{O}$, and set $M_{LV}$ to be the free $\Z[q^{\pm 1}]$-module with basis $\mathscr{D}$. Let $(W,S)$ be the Coxeter system associated to $(G, B)$ and $\bm{H} = \bm{H}(W,S)$ the corresponding Hecke algebra. For each $s \in S$, one can define endomorphisms 
\[
T_s: M_{LV} \rightarrow M_{LV}
\]
using push-pull functors to partial flag varieties or via a case-by-case analysis of root types (Definition~\ref{def: Ts}). The operators $T_s$ give $M_{LV}$ the structure of an $\bm{H}$-module. 

The geometric construction of $M_{LV}$ leads naturally  to a categorification within $D^b_K(X)$, the $K$-equivariant derived category of sheaves of $\Q$-vector spaces\footnote{For simplicity, we use $\Q$-coefficients throughout; however, our arguments would also work for any field of characteristic zero.} on $X$. Associated to each Lusztig--Vogan parameter $(\mc{O}, \mc{L}) \in \mathscr{D}$ is an indecomposible object in $D^b_K(X)$: the $K$-equivariant intersection cohomology complex of $\overline{\mc{O}}$ with coefficients in $\mc{L}$ (shifted by $\dim \mc{O}$) which we denote by $\bm{IC}(\overline{\mc{O}}, \mc{L})$. We define 
\[
\mc{M}_{LV}^{ss} = \langle \bm{IC}(\overline{\mc{O}},\mc{L}) \mid (\mc{O}, \mc{L}) \in \mathscr{D} \rangle_{\oplus, [1]}
\]
to be the additive subcategory of $D^b_K(X)$ generated by $\IC$ sheaves under direct sum ($\oplus$) and grading shift ($[1]$). Via right convolution, $D^b_K(X)$ admits a categorical action of the $B$-equivariant derived category $D^b_B(X)$. Thanks to the decomposition theorem, $\mc{M}_{LV}^{ss}$ is preserved under this action, and thus obtains a categorical action of the geometric Hecke category $\mc{H}$ (Definition~\ref{def: geometric Hecke category}). Passing to Grothendieck groups\footnote{The notation $[\mc{A}]_\oplus$ denotes the split Grothendieck group of an additive category $\mc{A}$}, we recover the Lusztig--Vogan module. 
\begin{theorem}
(Theorem~\ref{thm: semisimple categorification}) As right $\bm{H}$-modules, $[\mc{M}_{LV}^{ss}]_{\oplus} \simeq M_{LV}$.
\end{theorem}

Recall that, by the Beilinson--Bernstein localization theorem, modules for the Lie algebra $\mf{g}$ of $G$ may be understood as global sections of $D$-modules on the flag variety. The irreducible $(\mf{g}, K)$-modules with trivial infinitesimal character occur as the global sections of $D$-modules corresponding to $\bm{IC}(\overline{\mc{O}}, \mc{L})$ under the Riemann--Hilbert correspondence. In this way, the set $\mathscr{D}$ parametrizes certain equivalence classes of irreducible representations of the real reductive group corresponding to the Harish-Chandra pair $(\mf{g}, K)$. This was a major motivation for the work of Lusztig and Vogan, and leads to the block decomposition (Definition~\ref{def: block equivalence}, Remark~\ref{rmk: blocks}) of the Lustig--Vogan module. Blocks of $M_{LV}$ keep track of which irreducible $(\mf{g}, K)$-modules are related via composition series of standard modules. The block of the trivial representation, $M_{LV}^0$, arises in the geometric formulation above as the submodule of $M_{LV}$ generated by trivial local systems on closed orbits (Theorem~\ref{lem: closed orbits generate the block}). Set
\[
\mc{M}_{LV}^0:= \langle \bm{IC}(\mc{Q}, \Q_\mc{Q}) * \mc{H} \mid \mc{Q} \text{ a closed $K$-orbit} \rangle_{\oplus, \ominus, [1]}
\]
to be the subcategory of $\mc{M}_{LV}^{ss}$ generated by $\IC$ sheaves corresponding to trivial local systems on closed orbits. (Here $\mathbb{Q}_\mc{Q}$ is the constant sheaf on $\mc{Q}$ and the subscript $\oplus, \ominus, [1]$ denotes the the Karoubi envelope.) 
\begin{theorem}
(Theorem~\ref{thm: geometric categorification}) As right $\bm{H}$-modules, $[\mc{M}_{LV}^0]_\oplus \simeq M_{LV}^0.$
\end{theorem}

The categories $\mc{M}_{LV}^{ss}$ and $\mc{M}_{LV}^0$ provide geometric categorifications of the Lusztig--Vogan module $M_{LV}$ and its trivial block $M_{LV}^0$. The existence of these categorifications is not surprising---they are just reformulations of the construction in \cite{LV} and classical theorems in \cite{Vogan-green-book} using different geometric language. Our interest in formulating the Lusztig--Vogan module in these geometric terms is twofold: (1) we believe that the equivariant derived category provides the most natural language to describe Lusztig--Vogan's construction, and (2) this formulation allows us to give a bimodule description of the Lusztig--Vogan module which relates it to the monoidal category of Soergel bimodules. 

\subsection{A Soergel bimodule approach}
\label{sec: A Soergel bimodule approach}

Our tool for moving from the $K$-equivariant derived category to a category of bimodules is the hypercohomology functor (Definition~\ref{eq: equivariant hypercohomology}): 
\[
\mathbb{H}_K^\bullet(-):= \Hom_K^\bullet( \Q_X, - ): D^b_K(X) \rightarrow H^*_K(X)\mathrm{-gmod}.
\]
To make this approach work, we add the assumption $K$ is the {\em identity component} of $G^\theta$. We suspect that the main results hold for disconnected $K$ as well, but different arguments would be needed. Under this connectedness assumption on $K$, a standard computation (Lemma~\ref{lem: K equivariant cohomology of X}) proves that the $K$-equivariant cohomology of the flag variety is 
\[
H^*_K(X) = P^{W_K} \otimes_{R^W} R,
\]
where $P = S(X(T_K)_\Q)$ is the symmetric algebra on the $\Q$-span of the cocharacter lattice of a maximal torus $T_K$ in $K$, $R = S(X(T)_\Q)$ is the analogous ring for a maximal $\theta$-stable torus $T$ in $G$, and $W_K$ is the Weyl group of $K$. This allows us to view vector spaces in the image of $\mathbb{H}^\bullet_K$ as graded $(P^{W_K}, R)$-bimodules. 

Our algebraic Lusztig--Vogan category is defined to be the essential image of the functor $\mathbb{H}_K^\bullet$ when restricted to $\mc{M}_{LV}^0$. It can be described in an explicit combinatorial way by computing the hypercohomology of $\IC$ sheaves on closed orbits (Lemma~\ref{lem: standards are equiv cohomology}). For a closed $K$-orbit $\mc{Q}$, there exists an element $x \in W$ such that
\[
\mathbb{H}_K^\bullet(\bm{IC}(\mc{Q}, \Q_\mc{Q})) = P_x,
\]
where $P_x$ is the standard $(P^{W_K}, R)$-bimodule (Definition~\ref{def: standard modules}) corresponding to $x$. We define a category
\[
\mc{N}_{LV}^0:=\langle P_x \otimes_R \SBim \mid x\in W^\theta \rangle_{\oplus, \ominus, (1)}
\]
of $(P^{W_K}, R)$-bimodules which is generated by standard modules under the natural right action of the category of Soergel bimodules (denoted $\SBim$, Definition~\ref{def: SBim}). Here $W^\theta$ is the set of elements of $W$ fixed under the involution induced from $\theta$. These are precisely the elements of $W$ corresponding to $T$-fixed points in $X$ which are found in closed $K$-orbits. 

Though motivated by geometry, this construction can be done completely combinatorially with only the input data of the group $G$ and involution $\theta$. From the perspective of Soergel bimodules, this provides an interesting class of module categories for $\SBim$ coming from real groups. It also potentially leads to ``exotic'' Lusztig--Vogan modules attached to ``exotic'' involutions (e.g.\ of non-crystallographic Weyl groups). 

The main result of our paper is the following theorem. 
\begin{theorem}
\label{thm: main theorem} (Theorem~\ref{thm: image of hypercohomology}, Theorem~\ref{thm: fully faithful})
The hypercohomology functor 
\[
\mathbb{H}_K^\bullet: \mc{M}_{LV}^0 \rightarrow \mc{N}_{LV}^0
\]
is an equivalence of categories. 
\end{theorem}

\begin{remark}
This result can also be deduced from Theorem 3.12 in \cite{BV}.
\end{remark}

The proof of Theorem~\ref{thm: main theorem} relies on two main ingredients: the theory of parity sheaves developed in \cite{ParitySheaves}, and localization to torus fixed points as in \cite{GKM}. We find it interesting that this appears to be the first application of parity sheaf techniques to $K$-orbits on flag varieties, as they provide a powerful tool which has been used in many other settings in geometric representation theory. It is also of note that when specialized to the case of a complex group (i.e., the involution of $G \times G$ which switches the factors), our argument provides an alternate proof of \cite[Erweiterungssatz 17]{Soergel90} which does not rely on known facts about dimensions of modules in category $\mc{O}$.  

Theorem~\ref{thm: main theorem} immediately implies that $\mc{N}_{LV}^0$ categorifies the trivial block of the Lusztig--Vogan module. 
\begin{corollary}
(Corollary \ref{cor: main theorem}) As right $\bm{H}$-modules, $[\mc{N}_{LV}^0]_\oplus \simeq M_{LV}^0$.
\end{corollary}

A conjecture of Soergel \cite{Soe00} predicts that associated to the category $\mc{N}_{LV}^0$ (or $\mc{M}_{LV}^0$) is a Koszul dual category of Harish-Chandra modules for a real form of the Langlands dual group of $G$. The duality on the level of categories should provide a categorification of Vogan's character duality \cite{Vogan4}, which is a duality on the level of $\bm{H}$-modules. A version of this relationship is established in \cite{BV}. 

One might hope that by applying the functor $\mathbb{H}_K^\bullet$ to the semisimple category $\mc{M}_{LV}^{ss}$, we could establish a bimodule description of the entire Lusztig--Vogan module. Unfortunately, $\mathbb{H}_K^\bullet$ vanishes outside of $\mc{M}_{LV}^0$ (see Example~\ref{example: hypercohomology vanishes on mobius band}), so this strategy does not work. 

\subsection{Application: Resolutions of singularities}
\label{Application: Resolutions of singularities}
We can view the results in this paper as studying continuous actions of real reductive algebraic groups.
Gelfand's program asserts that one should first linearize such actions by considering vector spaces related to functions on the original space.
This philosophy lead to the definition of \(\cD\) through work of Harish-Chandra and Vogan \cite{Vogan3}.
In particular, continuous actions of \(G_\R\) are studied by considering the easier problem of an algebraic group \(K\) acting on an algebraic variety \(X\).
It is interesting to observe that in the case of \(G_\R\) acting on the flag variety, Matsuki duality gives a perfect relationship between orbit structures (cf., \cite{Wolf1969},\cite{Matsuki1979}).
Moreover, singularities of \(K\)-orbit closures in \(X\) are related to representation theory of \(G_\R\) through KLV polynomials and characteristic cycles of \(D\)-modules.

Closures of \(K\)-orbits in \(X\) have complicated singularities.
In \cite{BarbaschEvens1994}, equivariant resolutions of singularities were constructed to study the geometry of these varieties, and in \cite{Larson2020} equivariant proper maps were further analyzed.
As an application of our geometric categorification of the Lusztig--Vogan module, we give a formula in Theorem~\ref{theorem: fibers} to compute the cohomology of fibers of the resolutions introduced in \cite{Larson2020} using the character map
\begin{equation}
m\,\colon D_K^b(X)\to M_{LV}
\end{equation}
appearing in the proof of Theorem~\ref{thm: semisimple categorification}.
This result can be viewed through the paradigm of Gelfand and Harish-Chandra as studying equivariant resolutions of singularities first via linearization, then through algebraic methods using our algebraic categorification.

\subsection{Relation to the results of Bezrukavnikov--Vilonen}
\label{sec: relation to the results of Bezrukavnikov-Vilonen}
Recent results of Bezrukavnikov--Vilonen \cite{BV} establish an important special case of Soergel's conjecture on  a categorification of Vogan duality (see \ref{sec: Overview}). One side of their equivalence takes place in the equivariant derived category, and a key part of their argument involves the fully-faithfullness of a global sections functor. This part of their argument involves a theorem equivalent to Theorem \ref{thm: main theorem} above. Our work on this project began before we were aware of their results, and our proofs differ. (In particular, their proof involves a reduction to low ranks, and ours does not, see \ref{rmk: further remarks on BV relationship} below). For the benefit of the reader, we comment here on the relation between our results. 

The fully-faithfullness result alluded to above is \cite[Thm. 3.2]{BV}. They prove that (in their notation)
\[
R\Gamma_{\check{K}}:  \EuScript{S}^\mathrm{gr} \rightarrow \mathrm{Coh}^{\C^\times}(\mathfrak{B})
\]
is fully faithful. The relation between this result and Theorem \ref{thm: main theorem} is as follows. 
\begin{enumerate}
    \item Their category $\EuScript{S}^\mathrm{gr}$ is equivalent to our category $\mc{M}_\mathrm{LV}^0$. Indeed, both categories consist of certain $K$-equivariant semisimple complexes on $G/B$, and \cite[Prop. 3.6]{BV} establishes that $\EuScript{S}^\mathrm{gr}$ is generated by $\bm{IC}$-sheaves corresponding to trivial local systems on closed orbits. 
    \item There is an equivalence 
    \begin{equation}
        \label{eq: relation between bimodule categories}
    \mathrm{Coh}^{\C^\times}(\mathfrak{B}) \simeq P^K \otimes_{R^W}R \mathrm{-gmod},
    \end{equation}
    where we are using notation from \cite{BV} on the left and our notation on the right. This is because $\C[ \mathfrak{B}] = P^K \otimes_{R^W} R$ by \cite[Lem. 3.1]{BV}. The grading on the right hand side of (\ref{eq: relation between bimodule categories}) is due to the $\C^\times$-action on the left hand side of (\ref{eq: relation between bimodule categories}). The weights agree by the definition of the $\C^\times$-action in \cite[Thm. 3.2]{BV}.
    \item In \cite{BV}, the category $\EuScript{S}^\mathrm{gr}$ is defined using $K$-equivariant $D$-modules on $G/B$. It is then implicitly identified with a subcategory of the (constructible) equivariant derived category $D^b_K(G/B)$. In contrast, we work exclusively within the equivariant derived category. 
\end{enumerate}
We complete this section with some further remarks. 
\begin{remark}
\label{rmk: further remarks on BV relationship}
\begin{enumerate}[label=(\roman*)]
\item  Bezrukavnikov--Vilonen's proof uses reduction to small rank and a computation, see \cite[\S3.8]{BV}. Our proof avoids this reduction.
\item Our theorem describes the image in terms of Soergel bimodules. One can deduce a similar description from \cite{BV}. Note that our Soergel bimodule action corresponds to their convolution with $\alpha$-lines, see \cite[\S3.3]{BV}. 
\item For the purposes of Bezrukavnikov--Vilonen, it is necessary to consider all finite-index subgroups of $K$, included disconnected ones. They overcome this obstacle by imposing equivariance under a finite abelian group.
\end{enumerate}
\end{remark}

\subsection{Acknowledgements}
We would like to thank Jeff Adams, Adam Brown, Bill Graham, Dragan Milicic, Laura Rider, Peter Trapa, David Vogan and Geordie Williamson for useful discussions and comments. We would like to thank the University of Georgia for providing us the opportunity to first meet and discuss \(K\)-orbits, and the first author further thanks this university for providing an environment conducive to learning material related to this paper. The second author completed the final stages of this work on the University of Sydney's One Tree Island research station, which she would like to thank for providing an exceptional working environment. The second author acknowledges support from the National Science Foundation Award No. 1803059 and the Australian Research Council grant DP170101579.

\section{Index of notation}
\label{sec: Index of notation}

\noindent
$G$: connected reductive complex algebraic group (\ref{sec: Overview})

\vspace{2mm}
\noindent
$\theta$: holomorphic involution of $G$ (\ref{sec: Overview})

\vspace{2mm}
\noindent
$K$: finite-index subgroup of $G^\theta$ in Sections \ref{sec: preliminaries} - \ref{sec: a categorification of the trivial block}, the identity component of $G^\theta$ in Sections \ref{sec: An algebraic categorification} - \ref{sec: the hypercohomology functor}

\vspace{2mm}
\noindent
$T_K$: maximal torus in $K$ (\ref{sec: torus fixed points and closed K-orbits})

\vspace{2mm}
\noindent
$T=Z_G(T_K)$: $\theta$-stable maximal torus in $G$ (Lemma~\ref{lem: centralizer is a torus})

\vspace{2mm}
\noindent
$B$: Borel subgroup of $G$ containing $T$ (\ref{sec: introduction})

\vspace{2mm}
\noindent
$B_K = B \cap K$: Borel subgroup of $K$ (\ref{sec: torus fixed points and closed K-orbits})

\vspace{2mm}
\noindent
$\mf{g}, \mf{k}, \mf{t}, \mf{t}_K$: Lie algebras of $G, K, T, T_K$ (\ref{sec: torus fixed points and closed K-orbits})

\vspace{2mm}
\noindent 
$W = N_G(T)/T$, $W_K = N_K(T_K)/T_K$: Weyl groups of $G, K$ (Lemma~\ref{lem: Weyl group embedding})

\vspace{2mm}
\noindent
$W^\theta$: group of elements fixed by the involution on $W$ induced by $\theta$ (\ref{sec: torus fixed points and closed K-orbits})

\vspace{2mm}
\noindent
$\bm{H} = \bm{H}(W,S)$: Hecke algebra of $(W,S)$ (\ref{sec: The Hecke algebra})

\vspace{2mm}
\noindent
$\{\delta_x \mid x \in W\}$: standard basis of the Hecke algebra $\bm{H}$ (\ref{sec: The Hecke algebra})

\vspace{2mm}
\noindent
$\{b_x \mid x \in W\}$: Kazhdan--Lusztig basis of the Hecke algebra $\bm{H}$ (Theorem~\ref{thm: KL basis})

\vspace{2mm}
\noindent
$X(T), X(T_K)$: character lattices of $T$, $T_K$ (\ref{sec: A Soergel bimodule approach})

\vspace{2mm}
\noindent
$X(T)_\Q= X(T) \otimes_\Z \Q$, $X(T_K)_\Q= X(T_K) \otimes_\Z \Q$: $\Q$-vector spaces spanned by character lattices (\ref{sec: A Soergel bimodule approach})

\vspace{2mm}
\noindent
$R= S(X(T)_\Q)$: Symmetric algebra on $X(T)_\Q$ (\ref{sec: Soergel bimodules})

\vspace{2mm}
\noindent
$P=S(X(T_K)_\Q)$: Symmetric algebra on $X(T_K)_\Q$ (\ref{sec: An algebraic categorification})

\vspace{2mm}
\noindent
$B_s = R \otimes_{R^s} R(1)$, $BS(\underline{w}) = B_{s_1} \otimes_R \cdots \otimes_R B_{s_n}$: Bott--Samelson bimodules corresponding to $s \in S$ and $\underline{w}=(s_1, \ldots s_n)$ (Equation \ref{eq: Bs}, Definition~\ref{def: Bott-Samelson bimodule})

\vspace{2mm}
\noindent
$\SBim$: category of Soergel bimodules (Definition~\ref{def: SBim})

\vspace{2mm}
\noindent
$P^K:= P^{W_K}$: $W_K$-invariant polynomials in $P$ (\ref{sec: An algebraic categorification})

\vspace{2mm}
\noindent
$P_x$: standard bimodule in $(P^K, R)\mathrm{-gbim}$ corresponding to $x \in W$ (Definition~\ref{def: standard modules})

\vspace{2mm}
\noindent
$\mc{N}_{LV}^0$: subcategory in $(P^K, R)\mathrm{-gbim}$ generated by standard bimodules (Definition~\ref{def: LV category})

\vspace{2mm}
\noindent
$X=G/B$: the flag variety of $G$ (\ref{sec: The geometric Hecke category})

\vspace{2mm}
\noindent
$X_w=BwB/B$: Schubert cell corresponding to $w \in W$

\vspace{2mm}
\noindent
$\mc{H}$: the geometric Hecke category (Definition~\ref{def: geometric Hecke category}, Equation \ref{eq: second incarnation Hecke cateory})

\vspace{2mm}
\noindent
$\bm{IC}_w$, $\bm{IC}(\overline{\mc{O}}, \mc{L})$, $\bm{IC}_\mc{Q}$: intersection cohomology complexes on orbit closures $\overline{X}_w, \overline{\mc{O}}$, $\overline{\mc{Q}}=\mc{Q}$ (Equation \ref{eq: IC sheaf}, Equation \ref{eq: K-equivariant IC sheaves}, Equation \ref{eq: IC_Q})

\vspace{2mm}
\noindent
$[\mathcal{A}]_\oplus$: split Grothendieck group of an additive category $\mc{A}$ 

\vspace{2mm}
\noindent
$\mathbb{H}_H^\bullet: D^b_H(X) \rightarrow H^*_H(X)\mathrm{-gmod}$: equivariant hypercohomology functor

\vspace{2mm}
\noindent
$\mathscr{D} = \{(\mc{O}, \mc{L}) \mid \mc{O} \text{ is a $K$ orbit on $X$}, \mc{L} \in \Loc_K(\mc{O})\}$: Lusztig--Vogan parameters (\ref{eq: D})

\vspace{2mm}
\noindent 
$M_{LV}$: the Lusztig--Vogan module of the Hecke algebra (Equation \eqref{eq: LV module as a vector space}, Definition~\ref{def: Ts}, Lemma~\ref{lem: Ts give a Hecke algebra action})

\vspace{2mm}
\noindent
$M_{LV}^0$: trivial block of the Lusztig--Vogan module (Theorem~\ref{lem: closed orbits generate the block})

\vspace{2mm}
\noindent
$\mc{M}_{LV}^{ss}$: subcategory of semisimple complexes in $D^b_K(X)$ (Definition~\ref{def: semisimple category})

\vspace{2mm}
\noindent
$\mc{M}_{LV}^0$: subcategory in $D^b_K(X)$ generated by trivial local systems on closed orbits (Definition~\ref{def: geometric LV category})

\vspace{2mm}
\noindent
$\pi_s: X \rightarrow G/P_s$: projection to partial flag variety (\ref{eq: pi s})

\vspace{2mm}
\noindent
$\Theta_s= \pi_s^* \pi_{s*} [1]: D^b_K(X) \rightarrow D^b_K(X)$: push-pull functor (\ref{eq: Theta})


\section{Preliminaries}
\label{sec: preliminaries}

We begin by establishing our notation and conventions. We encourage readers who are familiar with this material to skip to Section \ref{sec: a geometric categorification} and refer back to this section only when needed.  


\subsection{The Hecke algebra}
\label{sec: The Hecke algebra}
Let $(W, S)$ be the Coxeter system associated to the pair $(G, B)$. We denote by $\ell:W \rightarrow \mathbb{N}$ the length function on $W$ and $\leq$ the Bruhat order. A sequence $\underline{w}=(s_1, s_2, \ldots, s_m)$ of simple reflections $s_i \in S$ is called an {\em expression} for the group element $w$ if $w = s_1 s_2 \cdots s_m \in W$. Notationally we distinguish between expressions and group elements by underlining. We say an expression $\underline{w}=(s_1, s_2, \ldots, s_m)$ is {\em reduced} if $m=\ell(w)$. 

The {\em Hecke algebra} $\bm{H}=\bm{H}(W,S)$ of the Coxeter system $(W,S)$ is the unital associative $\Z[v^{\pm 1}]$-algebra generated by symbols $\{\delta_s \mid s \in S\}$ subject to relations 
\begin{enumerate}[label=(\roman*)]
\item (quadratic) $\delta_s^2 = 1+(v^{-1}-v)\delta_s$ for all $s \in S$, and
\item (braid) $\delta_s \delta_t \delta_s \ldots = \delta_t \delta_s \delta_t\ldots $ for all $s \neq t \in S$,  
\end{enumerate}
where the products in (ii) consist of $m_{st}$ elements on each side of the equality. 

\begin{remark}
\label{rmk: T_s formulation of Hecke algebra}
Alternatively, the Hecke algebra can be presented as the unital $\Z[q^{\pm 1}]$-algebra generated by symbols $\{T_s \mid s \in S \}$ subject to relations
\begin{enumerate} [label=(\roman*)]
    \item $T_s^2 = (q-1) T_s + q$ for all $s \in S$, and 
    \item $T_s T_t T_s \ldots = T_t T_s T_t \ldots$ for all $s \neq t \in S$,  
\end{enumerate}
where each product in (ii) consists of $m_{st}$ elements. Our initial presentation can be obtained from this one by setting $v=q^{-1/2}$ and $\delta_s = vT_s$. 
\end{remark}

\vspace{2mm} 
The elements $\{\delta_s \mid s \in S\}$ are invertible, with inverses given by $\delta_s^{-1}=\delta_s+(v-v^{-1})$. For $x\neq e \in W$, define an element $\delta_x \in \bm{H}$ by choosing a reduced expression $\underline{x}=(s_1, \ldots, s_m)$ of $x$ and setting $\delta_x:=\delta_{s_1}\cdots \delta_{s_m}$. One can show that this element is independent of choice of reduced expression. Define $\delta_e:=1$. The set $\{\delta_x \mid x \in W\}$ is the {\em standard basis} of $\bm{H}$. 

The Hecke algebra comes equipped with a {\em bar involution} 
\[
\bar{ } : \bm{H} \rightarrow \bm{H}, \hspace{2mm}
h \mapsto \bar{h}
\]
which is defined to be the unique ring homomorphism sending $v \mapsto v^{-1}$ and $\delta_x \mapsto \delta_{x^{-1}}$. We say an element $h \in \bm{H}$ is {\em self-dual} if $\overline{h}=h$. 

\begin{theorem}
\label{thm: KL basis}
\cite{KL} \cite[Thm.\ 2.1]{Soergel97} For each $x \in W$, there exists a unique self-dual element $b_x \in \bm{H}$ such that $b_x = \delta_x+ \sum_{y<x}h_{y,x}\delta_y$ for $h_{y,x} \in v\Z[v]$. 
\end{theorem}
The set $\{b_x \mid x \in W\}$ is the {\em Kazhdan--Lusztig basis} of $\bm{H}$. The polynomials $h_{y,x} \in v\Z[v]$ are {\em Kazhdan--Lusztig polynomials}. One can check directly that the Kazhdan--Lusztig basis element corresponding to a simple reflection $s \in S$ is $b_s=\delta_s + v$. The other Kazhdan--Lusztig basis elements can be computed recursively using Theorem~\ref{thm: KL basis}.

\subsection{The Lusztig--Vogan module}
\label{sec: the lustig-vogan module}

The mathematical object at the heart of this project is a representation of the Hecke algebra introduced in \cite{LV} and \cite{Vogan3}, which we refer to as the {\em Lusztig--Vogan module}. This representation is constructed using the geometry of local systems on $K$-orbits in the flag variety of $G$. From this representation one can extract character formulas for irreducible representations of $G_\R$. In this section we sketch the construction of Lusztig--Vogan. We refer the reader to \cite{LV} for full details of the construction and \cite{Vogan3} for the relationship to representation theory of $G_\R$. 

Let $G$, $B$, $\theta$ and $K$ be as in Section \ref{sec: the geometry of the Lusztig--Vogan module}. The first ingredient in the construction of Lusztig--Vogan is a parameterizing set for a basis of their $\bm{H}=\bm{H}(W,S)$-module. This set is defined geometrically, but it also has significance in the representation theory of $G_\R$: it parameterizes a certain set of equivalence classes of irreducible representations of $G_\R$.  See \cite{Vogan1, Vogan2, Vogan3, Vogan4} for the full story or \cite{penrose} for a geometric survey.  

Let $X=G/B$ be the flag variety of $G$. The group $K$ acts on $X$ with finitely many orbits \cite[Prop 5.1]{penrose}. For a $K$-orbit $\mc{O}$, denote by $\Loc_K(\mc{O})$ the set of irreducible $K$-equivariant local systems on $\mc{O}$. Set
\begin{equation}
    \label{eq: D}
    \mathscr{D} = \{(\mc{O}, \mc{L}) \mid  \mc{O} \text{ is a $K$-orbit on $X$}, \mc{L} \in \Loc_K(\mc{O})\}.
\end{equation}
In a pair $(\mc{O}, \mc{L}) \in \mathscr{D}$, the local system $\mc{L}$ determines the orbit $\mc{O}$, so we sometimes write elements of $\mathscr{D}$ simply as $\mc{L} \in \mathscr{D}$. 

Let $q$ be an indeterminate. Set \begin{equation}
    \label{eq: LV module as a vector space}
    M_{LV}=\bigoplus_{\mc{L} \in \mathscr{D}} \Z[q^{\pm 1}] \mc{L}
\end{equation}
to be the free $\Z[q^{\pm 1}]$-module with basis $\mathscr{D}$. 

The second ingredient in the construction is a set of endomorphisms $\{T_s \mid s \in S\}$ of $M_{LV}$ which give it the structure of an $\bm{H}$-module. These endomorphisms can be constructed in two ways: via push-pull functors to partial flag varieties, or formulaically based on root type of points in the orbit. Both constructions are equivalent. We will describe the latter, and refer the reader to \cite[\S 3]{LV} for the details on the former. 

 For a simple reflection $s \in S$, denote by
\begin{align}
\label{eq: notations for s}
    &P_s = \text{ standard parabolic subgroup of $G$ of type $s$},\\
    \label{eq: pi s}
    &\pi_s: X \rightarrow G/P_s \text{ natural projection map, and } \\
   &L_x^s = \pi_s^{-1}(\pi_s(x)) \simeq \mathbb{P}^1 \text{ the line of type $s$ through $x \in X$}.
\end{align}
Further, for a $K$-orbit $\mc{O}$, set
\begin{equation}
\label{eq: O hat}
    \widehat{\mc{O}}_s:= \bigcup_{x \in \mc{O}} L_x^s.
\end{equation}

\begin{definition}
\label{def: Ts}
 Let $(\mc{O}, \mc{L}) \in \mathscr{D}$, fix $x \in \mc{O}$, and use notation (\ref{eq: notations for s}) - (\ref{eq: O hat}). For $s \in S$, define $T_s: M_{LV} \rightarrow M_{LV}$ to be the endomorphism determined by the following formulas\footnote{We write the endomorphism $T_s$ on the right to indicate that the corresponding action of the Hecke algebra is a right action.} (enumerated and labeled as to align with \cite[Def.\ 6.4]{Vogan3}). 
\begin{itemize}
    \item[(a)] (compact imaginary) If $L_x^s \subseteq \mc{O}$, then
    \[
    \mc{L} T_s  = q \mc{L}.
    \]
    \item[(b1)] (complex ascent) If $L_x^s \cap \mc{O} = \{x\}$ and $\widehat{\mc{O}} \smallsetminus \mc{O}$ is a single $K$-orbit, then 
    \[
    \mc{L} T_s = \widehat{\mc{L}}|_{\widehat{\mc{O}} \smallsetminus \mc{O}},
    \]
    where $\widehat{\mc{L}}$ is the unique extension (as a local system) of $\mc{L}$ to $\widehat{\mc{O}}$. 
    \item[(b2)] (complex descent) If $L_x^s \cap \mc{O} = L_x^s \smallsetminus \{\text{point}\}$, then $\widehat{\mc{O}} \smallsetminus \mc{O}$ is a single $K$-orbit and we define
    \[
    \mc{L} T_s = (q-1)\mc{L} + q(\widehat{\mc{L}}|_{\widehat{\mc{O}} \smallsetminus \mc{O}}),
    \]
    where $\widehat{\mc{L}}$ is the unique extension of $\mc{L}$ to $\widehat{\mc{O}}$.
    \item[(c1)] (type II noncompact imaginary) If $L_x^s \cap \mc{O} = \{x, x'\}$, then $\widehat{\mc{O}} \smallsetminus \mc{O}$ is a single $K$-orbit and $\mc{L}$ has two distinct extensions to $\widehat{\mc{O}}$, which we denote $\widehat{\mc{L}}_1$ and $\widehat{\mc{L}}_2$. In this case, we define
    \[
    \mc{L} T_s = \mc{L} + (\widehat{\mc{L}}_1 + \widehat{\mc{L}}_2)|_{\widehat{\mc{O}} \smallsetminus \mc{O}}.
    \]
    \item[(c2)] (type II real) If $L_x^s \cap \mc{O} = L_x^s \smallsetminus \{\text{two points in one $K$-orbit}\}$ and $\mc{L}$ extends to $\widehat{\mc{O}}$, then
    \[
    \mc{L} T_s = (q-1) \mc{L} - \widehat{\mc{L}}_2|_{\mc{O}} + (q-1) \mc{L}',
    \]
    where $\mc{L}'=\widehat{\mc{L}}|_{\widehat{\mc{O}} \smallsetminus \mc{O}}$ (here $\widehat{\mc{L}}$ is the extension of $\mc{L}$ to $\widehat{\mc{O}}$) and  $\widehat{\mc{L}}_2$ is the other extension of $\mc{L}'$ to $\widehat{\mc{O}}$.
    \item[(d1)] (type I noncompact imaginary) If $L_x^s \cap \mc{O} = \{x\}$ and $\widehat{\mc{O}} \smallsetminus \mc{O}$ is the union of two $K$-orbits $\mc{O}'$ and $\mc{O}''$ (labeled so that $\dim \mc{O} = \dim \mc{O}'' = \dim \mc{O}'-1$), then 
    \[
    \mc{L} T_s = \widehat{\mc{L}}|_{\mc{O}'} + \widehat{\mc{L}}|_{\mc{O}''},
    \]
    where $\widehat{\mc{L}}$ is the unique extension of $\mc{L}$ to $\widehat{\mc{O}}$. 
    \item[(d2)] (type I real) If $L_x^s \cap \mc{O} = L_x^s \smallsetminus \{\text{two points in two $K$-orbits $\mc{O}'$ and $\mc{O}''$}\}$ and $\mc{L}$ extends to $\widehat{\mc{O}}$, then 
    \[
    \mc{L} T_s = (q-2)\mc{L} + (q-1)(\widehat{\mc{L}}|_{\mc{O}'} + \widehat{\mc{L}}|_{\mc{O}''}),
    \]
    where $\widehat{\mc{L}}$ is the unique extension of $\mc{L}$ to $\widehat{\mc{O}}$.
    \item[(e)] (real nonparity) If $L_x^s \cap \mc{O} = \{\text{two points}\}$ and $\mc{L}$ does not extend to $\widehat{\mc{O}}$, then 
    \[
    \mc{L} T_s = -\mc{L}.
    \]
\end{itemize}
\end{definition}

\begin{remark}
A word on the labels in Definition~\ref{def: Ts}: To an element $(\mc{O}, \mc{L}) \in \mathscr{D}$, one can associate a $\theta$-stable maximal torus $T$ of $G$ \cite[Cor.\ 2.2]{Vogan3}, which gives an action of $\theta$ on roots. The roots of $(G,B)$ can then be partitioned into ``root types'' based on their behavior under this action \cite[Def.\ 2.5]{Vogan3}. These root types are the origin of the labels in Definition~\ref{def: Ts}. These labels play no role in our arguments, but we include them to orient any readers familiar with this perspective.
\end{remark}

\begin{lemma}
\label{lem: Ts give a Hecke algebra action} \cite[Cor.\ 3.5]{LV} For any $s \in S$, the endomorphism $T_s$ of Definition~\ref{def: Ts} satisfies 
\[
T_s^2 = (q-1)T_s + q.
\]
Hence the operators $\{T_s, s \in S\}$ generate a right action of the Hecke algebra $\bm{H} = \bm{H}(W,S)$ (in the formulation in Remark~\ref{rmk: T_s formulation of Hecke algebra}) on $M_{LV}$.
\end{lemma}

Definition~\ref{def: Ts} and Lemma~\ref{lem: Ts give a Hecke algebra action} define the Lusztig--Vogan module $M_{LV}$. In \cite{LV}, several important properties of $M_{LV}$ are established. Most notably, the existence and uniqueness of a self-dual (with respect to an appropriately defined duality endomorphism on $M_{LV}$) basis for $M_{LV}$ is proven. The change-of-basis matrix between this self-dual basis and the basis $\mathscr{D}$ in (\ref{eq: LV module as a vector space}) is given by a collection of polynomials which we will refer to as Kazhdan--Lusztig--Vogan (KLV) polynomials. From these polynomials one can extract character formulas for $G_\R$. We encourage the uninitiated reader to explore \cite{LV} and \cite{Vogan3} in more detail to learn more about this beautiful story. 

\begin{example}
\label{example: LV module for SL(2,R)} $($The Lusztig--Vogan module of $SL_2(\R))$ Let $G=SL_2(\C)$ and $\theta: G \rightarrow G$ the involution given by 
\[
\theta g = \bp 1 & 0 \\ 0 & -1 \ep g \bp 1 & 0 \\ 0 & -1 \ep.
\]
Then $K=G^\theta$ is the torus of diagonal elements in $G$, 
\[
K= \left\{ \bp a & 0 \\ 0 & a^{-1} \ep \mid a \in \C^\times \right\} \simeq \C^\times. 
\]
The flag variety $X$ is isomorphic to $\mathbb{P}^1$, which we identify with the set of lines through the origin in $\C^2$. Denote the line through $(x, y) \in \C^2$ by $[x,y]$. The group $K$ acts on $X$ with three orbits:
\[
\mc{Q}_0 := \{[1,0]\}, \hspace{2mm} \mc{Q}_\infty := \{ [0,1]\}, \text{ and }\hspace{1mm} \mc{O} := \{ [x,y] \mid x \neq 0, y \neq 0 \} \simeq \C^\times.
\]
The stabilizers are $K$, $K$, and $\{ \pm I\}$, respectively. As irreducible $K$-equivariant local systems correspond to irreducible representations of the component group of the stabilizer, we have four elements of $\mathscr{D}$:
\[
\mathscr{D} = \{(\mc{Q}_0, \Q_{\mc{Q}_0}), (\mc{Q}_\infty, \Q_{\mc{Q}_\infty}), (\mc{O}, \Q_\mc{O}), (\mc{O}, \mc{L}) \},
\]
where $\mc{L}$ is the ``M\"obius band'' local system corresponding to the sign representation of $\{\pm I\}$. 

Let $s \in S$ be the non-trivial simple reflection in $W=S_2$. We have $\pi_s: X \rightarrow \mathrm{pt}$, and $L_x^s = X$ for any $x \in X$. An examination of Definition~\ref{def: Ts} shows that the reflection $s$ is type I noncompact imaginary (d1) with respect to the parameters $(\mc{Q}_0, \Q_{\mc{Q}_0})$ and $(\mc{Q}_\infty, \Q_{\mc{Q}_\infty})$, type I real (d2) with respect to the parameter $(\mc{O}, \Q_{\mc{O}})$, and non-parity real (e) with respect to the parameter $(\mc{O}, \mc{L})$. Hence the action of the element $T_s\in \bm{H}$ on the basis of local systems in $M_{LV}$ is given by the following formulas: 
\begin{align}
\label{eq: T_s formulas for SL2}
    & \Q_{\mc{Q}_0} \cdot T_s= \Q_{\mc{Q}_\infty} + \Q_\mc{O}, \hspace{7mm}   \Q_\mc{O} \cdot T_s = (q-2) \Q_\mc{O} + (q-1)(\Q_{\mc{Q}_0} + \Q_{\mc{Q}_\infty}),\\
    \nonumber
    & \Q_{\mc{Q}_\infty} \cdot T_s = \Q_{\mc{Q}_0} + \Q_{\mc{O}}, \hspace{7mm}   \mc{L} \cdot T_s = -\mc{L}.
\end{align}
After extending scalars to $\Z[v^{\pm 1}] = \Z[q^{\pm 1/2}]$ (Remark~\ref{rmk: T_s formulation of Hecke algebra}), we can also describe the action of the Kazhdan--Lusztig basis element $b_s = vT_s + v$ on the basis $\mathscr{D}$:
\begin{align}
\label{eq: b_s formulas for SL2}
     \Q_{\mc{Q}_0} \cdot b_s  = v(\Q_{\mc{Q}_0} + \Q_{\mc{Q}_\infty} + \Q_\mc{O}), \hspace{3mm} &  \Q_\mc{O} \cdot b_s= (v^{-1} - v)(\Q_{\mc{Q}_0} + \Q_{\mc{Q}_\infty} + \Q_\mc{O}), \\
    \nonumber
     \Q_{\mc{Q}_\infty} \cdot b_s  = v(\Q_{\mc{Q}_0} + \Q_{\mc{Q}_\infty} + \Q_\mc{O}),\hspace{2mm}  &  \mc{L} \cdot b_s= 0.
\end{align}
\end{example}



\subsection{Soergel bimodules}
\label{sec: Soergel bimodules}
Associated to the Coxeter system $(W,S)$ is a category of graded bimodules called Soergel bimodules. The category of Soergel bimodules categorifies the Hecke algebra of $(W,S)$. In this section, we describe the construction of this category and establish our notation. For a detailed description of the category of Soergel bimodules from both an algebraic and diagrammatic point of view, see \cite{SBim}. 

Fix a maximal torus $T$ of $G$ contained in $B$ and let
\begin{equation}
    \label{eq: R}
R:=S(X(T)_\Q)
\end{equation}
be the symmetric algebra on $X(T)_\Q:=X(T) \otimes_\Z \Q$, graded so that $X(T)_\Q$ has degree $2$. 

We will work in the category $R$-gbim of finitely generated (as both left and right $R$-modules) graded $R$-bimodules, with morphisms given by graded $R$-bimodule maps which are homogeneous of degree $0$. The category $R$-gbim is the ambient home of Soergel bimodules. There is a natural shift functor $(n): R\mathrm{-gbim} \rightarrow R\mathrm{-gbim}$ for each $n\in \Z$ sending $M \mapsto M(n)$, where $M(n)^i:=M^{n+i}$. Given an object $M$ in $R$-gbim and a Laurent polynomial $p=\sum p_iv^i \in \Z_{\geq 0}[v^{\pm 1}]$, denote by
\begin{equation}
    \label{eq: polynomial powers}
    M^{\oplus p} := \bigoplus M(i)^{\oplus p_i}.
\end{equation}
The tensor product $ - \otimes_R -$ gives $R$-gbim the structure of a monoidal category. 

The natural action of $W$ on $X(T)_\Q$ extends to an action\footnote{See Notation \ref{not: W action}.} of $W$ on $R$. For an element $s \in S$, denote by 
\[
R^s = \{r \in R \mid s  r =r\}.
\]
For $s \in S$, we define an object $B_s$ in $R$-gbim by 
\begin{equation}
    \label{eq: Bs}
    B_s:=R \otimes_{R^s}R(1).
\end{equation}
The modules $B_s$ for $s \in S$ are the building blocks of Bott--Samelson bimodules. 
\begin{definition}
\label{def: Bott-Samelson bimodule} Let $\underline{w}=(s_1, \ldots, s_n)$ be an expression for $w \in W$. The {\em Bott--Samelson bimodule} corresponding to  $\underline{w}$ is the graded $R$-bimodule
\[
BS(\underline{w}):=B_{s_1} \otimes_R B_{s_2} \otimes_R \cdots \otimes_R B_{s_n}. 
\]
\end{definition}
It is easy to see that there is a canonical isomorphism 
\[
BS(\underline{w})=R \otimes_{R^{s_1}} R \otimes_{R^{s_2}} \cdots \otimes_{R^{s_n}} R (\ell(\underline{w})). 
\]
With this we can define Soergel bimodules. 
\begin{definition}
\label{def: SBim}
Let $\SBim$ be the additive Karoubian subcategory of $R$-gbim generated by Bott-Samelson bimodules and their shifts. A {\em Soergel bimodule} is an object in $\SBim$. In other words, a Soergel bimodule is a graded $R$-bimodule which is isomorphic to a direct summand of a finite direct sum of grading shifts of Bott-Samelson bimodules.  
\end{definition}
The category $\SBim$ of Soergel bimodules is closed under tensor products, so it is a monoidal category. It is additive, but not abelian. 

\begin{example}
\label{example: SBim for type A1} $($Soergel bimodules of type $A_1)$ Let $W=S_2$, generated by the simple reflection $s$. The bimodules $R$ and $B_s$ are indecomposible Soergel bimodules, and the square of $B_s$ is 
\begin{align}
    \label{eq: square of Bs}
    B_s\otimes_R B_s &=R\otimes_{R^s}R \otimes_{R^s}R(2) \\
    \nonumber &\simeq R \otimes_{R^s} (R^s \oplus R^s(-2)) \otimes_{R^s} R(2) \\
    \nonumber &= B_s(1) \oplus B_s(-1).
\end{align}
Hence up to shift and isomorphism, $R$ and $B_s$ are the only indecomposible objects in $\SBim$. 
\end{example}

Denote by $[\SBim]_\oplus$ the split Grothendieck group of $\SBim$. That is, $[\SBim]_\oplus$ is the abelian group generated by symbols $[M]$ for all objects $M \in \SBim$ subject to the relations $[M]=[M]'+[M'']$ whenever $M \simeq M' \oplus M''$ in $\SBim$. We can give $[\SBim]_\oplus$ the structure of a $\Z[v^{\pm 1}]$-module by 
\[
p[M]:=[M^{\oplus p}]
\]
for $M \in \SBim$ and $p \in \Z_{\geq 0}[v^{\pm 1}]$, with notation as in (\ref{eq: polynomial powers}). Because $\SBim$ is monoidal, $[\SBim]_\oplus$ has the structure of a $\Z[v^{\pm 1}]$-algebra. 

Soergel showed that $\SBim$ categorifies the Hecke algebra. 
\begin{theorem}
\label{thm: Soergel's categorification theorem}
{\em (Soergel's categorification theorem)} The map $b_s \mapsto [B_s]$ defines an isomorphism of $\Z[v^{\pm 1}]$-algebras 
\[
\bm{H} \xrightarrow{\sim} [\SBim]_\oplus. 
\]
\end{theorem}


\subsection{The equivariant derived category}
\label{sec: the equivariant derived category}
The geometric setting of this paper is the equivariant derived category, as developed in \cite{BL}. Here we briefly recall some of its basic structure. 

Let $H$ be a complex algebraic group and $X$ a $H$-variety (over $\C$, equipped with the classical topology). Denote by $D_c^b(X):=D^b_c(X; \Q)$ the bounded derived category of constructible (with respect to the $H$-orbit stratification) sheaves of $\Q$-vector spaces on $X$. Fix a smooth model $EH \rightarrow BH$ for the classifying space of $H$, as in \cite[\S 12.4.1]{BL}. Let 
\[
X \xleftarrow{p} X \times EH \xrightarrow{q} X \times_H EH
\]
be the natural projections. We denote elements in quotient spaces using brackets; e.g., we write $[x, e] \in X \times_H EH$ for $x \in X$ and $e \in EH$. Denote by $D^b(X \times EH)$ and $D^b(X \times_H EH)$ the bounded derived categories of $\Q$-vector spaces on $X \times EH$ and $X \times_H EH$, respectively.  The (bounded, constructible) {\em equivariant derived category $D^b_H(X) := D^b_H(X; \Q)$} is defined as follows. 
\begin{definition}
\label{def: equivariant derived category} An object $\mc{G} \in D^b_H(X)$ is a triple $\mc{G} = (\mc{G}_X, \overline{\mc{G}}, \beta_{\mc{G}})$ with $\mc{G}_X \in D^b_c(X)$, $\overline{\mc{G}} \in D^b(X \times_H EH)$, and $\beta_{\mc{G}}:p^*\mc{G}_X \xrightarrow{\sim} q^*\overline{\mc{G}}$ an isomorphism in $D^b(X \times_H EH)$. A morphism $\varphi:\mc{G} \rightarrow \mc{F}$ in $D_H^b(X)$ is a pair $\varphi = (\varphi_X, \overline{\varphi})$ with $\varphi_X: \mc{G}_X \rightarrow \mc{F}_X$ a morphism in $D^b_c(X)$ and $\overline{\varphi}: \overline{\mc{G}} \rightarrow \overline{\mc{F}}$ a morphism in $D^b(X \times_H EH)$ such that the diagram 
\[
\begin{tikzcd}
p^* \mc{G}_X \arrow[r, "p^* \varphi_X"] \arrow[d, "\beta_\mc{G}"] & p^* \mc{F}_X \arrow[d, "\beta_\mc{F}"] \\
q^* \overline{\mc{G}} \arrow[r, "q^* \overline{\varphi}"] & q^* \overline{\mc{F}}
\end{tikzcd}
\]
commutes. 
\end{definition}

There is a forgetful functor
\begin{equation}
    \label{eq: forgetful functor}
    \mathrm{For}: D^b_H(X) \rightarrow D^b(X)
\end{equation}
sending $\mc{G}=(\mc{G}_X, \overline{\mc{G}}, \beta_\mc{G})$ to $\For\mc{G}:=\mc{G}_X$. Let $\Q_X \in D^b_c(X)$ be the constant sheaf, viewed as a complex concentrated in degree zero. There is a canonical  equivariant lift $(\Q_X, \Q_{X \times_H EH}, id) \in D^b_H(X)$ of $\Q_X$ which we will also denote by $\Q_X$. Whether we are referring to the object of $D^b_H(X)$ or $D^b_c(X)$ should always be clear from context. 

The equivariant derived category $D^b_H(X)$ is triangulated and supports the usual six functor formalism in a way that is compatible with the forgetful functor. It also supports a natural shift functor 
\[
[n]: \mc{G} \mapsto \mc{G}[n]:= (\mc{G}_X[n], \overline{\mc{G}}[n], \beta_\mc{G}[n])
\]
for each $n \in \Z$. Given $\mc{G}, \mc{F} \in D^b_H(X)$, we denote by 
\begin{equation}
    \label{eq: Hom in equivariant derived category}
    \Hom_H(\mc{G}, \mc{F}):= \Hom_{D^b_H(X)}(\mc{G}, \mc{F}),
\end{equation}
and 
\begin{equation}
    \label{eq: Hom dot}
    \Hom^\bullet_H (\mc{G}, \mc{F}):= \bigoplus_{n \in \Z} \Hom_{D^b_H(X)}(\mc{G}, \mc{F}[n]). 
\end{equation}

There are two cohomology functors on $D^b_H(X)$: the {\em equivariant hypercohomology functor}
\begin{equation}
    \label{eq: equivariant hypercohomology}
    \mathbb{H}^\bullet_H:= \Hom^\bullet_H(\Q_X, -): D^b_H(X) \rightarrow H^*_H(X)\mathrm{-gmod}
\end{equation}
and the {\em ordinary hypercohomology functor}
\begin{equation}
    \label{eq: ordinary hypercohomology}
    \mathbb{H}^\bullet:= \Hom^\bullet_{D^b_c(X)}(\Q_X, \For(-)):D^b_H(X) \rightarrow H^*(X)\mathrm{-gmod}.
\end{equation}
Here $H^*_H(X)-$gmod (resp.\ $H^*(X)-$gmod) is the category of graded modules over the graded ring $H^*_H(X):=H^*(X \times_H EH)$ (resp.\ $H^*(X)$). The graded vector space $\Hom_H^\bullet(\Q_X, \mc{G})$ (resp.\ $\Hom^\bullet_{D^b_c(X)}(\Q_X, \mc{G})$) gains a $H^*_H = \Hom^\bullet_H(\Q_X, \Q_X)$-action (resp.\ $H^*(X) = \Hom^\bullet_{D^b_c(X)}(\Q_X, \Q_X)$-action) via precomposition. 

The hypercohomology functors agree with the usual sheaf cohomology functors: for every $\mc{G} \in D^b_H(X)$, 
\begin{equation}
    \label{eq: hypercohomology agrees with sheaf cohomology}
    \mathbb{H}_H^\bullet(\mc{G}) = H^*(X \times_H EH; \overline{\mc{G}}), \text{ and } \mathbb{H}^\bullet(\mc{G}) = H^*(X; \mc{G}_X).
\end{equation}
Equivariant hypercohomology can be computed using the Leray spectral sequence associated to the fibration $X\times_H EH \rightarrow BH$, which simplifies to 
\begin{equation}
    \label{eq: spectral sequence for equivariant hypercohomology}
    E^{pq}_2=H^p(BH) \otimes \mathbb{H}^q(\mc{G}) \implies \mathbb{H}^{p+q}_H(\mc{G}). 
\end{equation}


\subsection{The geometric Hecke category}
\label{sec: The geometric Hecke category}
 The definitions of Section \ref{sec: Soergel bimodules} have their roots in geometry. There is another categorification of the Hecke algebra which can be found within the equivariant derived category $D^b_B(G/B)$, coming from Grothendieck's function-sheaf dictionary. This geometric incarnation of the Hecke category will play a key role in our arguments that follow, so we recall its construction. For more details (including a more general framework which makes the function-sheaf perspective explicit), see \cite[\S2]{parity}.

Return to the setup of Section \ref{sec: introduction}, and let $X:=G/B$ be the flag variety of $G$. The action of $B$ on $X$ stratifies $X$ into Schubert cells:
\[
X= \bigsqcup_{w \in W} X_w \text{ where }X_w:=B wB/B \simeq \mathbb{A}^{\ell(w)}.
\]

Let $D^b_B(X):=D^b_B(X; \Q)$ be the $B$-equivariant derived category, as in Section \ref{sec: the equivariant derived category}. The category $D^b_B(X)$ has a monoidal structure given by convolution, which is constructed as follows. Denote by $\mathrm{mult}:G \times_B X \rightarrow X$ the map induced by multiplication in $G$. The convolution of $\mc{F}, \mc{G} \in D^b_B(X)$ is defined to be 
\[
\mc{F} * \mc{G} := \mathrm{mult}_*(\mc{F} \boxtimes_B \mc{G}),
\]
where $\mc{F} \boxtimes_B \mc{G}$ is obtained by descent from $\mc{F} \boxtimes \mc{G} \in D^b_{B^3}(G \times X)$. 

For $s \in S$, let 
\[
P_s:=\overline{BsB}=BsB \sqcup B 
\]
be the corresponding parabolic subgroup of $G$. Denote by $\Q_{P_s/B}$ the constant sheaf on $P_s/B$, considered as a sheaf on $X$ by extending by zero. 

\begin{definition}
\label{def: geometric Hecke category}
The {\em geometric Hecke category} $\mc{H}$ is the full subcategory of $D^b_B(X)$ generated by $\{\Q_{P_s/B}\}_{s \in S}$ under convolution ($*$), direct sums ($\oplus$), homological shifts ($[1]$), and direct summands ($\ominus$):
\[
\mc{H}=\langle \Q_{P_s/B} \mid s \in S \rangle _{*, \oplus, [1], \ominus}.
\]
\end{definition}
By construction, $\mc{H}$ is closed under convolution, so it is a monoidal category. It is additive, but not abelian. 

Let $[\mc{H}]_\oplus$ be the split Grothendieck group\footnote{See end of Section \ref{sec: Soergel bimodules} for a definition of split Grothendieck group.} of $\mc{H}$. We can give $[\mc{H}]_\oplus$ the structure of a $\Z[v^{\pm 1}]$-module by setting
\[
v \cdot [\mc{F}]:=[\mc{F}[1]]. 
\]
The monoidal structure on $\mc{H}$ gives $[\mc{H}]_\oplus$ the structure of a $\Z[v^{\pm 1}]$-algebra. As one might hope from its name, the geometric Hecke category categorifies the Hecke algebra. 
\begin{theorem}
\label{thm: geometric categorification}
 The map $b_s \mapsto [\Q_{P_s/B}[1]]$ defines an isomorphism of $\Z[v^{\pm 1}]$-algebras
\[
\bm{H} \xrightarrow{\sim} [\mc{H}]_\oplus.
\]
\end{theorem}
For a proof, see, for example \cite{parity}. 


\subsection{From the geometric Hecke category to Soergel bimodules}
\label{sec: Relationship between the geometric Hecke category and Soergel bimodules}

By Theorem~\ref{thm: Soergel's categorification theorem} and Theorem~\ref{thm: geometric categorification}, both $\SBim$ and $\mc{H}$ categorify $\bm{H}$. In fact, their relationship can be made more explicit using an alternate characterization of the geometric Hecke category. 

The category $D^b_B(X)$ is Krull-Schmidt, meaning that every object decomposes into a finite direct sum of indecomposable objects, which are characterized by having local endomorphism rings. Hence, to understand the category $\mc{H}$, we should start by understanding the indecomposible objects. By definition, objects in $\mc{H}$ are direct summands of finite direct sums of shifts of objects of the form 
\[
\mc{E}_{\underline{w}}:=\Q_{P_{s_1}/B} * \Q_{P_{s_2}/B} * \cdots * \Q_{P_{s_n}/B}
\]
for some expression $\underline{w} = (s_1, s_2, \ldots, s_n)$ of an element $w \in W$. The objects $\mc{E}_{\underline{w}}$ can be described in terms of Bott-Samelson resolutions of the Schubert varieties $\overline{X}_w$. If $\underline{w}=(s_1, s_2, \ldots, s_n)$ is an expression for $w \in W$, then the Bott-Samelson space 
\[
BS_{\underline{w}} := P_{s_1} \times_B P_{s_2} \times_B \cdots \times_B P_{s_n}/B 
\]
is a smooth algebraic variety. There is a natural projective morphism 
\[
\xi: BS_{\underline{w}} \rightarrow \overline{X}_w
\]
induced by multiplication in $G$. For any expression $\underline{w}$, there is a canonical isomorphism \cite[Lem.\ 3.2.1]{Soergel2000}
\[
\xi_* \Q_{BS_{\underline{w}}}=\mc{E}_{\underline{w}}. 
\]

\begin{remark}
\label{rem: Bott-Samelson resolution}
If $\underline{w}$ is a reduced expression for $w \in W$, then $\xi$ is a resolution of singularities of $\overline{X}_w$. Such resolutions are sometimes called {\em Bott-Samelson resolutions}. 
\end{remark}

Let
\begin{equation}
    \label{eq: IC sheaf}
\bm{IC}_w:=\bm{IC}(\overline{X}_w) \in D^b_B(G/B)
\end{equation}
 denote the $B$-equivariant intersection cohomology complex on  $\overline{X}_w$, shifted by $\ell(w)$, extended by zero, so it is perverse on $X$. The morphism $\xi$ is proper, so by the decomposition theorem, $\mc{E}_{\underline{w}}$ decomposes into a finite direct sum of shifts of $\bm{IC}_w$ for $w \in W$. This gives us a second characterization of $\mc{H}$:
\begin{equation}
\label{eq: second incarnation Hecke cateory}
\mc{H} = \langle \bm{IC}_w \mid w \in W \rangle_{\oplus, [1]}.
\end{equation}

Using the characterization (\ref{eq: second incarnation Hecke cateory}), the relationship between $\mc{H}$ and $\SBim$ becomes clear by applying hypercohomology. Let
\begin{equation}
    \label{eq: B equivariant cohomology}
\HH_B^\bullet: D^b_B(X) \rightarrow H^*_B(X)-\mathrm{gmod}
\end{equation}
be the equivariant hypercohomology functor, as in (\ref{eq: equivariant hypercohomology}). There is a canonical isomorphism of graded algebras
\[
H^*_B(X) = H^*_T(X) = R \otimes_{R^W} R,
\]
where $R=S(\mf{t}^*)$ as in (\ref{eq: R}) (cf.\ \cite[Prop.\ 1(iii)]{Brion}). The quotient map $R \otimes R \rightarrow R \otimes_{R^W} R$ induces a fully faithful functor 
\begin{equation}
    \label{eq: from cohomology modules to bimodules}
R \otimes_{R^W} R \mathrm{-gmod} \rightarrow R\mathrm{-gbim}. 
\end{equation}
By identifying $H_B^*(X)$ with its image under (\ref{eq: from cohomology modules to bimodules}), we can rewrite (\ref{eq: B equivariant cohomology}) as
\[
\HH^\bullet_B: D^b_B(X) \rightarrow R\mathrm{-gbim}.
\]
Soergel observed that when restricted to the geometric Hecke category, this functor is fully faithful, and its essential image is the category of Soergel bimodules. Hence $B$-equivariant hypercohomology provides an equivalence
\begin{equation}
    \label{eq: hypercohomology is an equivalence}
    \HH^\bullet_B: \mc{H} \xrightarrow{\sim} \SBim. 
\end{equation}
Under the equivalence (\ref{eq: hypercohomology is an equivalence}), we can see the players of Section \ref{sec: Soergel bimodules} in their geometric guise:
\begin{align}
\label{eq: hypercohomology of ICs}
\HH^\bullet_B(\bm{IC}_s) &= B_s \text{ for all }s \in S, \text{ and }\\
\label{eq: hypercohomology of Bott-Samelsons}
\HH^\bullet_B(\mc{E}_{\underline{w}}) &= BS(\underline{w}) \text{ for all expressions } \underline{w} \text{ of all elements } w \in W.
\end{align}

\begin{remark}
Equation (\ref{eq: hypercohomology of Bott-Samelsons}) together with Remark~\ref{rem: Bott-Samelson resolution} explains the terminology ``Bott-Samelson bimodule'' in Definition~\ref{def: Bott-Samelson bimodule}. 
\end{remark}


\section{A geometric categorification}
\label{sec: a geometric categorification}

In this section, we use the equivariant derived category \(D_K^b(X)\) to give a geometric categorification of the Lusztig--Vogan module (q.v.\ Theorem~\ref{thm: semisimple categorification}), we study convolution of complexes of sheaves via  character maps to the Lusztig--Vogan module, and we compute cohomology of all fibers for any resolution constructed in \cite{Larson2020} (q.v.\ Theorem~\ref{theorem: fibers}).
On one hand, our geometric categorification can be viewed as studying equivariant proper maps via linearization, and on the other hand, it is a stepping stone to our algebraic categorification in Section~\ref{sec: An algebraic categorification}.

Recall that \(K\)-equivariant local systems on \(\cO\) are in bijection with representations of the component group \(K_x/K_x^0\) of the stabilizer \(K_x\) of \(x\in\cO\) in \(K\).
Irreducible representations correspond to irreducible local systems.
Given an irreducible \(K\)-equivariant local system \(\cL\in\Loc_K(\cO)\), denote by
\begin{equation}
\widetilde\cL\in D_K^b(X)
\end{equation}
the intermediate extension of \(\cL\) to \(\overline{\cO}\)  extended by zero to \(X\).
Set
\begin{equation}
\label{eq: K-equivariant IC sheaves}
\IC(\overline{\cO},\cL)=\widetilde\cL[\dim(\cO)]
\end{equation}
to be the corresponding self-dual object.

\subsection{Semisimple complexes}
We study a subcategory of the \(K\)-equivariant derived category.
The category consists of all semisimple complexes in \(D_K^b(X)\), which is the \(K\)-equivariant analogue of the geometric Hecke category (q.v.\ \eqref{eq: second incarnation Hecke cateory}).

\begin{definition}
\label{def: semisimple category}
Let 
\[
\mc{M}_{LV}^{ss}:= \langle \bm{IC}(\overline{\cO}, \mc{L}) \mid (\cO, \mc{L}) \in \mathscr{D} \rangle_{\oplus, [1]}
\]
be the subcategory of $D^b_K(X)$ generated by $\bm{IC}$ sheaves under the operations of taking direct sums and grading shifts. 
\end{definition}

By construction, the category $\mc{M}_{LV}^{ss}$ is an additive category. 
It also admits a natural categorical action (in the sense of \cite[Ch.\  7]{TensorCategoriesBook}) of the monoidal category $\mc{H}$ via convolution.
The convolution formalism described in Section \ref{sec: The geometric Hecke category} gives a right action of $D^b_B(X)$ on $D^b_K(X)$: for $\mc{F} \in D^b_K (X)$ and $\mc{G} \in D^b_B(X)$, define 
\[
\mc{F} * \mc{G}=\mathrm{mult}_*(\mc{F}\boxtimes_B\mc{G}),
\]
where $\mathrm{mult}:G \times_B X \rightarrow X$ is as in Section \ref{sec: The geometric Hecke category} and $\mc{F} \boxtimes_B \mc{G}$ is obtained via descent from $\mc{F} \boxtimes \mc{G} \in D^b_{K \times B \times B}(G \times X)$. By restricting this action to the geometric Hecke category $\mc{H}$, we obtain an action of $\mc{H}$ on $D^b_K(X)$. 

The convolution action of $\mc{H}$ on $D^b_K(X)$ can also be realized in terms of push-pull functors. 
For $s \in S$, define a functor\footnote{We write $\mc{F} \Theta_s$ for the image of $\mc{F}$ under the functor $\Theta_s$ to remind ourselves that we should think of the functors $\Theta_s$, $s \in S$ as providing a right action of $\mc{H}$ on $D^b_K(X)$.}   $\Theta_s: D^b_{K}(X) \rightarrow D^b_K(X)$ by
\begin{equation}
    \label{eq: Theta}
\mc{F} \Theta_s := \pi_s^* \pi_{s*} \mc{F} [1]
\end{equation}
for $\mc{F} \in D^b_K(X)$.
As in (\ref{eq: IC sheaf}), denote by $\bm{IC}_s:=\bm{IC}(P_s/B) \in \mc{H}$. 
\begin{lemma}
\label{lem: push pull is convolution with ICs}
For any object $\mc{F} \in D_K^b(X)$ and $s \in S$, 
\[
\mc{F} \Theta_s \simeq \mc{F} * \bm{IC}_s.
\]
\end{lemma}
\begin{proof}
See \cite{Soergel90, LV}. 
\end{proof}

Because the map $\pi_s$ is proper for any $s \in S$, the decomposition theorem stipulates that for $(\mc{O}, \mc{L}) \in \mathscr{D}$, the complex $\bm{IC}(\overline{\mc{O}}, \mc{L})\Theta_s $ decomposes into a direct sum of shifts of $\bm{IC}$ sheaves. In particular, 
\[
\bm{IC}(\overline{\mc{O}}, \mc{L})\Theta_s \in \mc{M}_{LV}^{ss} \text{ for any } (\mc{O},\mc{L}) \in \mathscr{D} \text{ and } s \in S. 
\]
This shows that the category $\mc{M}_{LV}^{ss}\subseteq D^b_K(X)$ is stable under the action of $\mc{H}$.

For technical reasons we extend scalars of \(M_{LV}\) to \(\Z[q^{\pm 1/2}]\).
This will enable odd shifts of \(\widetilde\cL\) in our category; we do this to include all semisimple complexes, and in particular, self-dual complexes.
Given a \(K\)-orbit \(\cO\), let \(k_\cO\,\colon \cO\to X\) be the locally closed inclusion.
By construction, the cohomology sheaves \(\cH^n(k_\cO^\ast\cF)\) are objects in the abelian category of \(K\)-equivariant constructible sheaves on \(\cO\).
For an irreducible \(K\)-equivariant local system \(\cL\) on \(\cO\), denote by
\begin{equation}
[\cL\colon\cH^n(k_\cO^\ast\cF)]
\end{equation}
the multiplicity of \(\cL\) in \(\cH^n(k_\cO^\ast\cF)\).
Define \(\ch\,\colon D_K^b(X)\to M_{LV}\) by
\begin{equation}\label{equation: ch}
\cF\mapsto \sum_{n,(\cO,\cL)}[\cL:\cH^n(k_{\cO}^\ast\cF)]\,q^{n/2}\,\cL.
\end{equation}
The character map \(\ch\) does not intertwine Hecke actions on local systems, but we show in Section~\ref{sec: proof of geometric categorification}  that \(\ch\) does intertwine the Hecke actions when restricted to our semisimple category.

\begin{theorem}
\label{thm: semisimple categorification}
Let $\mc{M}_{LV}^{ss}$ be as in Definition~\ref{def: semisimple category} and $M_{LV}$ the Lusztig--Vogan module (Equation (\ref{eq: LV module as a vector space}), Definition~\ref{def: Ts}, Lemma~\ref{lem: Ts give a Hecke algebra action}). As right $\bm{H}$-modules, 
\[
[\mc{M}_{LV}^{ss}]_\oplus \simeq M_{LV}.
\]
\end{theorem}

We prove Theorem~\ref{thm: semisimple categorification} in Section~\ref{sec: proof of geometric categorification} by relating \(\ch\) to an intermediary character map and reducing computations to root types.
The formula for each root type will be seen to agree with the formulas in Definition~\ref{def: Ts}.
We begin by analyzing local systems.

\subsection{Local systems}
We start by introducing an auxiliary character map.
Define \(m\,\colon D_K^b(X)\to M_{LV}\) by
\begin{equation}\label{equation: m}
\cF\mapsto\sum_{n,(\cO,\cL)}(-1)^n[\cL: \cH^n(k^\ast_\cO\cF)]\,q^{\lfloor n/2\rfloor}\,\cL,
\end{equation}
where \(\lfloor n\rfloor\) is the floor function.
We include some easy facts about \eqref{equation: m}.
\begin{itemize}
\item For every \(\cF\in D_K^b(G/B)^{\ast-even}\) (q.v.\ Definition~\ref{def: parity}) we have
\begin{equation}\label{equation: ch is m}
    m(\cF)=\ch(\cF).
\end{equation}
\item For every \(\cF\in D_K^b(X)\) we have
\begin{equation}\label{equation: shift}
m(\cF[-2])=qm(\cF).
\end{equation}
\item If \(\cH^\ast(\cF)\in D_K^b(X)\) denotes the direct sum of the cohomology sheaves of \(\cF\), then we have
\begin{equation}\label{equation: cohomology sheaves}
m(\cF)=m(\cH^\ast(\cF)).
\end{equation}
\item If \(\cE\to\cF\to\cG\overset{+1}{\longrightarrow}\) is a distinguished triangle in \(D_K^b(X)\) such that \(\cG\in D_K^b(X)^{\ast-even}\), then
\begin{equation}\label{equation: additivity}
m(\cF)=m(\cE)+m(\cG).
\end{equation}
\end{itemize}

\begin{lemma}\label{lemma: local systems}
For every \((\cO,\cL)\in\cD\) and \(s\in S\), we have
\begin{equation}\label{equation: push-pull}
m(\pi_s^\ast\pi_{s\ast}k_{\cO!}\cL)=\cL(T_s+1).
\end{equation}
\end{lemma}

The proof is a case-by-case analysis of root types provided below.
We require lemmata before proceeding.

For every \(s\in S\) and \(x\in X\) define
\begin{equation}\label{equation: mu and pi}
\begin{tikzcd}
KxB\times_BBsB/B\arrow[r, "\mu"]\arrow[d, "\pi"]&KxBsB/B\\
KxB/B&
\end{tikzcd}
\end{equation}
to be the natural projection and multiplication maps. 
For every \(y\in KxBsB/B\) we have
\begin{equation}\label{equation: fibers of mu}
\mu^{-1}(y)=KxB/B\cap yBsB/B.
\end{equation}

\begin{lemma}
For every \((\cO,\cL)\in\cD\) and \(s\in S\), the relation \eqref{equation: push-pull} holds true if and only if
\begin{equation}\label{equation: equivalent push-pull}
m(k_{!}\mu_!\pi^\ast\cL)=\cL T_s,
\end{equation}
where \(k\,\colon KxBsB/B\to X\) is the locally closed inclusion, and \(\mu\) and \(\pi\) are as in \eqref{equation: mu and pi}.
\begin{proof}
The diagram
\begin{equation}
\begin{tikzcd}
KxB\times_BP_s/B\arrow[r, "\bar\mu"]\arrow[d, "\bar\pi"] &KxP_s/B\arrow[d, "\pi_s'"]\arrow[r, "k''"]&G/B\arrow[d, "\pi_s"]\\
KxB/B\arrow[r, "\widetilde\pi_s"]&KxP_s/P_s\arrow[r, "\ell"]&G/P_s
\end{tikzcd}
\end{equation}
is cartesian so by proper base change
\begin{equation}
\begin{split}
\pi_s^\ast\pi_{s\ast}k_{\cO!}&=\pi_s^\ast(\pi_sk_\cO)_!\\
&=\pi_s^\ast(\ell\widetilde\pi_s)_!\\
&=(k''\bar\mu)_!\bar\pi^\ast
\end{split}
\end{equation}
holds true.
Let 
\begin{equation}
\begin{split}
i\,\colon KxB/B&\to KxB\times_BP_s/B\\ j\,\colon KxB\times_BBsB/B&\to KxB\times_BP_s/B
\end{split}
\end{equation}
be complementary closed and open inclusions.
The distinguished triangle
\begin{equation}
j_!j^\ast\bar\pi^{\ast}\cL\to\bar\pi^{\ast}\cL\to i_\ast i^\ast\bar\pi^{\ast}\cL\overset{+1}{\longrightarrow}
\end{equation}
gives the distinguished triangle
\begin{equation}\label{equation: dt}
(k\mu)_!\pi^\ast \cL\to \pi_s^\ast(\pi_sk_\cO)_!\cL   \to k_{\cO!}\cL\overset{+1}{\longrightarrow}
\end{equation}
under the triangulated functor \((k''\bar\mu)_!\).
The character map \(m\) \eqref{equation: m} is additive on \eqref{equation: dt} by \eqref{equation: additivity}.
Because \(m(k_{\cO!}\cL)=\cL\), we have
\begin{equation}\label{equation: this guy}
m(\pi_s^\ast(\pi_sk_{\cO})_!\cL)=m((k\mu)_!\pi^\ast\cL)+\cL.
\end{equation}
The result follows.
\end{proof}

\end{lemma}

\begin{lemma}\label{lemma: Leray-Hirsch pullback}
Let \(\zeta\,\colon Z\to Y\) be a locally trivial fiber bundle with fiber \(F=F_y=\zeta^{-1}(y)\).
Assume that for every \(n\in\Z\) there are finitely many elements \(f_{n\ell}\in H^n(Z)\) such that for every \(y\in Y\) we obtain a \(\Q\)-basis for \(H^n(F_y)\) by restriction.
Let \(\zeta'\,\colon Z'\to Y'\) be the pullback of \(\zeta\) along a map \(Y'\to Y\).
Then for every \(p\in\Z\), every \(z'\in H^{p}(Z')\) can be written uniquely as
\begin{equation}
z'=\sum_{n\ell}y_{n\ell}'f_{n\ell},
\end{equation}
for some \(y_{n\ell}'\in H^{p-n}(Y')\).
\begin{proof}
See, e.g., \cite[App.\ A \S4]{AndersonFulton2021} for the case of identity map \(Y'\to Y\) (and \(Z'=Z\)).
Let \(y'\mapsto y\) so we have the canonical homeomorphism \(F_{y'}\to F_y\) as usual.  
The commutative diagram
\begin{equation}
\begin{tikzcd}
F_{y'}\arrow[r,"\isom"]\arrow[d]&F_y\arrow[d]\\
Z'\arrow[r]&Z
\end{tikzcd}
\end{equation}
shows that pulling back \(f_{ij}\) to \(f_{ij}'\in H^n(Z')\) restricts to a basis in \(H^n(F_{y'})\).
Applying the case of \(Z'=Z\) to \(Z'\) proves the lemma.
\end{proof}
\end{lemma}

\begin{lemma}\label{lemma: derived pushforward}
Let \(\zeta\,\colon Z\to Y\) be a smooth map with fiber \(F=F_y=\zeta^{-1}(y)\).
Assume that for every \(n\in\Z\) there are finitely many elements \(f_{n\ell}\in H^n(Z)\) such that for every \(y\in Y\) we obtain a \(\Q\)-basis for \(H^n(F_y)\) by restriction.
For every \(\cF\in D^b(Y)\) we have
\begin{equation}\label{equation: smooth pushforward}
\zeta_!\zeta^\ast\cF\cong \cF\boxtimes H_c^\ast(F).
\end{equation}

\begin{proof}
We have
\begin{equation}
\zeta_!\zeta^\ast\cF\isom\cF\underset{\Q_Y}{\otimes}\zeta_!\Q_Z
\end{equation}
by \cite[Prop.\ 2.6.6]{KashiwaraSchapira1994}, so it is enough to prove \eqref{equation: smooth pushforward} for the constant sheaf.
Identify \(\zeta=\zeta\times\id_{\pt}\).
Then
\begin{equation}
\cHom(\zeta_!\zeta^\ast\Q_Y,\Q_Y\boxtimes H_c^\ast(F))\isom \zeta_\ast \cHom(\zeta^\ast\Q_Y,\zeta^!\Q_Y\boxtimes H_c^\ast(F))
\end{equation}
by Verdier duality.
For smooth morphisms the relation
\begin{equation}
\zeta^!\Q_Y\isom\Q_Z[2d]
\end{equation}
is a relative Poincar\'e duality.
It follows that
\begin{equation}\label{equation: sheaf hom}
\cHom(\zeta_!\zeta^\ast\Q_Y,\Q_Y\boxtimes H_c^\ast(F))\isom \zeta_\ast \cHom(\Q_Z,\Q_Z\boxtimes H_c^\ast(F)[2d]).
\end{equation}
We prove the lemma by constructing a global section of \eqref{equation: sheaf hom} restricting to isomorphisms on trivializable open subsets.

Let \(U\subseteq Y\) be connected such that \(\zeta\vert_{\zeta^{-1}(U)}\isom \pr_1\) is projection from \(U\times Y\).
We have
\begin{equation}\label{equation: trivial sheaf hom}
\cHom(\Q_U\boxtimes H_c^\ast(F),\Q_U\boxtimes H_c^\ast(F))\isom\zeta_\ast\cHom(\Q_{U\times F},\Q_{U\times F}\boxtimes H^\ast_c(F)[2d])
\end{equation}
since \(\zeta_!\Q_{U\times Y}\isom \Q_U\boxtimes H_c^\ast(F)\) by the K\"unneth formula.
Global sections in the RHS of \eqref{equation: trivial sheaf hom} are given by 
\begin{equation}\label{equation: generic element}
\begin{split}
Hom_{D(U\times F)}(\Q_{U\times F},\bigoplus_{n\in N}\Q_{U\times F}[2d-n])&=\bigoplus_{n\in N}H^{2d-n}(U\times F)\\
&\supseteq\bigoplus_{n\in N}H^0(U)\otimes H^{2d-n}(F)
\end{split}
\end{equation}
where \(N\) is the graded dimension of \(H_c^\ast(F)\).
The identity element in the LHS of \eqref{equation: trivial sheaf hom} corresponds to a generic element in RHS of \eqref{equation: generic element}, one coordinate for each element of \(N\).

The RHS of \eqref{equation: sheaf hom} has global sections equal to
\begin{equation}
\Hom_{D(Z)}(\Q_Z,\bigoplus_{n\in N}\Q_Z[2d-n])=\bigoplus_{n\in N}H^{2d-n}(Z).
\end{equation}
For every \(n\in N\) we have by assumption an element \(f_n\in H^{2d-n}(Z)\) restricting to a basis of \(H^{2d-n}(F)\), where we have \(H^{2d-n}(F)=H^n_c(F)\) by Poincar\'e duality; i.e., \((f_n)_{n\in N}\) restricts to a basis of \(H^\ast(F)\isom H_c^\ast(F)\).
The element \((f_n)_{n\in N}\) corresponds on the LHS of \eqref{equation: sheaf hom} to our desired isomorphism in \(\Hom_{D(Y)}(\zeta_!\zeta^\ast\Q_Y,\Q_Y\boxtimes H_c^\ast(F))\) since it restricts to isomorphisms over trivializable open sets by Lemma~\ref{lemma: Leray-Hirsch pullback}.
\end{proof}
\end{lemma}

\begin{remark}
We apply Lemma~\ref{lemma: derived pushforward} directly to the equivariant derived category by considering the finite dimensional smooth approximations in our smooth model (as in, e.g., \cite[Lem.\ 2.1]{AndersonFulton2021}).
For example, the map
\begin{equation}
ET\times_TG/B\to ET\times_{T}G/P_s,\qquad [e,x]\mapsto [e,\pi_s(x)]
\end{equation}
satisfies the Leray-Hirsch theorem.
\end{remark}

We are now ready to begin our case-by-case analysis of root types.
We align our numeration and terminology with Definition~\ref{def: Ts}.

\subsubsection*{(a) Compact imaginary}

\begin{example}[$G_\R$ maximal compact subgroup of $G$]
Let $G$ be any connected complex reductive algebraic group and \(\theta\) be the identity involution.
Then \(X\) is a homogeneous space for \(K=G\), and the simple reflection $s$ is compact imaginary (a) with respect to the single Lusztig--Vogan parameter $(X, \Q_X)$.
We have
\begin{equation}
\pi_s^\ast\pi_{s\ast}\Q_X=\Q_X\oplus\Q_X[-2]
\end{equation}
by Lemma~\ref{lemma: derived pushforward}. It follows that
\begin{equation}
m(\pi_s^\ast\pi_{s\ast}\Q_X)={\Q_X}(T_s+1)
\end{equation}
by \eqref{equation: shift}.
\end{example}

In general, for a compact imaginary reflection \(s\) with respect to \((\cO,\cL)\), the maps in \eqref{equation: mu and pi} have the form
\begin{equation}\label{equation: mu}
\begin{tikzcd}
&K/K_{[x,s]}\arrow[r, "\mu"]\arrow[d, "\pi"]&K/K_x\\
&K/K_x
\end{tikzcd}
\end{equation}
for \(x\in\cO=KxB/B\).
In this case \(\mu\) is smooth with fiber the affine line by \eqref{equation: fibers of mu}, and \(\pi=\mu\) up to a \(K\)-equivariant automorphism of \(K/K_{[x,s]}=KxB\times_BBsB/B\).
In particular, \(\pi^\ast\cF\) and \(\mu^\ast\cF\) are isomorphic in \(D_K^b(K/K_{[x,s]})\).
We have
\begin{equation}
\mu_!\pi^\ast\cL=\cL[-2]
\end{equation}
by Lemma~\ref{lemma: derived pushforward}.
It follows that
\begin{equation}
\begin{split}
m(k_{\cO!}\mu_!\pi^\ast\cL)&=q\cL\\
&=\cL T_s
\end{split}
\end{equation}
by \eqref{equation: shift}, so \eqref{equation: push-pull} and \eqref{equation: equivalent push-pull} hold true.

\subsubsection*{(b) Complex}

\begin{example}\label{example: complex root}
$(G_\R = SL_2(\C))$ Set $G=SL_2(\C) \times SL_2(\C)$ and let $\theta$ be the involution that switches the factors; i.e., $\theta(g_1, g_2) = (g_2, g_1)$. Then $K\simeq SL_2(\C)$ is the diagonal subgroup in $SL_2(\C) \times SL_2(\C)$, and the flag variety is isomorphic to $\mathbb{P}^1 \times \mathbb{P}^1$. Let $B_K$ be the Borel subgroup of upper triangular matrices in $SL_2(\C)$. We can identify $D^b_{SL_2(\C)}(\mathbb{P}^1 \times \mathbb{P}^1) = D^b_{B_K}(\mathbb{P}^1)$ via induction and study $B_K$-equivariant local systems on $\mathbb{P}^1$. The $B_K$-orbits on $\mathbb{P}^1$ are
\[
\mc{Q}_0 := \{[1,0]\} \text{ and } \mc{O}:= \{[x, 1]\} \simeq \C.
\]
The simple reflection $s$ is complex ascent (b1) with respect to the Lustig--Vogan parameter $(\mc{Q}_0, \Q_{\mc{Q}_0})$ and complex decent (b2) with respect to the parameter $(\mc{O}, \Q_\mc{O})$.

Let $i_0: \mc{Q}_0 \rightarrow \mathbb{P}^1$ and $j: \mc{O} \rightarrow \mathbb{P}^1$ be the closed and open embeddings. Because $i_{0!} = i_{0*}$, we have
\begin{align}
\label{eq: complex ascent}
\pi_s^* \pi_{s*} i_{0!} \Q_{\mc{Q}_0} &= \pi_s^* \pi_{s*} i_{0*} \Q_{\mc{Q}_0}\\
\nonumber &= \pi_s^* \Q_\mathrm{pt} \\
\nonumber &= \Q_{\mathbb{P}^1}.
\end{align}
By applying Lemma~\ref{lemma: derived pushforward} to the smooth map $\pi_s \circ j$, we compute 
\begin{align}
\label{eq: complex descent}
\pi_s^* \pi_{s*} j_! \Q_\mc{O} &= \pi_s^*(\pi_s \circ j)_! \Q_\mc{O} \\
\nonumber &= \pi_s^*(\pi_s \circ j)_! (\pi_s \circ j)^*\Q_\mathrm{pt} \\
\nonumber &= \pi_s^*\Q_\mathrm{pt}[-2] \\
\nonumber &= \Q_{\mathbb{P}^1}[-2]. 
\end{align}
It follows from \eqref{equation: shift} that \eqref{equation: push-pull} holds true for the above complex reflections.

\end{example}

In general if \(s\) is a complex ascent with respect to \((\cO=KxB/B,\cL)\) we have
\begin{equation}
\begin{tikzcd}
KxB\times_BBsB/B\arrow[r, "\mu"]\arrow[d, "\pi"]&KyB/B\\
KxB/B
\end{tikzcd}
\end{equation}
where \(\mu\) is an isomorphism.
For \(\cL=k_\cO^\ast k_{\cO!}\cL\), we have
\begin{equation}\label{equation: complex ascent}
\mu_!\pi^\ast\cL=\widehat\cL_{\widehat \cO\smallsetminus \cO}
\end{equation}
where \(\widehat\cO=\bigcup_{x\in\cO}L_x^s\) (as in Definition~\ref{def: Ts}) and \(\widehat\cL\) is the extension of \(\cL\) to \(\widehat\cO\), so \eqref{equation: equivalent push-pull} holds true.

If \(s\) is a complex descent with respect to \((\cO,\cL)\), let \(y\) be such that \(\widehat\cO=KyB/B\cup KxB/B\) and denote by \(k_{\widehat\cO}\) the inclusion of \(\widehat\cO\) in \(G/B\).
Then \(s\) is a complex ascent with respect to \((\cO'=KyB/B,\cL')\), where \((\cO',\cL')\) is the corresponding element in \(\cD\) under the identification \(K_y/K_y^0=K_x/K_x^0\).
The cartesian diagram
\begin{equation}
\begin{tikzcd}
KyB\times_BP_s/B\arrow[r, "\bar\mu"]\arrow[d, "\bar\pi"]&\widehat\cO\arrow[d,"\pi_sk_{\widehat\cO}"]\\
KyB/B\arrow[r, "\pi_s'"]&KyP_s/P_s
\end{tikzcd}
\end{equation}
gives the unique extension of \(\cL\) to \(\widehat\cO\),  \(\widehat\cL:=\bar\mu_\ast\bar\pi^\ast\cL'\), since \(\bar\mu\) is an isomorphism.
By change of base, the \(K\)-equivariant local system \(\widehat\cL\) is the pull-back of a \(K\)-equivariant local system on \(KyP_s/P_s\), so
\begin{equation}
m(\pi_s^\ast\pi_{s\ast}(k_{\widehat\cO!}\widehat\cL))=(1+q)\widehat\cL
\end{equation}
by Lemma~\ref{lemma: derived pushforward} and \eqref{equation: shift}.
The distinguished triangle
\begin{equation}\label{equation: complex dt}
\pi_s^\ast\pi_{s\ast}k_{\cO!}\cL\to \pi_s^\ast\pi_{s\ast}k_{\widehat\cO!}\to \pi_s^\ast\pi_{s\ast}k_{\cO'!}\cL'\overset{+1}{\longrightarrow}
\end{equation}
is obtained by applying the triangulated functor \(\pi_s^\ast\pi_{s\ast}k_{\widehat\cO!}\) to the open-closed distinguished triangle of \(\widehat\cO=\cO\cup\cO'\), and satisfies \eqref{equation: additivity} since \(\pi_sk_{\cO'}\) is finite. By the complex ascent case \eqref{equation: complex ascent}, we have 
\begin{equation}
\begin{split}
m(\pi_s^\ast\pi_{s\ast}(k_{\cO!}\cL))&=m(\pi_s^\ast\pi_{s\ast}k_{\widehat\cO!})-m(\pi_s^\ast\pi_{s\ast}k_{\cO'!}\cL')\\
&=q\widehat\cL\\
&=\cL (T_s+1)
\end{split}
\end{equation}
by Definition~\ref{def: Ts}.
We conclude that \eqref{equation: push-pull} and \eqref{equation: equivalent push-pull} hold true for all complex reflections.

\subsubsection*{(c) Type II}

\begin{example}
\label{example: type II}
$(G_\R=PSL_2(\R))$ Let $G=PSL_2(\C)$, and take $\theta$ to be the involution
\[
\theta(g)= \bp 1 & 0 \\ 0 & -1 \ep g \bp 1 & 0 \\ 0 & -1 \ep.
\]
Then 
\[
K = \left\{ \bp a & 0 \\ 0 & a^{-1} \ep \mid a \in \C^\times \right\} \sqcup \left\{ \bp 0 & b \\ b^{-1} & 0 \ep \mid b \in \C^\times \right\} \hspace{2mm} (\mathrm{mod} \pm I)
\]
is disconnected and isomorphic to the complex orthogonal group. The group $K$ acts on the flag variety $X =\mathbb{P}^1$ with two orbits, 
\[
\mc{Q}_{0, \infty} = \{[0, 1], [1, 0]\} \text{ and } \mc{O} = \{[x, y] \mid x \neq 0, y \neq 0 \} \simeq \C^\times.
\]
The stabilizer $K_{[1,0]} = K_{[0,1]} \simeq \C^\times$ is connected, so there is one irreducible $K$-equivariant local system on the closed orbit $\mc{Q}_{0, \infty}$. For $[x, y] \in \mc{O}$, the stabilizer $K_{[x, y]} \simeq \Z/2\Z$ is disconnected, so there are two irreducible $K$-equivariant local systems on the open orbit $\mc{O}$, corresponding to the trivial and sign representations of $\Z/2\Z$. Denote the non-trivial $K$-equivariant local system by $\mc{L}$. The reflection $s$ is noncompact imaginary of type II (c1) with respect to $(\mc{Q}_{0, \infty}, \Q_{\mc{Q}_{0, \infty}})$ and real of type II (c2) with respect to $(\mc{O}, \Q_\mc{O})$ and $(\mc{O}, \mc{L})$.

The irreducible $K$-equivariant local systems on $\mc{O}$ have the same underlying local system and only differ in their $K$-equivariant structures. This can be seen by examining the long exact sequence in homotopy groups associated to the fibration 
\[
\Z/2\Z \simeq K_{[x,y]} \hookrightarrow K \twoheadrightarrow K/K_{[x,y]} \simeq \mc{O}.
\]
Indeed, the long exact sequence in homotopy groups has the form 
\[
\cdots \rightarrow \pi_1(\mc{O}) \xrightarrow{\alpha} \pi_0(K_{[x,y]}) \xrightarrow{\beta} \pi_0(K) \rightarrow 1, 
\]
so by exactness, $\beta$ is an isomorphism and the image of $\alpha$ is $1$. Pulling back via $\alpha$ gives a functor
\[
 \pi_0(K_{[x,y]})\mathrm{-mod}\xrightarrow{\alpha^*}
\pi_1(\mc{O})\mathrm{-mod}, 
\]
which determines the underlying local system of a $K$-equivariant local system on $\mc{O}$. Because $\alpha$ is trivial, we conclude that the $K$-equivariant local systems $\Q_\mc{O}$ and $\mc{L}$ must have the same (trivial) underlying local system.

Let $i_{0, \infty}: \mc{Q}_{0, \infty} \rightarrow X$ be inclusion, so $\pi_s \circ i_{0, \infty}$ is the projection of two points onto one point. There are two irreducible $K$-equivariant local systems on a point, $\Q_\mathrm{pt}$ and $\mc{L}_{\mathrm{sgn}}$, corresponding to the two irreducible representations of the component group of $K$, which is isomorphic to $\Z/2\Z$. Under the equivalence between $K$-equivariant local systems on a point and $\Z/2\Z$-representations, $(\pi_s \circ i_{0, \infty})_*\Q_{\mc{Q}_{0, \infty}}$ corresponds to the regular representation, so 
\[
(\pi_s \circ i_{0, \infty})_* \Q_{\mc{Q}_{0, \infty}} = \Q_\mathrm{pt} \oplus \mc{L}_\mathrm{sgn}.
\]
Hence 
\begin{align}
    \label{eq: type II noncompact imaginary}
    \pi_s^*\pi_{s*} i_{0, \infty !} \Q_{\mc{Q}_{0, \infty}} &= \pi_s^* (\Q_\mathrm{pt} \oplus \mc{L}_\mathrm{sgn}) \\
    \nonumber &= \Q_X \oplus \widehat{\mc{L}},
\end{align}
where $\widehat{L} = \pi_s^*\mc{L}_\mathrm{sgn}$ is the unique extension of $\mc{L}$ to $X$. This shows that \eqref{equation: push-pull} holds for $(\mc{Q}_{0, \infty}, \Q_{\mc{Q}_{0, \infty}})$.

Let $j: \mc{O} \rightarrow X$ be inclusion of the open orbit. First, note that 
\begin{equation}
    \label{eq: computing using k_X}
    m(\pi_s^* \pi_{s*} \Q_X) = m(\pi_s^* \pi_{s*} i_{0, \infty !} \Q_{\mc{Q}_{0, \infty}}) + m(\pi_s^* \pi_{s*} j_! \Q_{\mc{O}})
\end{equation}
by \eqref{equation: additivity}. (We obtain a distinguished triangle satisfying the conditions of \eqref{equation: additivity} by applying the triangulated functor $\pi_s^* \pi_{s*}$ to the open-closed distinguished triangle for the pair $(j, i_{0, \infty})$ and using \eqref{eq: type II noncompact imaginary}.) We have
\[
\pi_s^* \pi_{s*} \Q_X = \Q_X \oplus \Q_X[-2]
\]
by Lemma~\ref{lemma: derived pushforward}, so 
\begin{equation}
\label{eq: from k_X to orbits}
m(\pi_s^* \pi_{s*} \Q_X) = (1+q) (\Q_{\mc{Q}_{0, \infty}} + \Q_{\mc{O}}). 
\end{equation}
Combining equations \eqref{eq: computing using k_X}, \eqref{eq: from k_X to orbits} and \eqref{eq: type II noncompact imaginary}, we see that 
\[
m(\pi_s^* \pi_{s*} j_! \Q_\mc{O}) = q\Q_{\mc{O}} + (q-1)\Q_{\mc{Q}_{0, \infty}} - \mc{L},  
\]
so \eqref{equation: push-pull} holds for the parameter $(\mc{O}, \Q_{\mc{O}})$. 

Finally, to see that \eqref{equation: push-pull} holds for the parameter $(\mc{O}, \mc{L})$, we apply Lemma~\ref{lemma: derived pushforward} to $\pi_s$ and  $\mc{L}_\mathrm{sgn}$. Specifically, if we denote the extension of $\mc{L}$ to $X$ by $\widehat{\mc{L}}$ as above, then $\pi_s^* \mc{L}_\mathrm{sgn} = \widehat{\mc{L}}$, so 
\begin{equation}
    \label{eq: push pull L}
\pi_s^* \pi_{s*} \widehat{\mc{L}} =  \widehat{\mc{L}} \oplus \widehat{\mc{L}}[-2].
\end{equation}
Analogously to the previous case, we have 
\begin{equation}
m(\pi_s^*\pi_{s*} \widehat{\mc{L}}) = m(\pi_s^* \pi_{s*} i_{0, \infty !} \Q_{\mc{Q}_{0, \infty}}) + m(\pi_s^* \pi_{s*} j_!\mc{L}),
\end{equation}
so $m(\pi_s^* \pi_{s*} j_!\mc{L})$ can be computed using \eqref{eq: push pull L} and \eqref{eq: type II noncompact imaginary}. The result of this computation is that \eqref{equation: push-pull} also holds for the parameter $(\mc{O}, \mc{L})$, so it holds for all type II reflections in this example. 

\end{example}

In general, for a noncompact imaginary reflection \(s\) of type II with respect to \((\cO=KxB/B,\cL)\), we have
\begin{equation}
\begin{tikzcd}
KxB\times_BP_s/B\arrow[r,"\bar\mu"]\arrow[d,"\bar\pi"]&\widehat\cO\arrow[d, "\pi_sk_{\widehat\cO}"]\\
KxB/B\arrow[r, "\pi_s'"]&KxP_s/P_s
\end{tikzcd}
\end{equation}
such that \(\bar\mu\) and \(\pi_s'\) are finte maps of degree two.
We have
\begin{equation}\label{equation: old 4.50}
\bar\mu_!\bar\pi^\ast\cL=(\pi_sk_{\widehat\cO})^\ast\Ind_{K_x}^{K\cap xP_sx^{-1}}(\cL)=\widehat\cL_1\oplus \widehat\cL_2
\end{equation}
for local systems \(\widehat\cL_1\) and \(\widehat\cL_2\) on \(\widehat\cO\), which are distinct when restricted to the open orbit \(\cO'=K/K_y\) in \(\widehat\cO\).

By Definition~\ref{def: Ts},
\begin{equation}\label{equation: old 4.51}
\begin{split}
m(\pi_s^\ast\pi_{s\ast}k_{\cO!}\cL)&=m(\widehat\cL_1)+m(\widehat\cL_2)\\
&=2\cL+(\widehat\cL_1+\widehat\cL_2)\vert_{\cO'}\\
&=\cL(T_s+1)
\end{split}
\end{equation}
so \eqref{equation: push-pull} holds true for all type II noncompact imaginary reflections.

Suppose \(s\) is type II real with respect to  \((\cO,\cL)\), and let \(y\) be such that \(\widehat\cO=KyB/B\cup KxB/B\).
By assumption, there exists \(\cL'\) on \(\cO'=KyB/B\) such that
\begin{equation}
\begin{tikzcd}
KyB\times_BP_s/B\arrow[r, "\bar\mu"]\arrow[d, "\bar\pi"]&\widehat\cO\arrow[d, "\pi_sk_{\widehat\cO}"]\\
KyB/B\arrow[r, "\pi_s'"]&KyP_s/P_s
\end{tikzcd}
\end{equation}
gives rise to one of the local systems \(\widehat\cL\) from \eqref{equation: old 4.50} satisfying \(\widehat\cL\vert_{\cO}=\cL\).
Then \(s\) is type II noncompact imaginary with respect to \((\cO',\cL')\).
By \eqref{equation: old 4.50}, \(\widehat\cL\) is the pull-back of a \(K\)-equivariant local system on \(KyP_s/P_s\), so
\begin{equation}
m(\pi_s^\ast\pi_{s\ast}(k_{\widehat\cO!}\widehat\cL))=(q+1)\widehat\cL
\end{equation}
by Lemma~\ref{lemma: derived pushforward} and \eqref{equation: shift}.
By the distinguished triangle \eqref{equation: complex dt} corresponding to the open-closed decomposition \(\widehat\cO=\cO\cup\cO'\) and the type II noncompact imaginary case \eqref{equation: old 4.51}, we have
\begin{equation}
\begin{split}
m(\pi_s^\ast\pi_{s\ast}(k_{\cO!}\cL))&=m(\pi_s^\ast\pi_{s\ast}k_{\widehat\cO!}\widehat\cL)-(2\cL'+\cL+\widehat\cL_2\vert_{\cO})\\
&=q\cL-\widehat\cL_2\vert_{\cO}+(q-1)\cL'\\
&=\cL(T_s+1)
\end{split}
\end{equation}
where \(\widehat\cL_2\) is the other \(K\)-equivariant local system on \(\widehat\cO\).
It follows that \eqref{equation: push-pull} and \eqref{equation: equivalent push-pull}
hold true for every type II real reflection.

\subsubsection*{(d) Type I}

\begin{example}\label{example: Type I}
$(G_\R=SL_2(\R))$ Let $G, K$, and $\theta$ be as in Example~\ref{example: LV module for SL(2,R)}. Denote by $i_0$ and $i_\infty$ the inclusions of the single-point orbits $\mc{Q}_0$ and $\mc{Q}_\infty$ into $X = \mathbb{P}^1$, respectively, and $j:\mc{O} \rightarrow X$ the inclusion of the open orbit. The reflection $s$ is type I noncompact imaginary (d1) for $(\mc{Q}_0, \Q_{\mc{Q}_0})$ and $(\mc{Q}_\infty, \Q_{\mc{Q}_\infty})$, and type I real (d2) for $(\mc{O}, \Q_{\mc{O}})$.

As in equation \eqref{eq: complex ascent}, we have
\begin{equation}
\label{eq: type I noncompact imaginary}
\begin{split}
\pi_s^\ast\pi_{s\ast}i_{0!}\Q_{\mc{Q}_0}&=\Q_X, \text{ and }\\
\pi_s^\ast\pi_{s\ast}i_{\infty !}\Q_{\mc{Q}_\infty}&=\Q_X.\\
\end{split}
\end{equation}
It follows that \eqref{equation: push-pull} holds for $(\mc{Q}_0, \Q_{\mc{Q}_0})$ and $(\mc{Q}_\infty, \Q_{\mc{Q}_\infty})$. 

To show that \eqref{equation: push-pull} holds for the parameter $(\mc{O}, \Q_\mc{O})$, we proceed as in Example~\ref{example: type II}. Let $i_{0, \infty}: \mc{Q}_0 \cup \mc{Q}_\infty \rightarrow X$ be the inclusion of the union of the two single point orbits into $X$. By applying the triangulated functor $\pi_s^* \pi_{s*}$ to the open-closed distinguished triangle for the pair $(i_{0, \infty}, j)$, we obtain the distinguished triangle
\begin{equation}
\label{equation: dt for summing m}
\pi_s^* \pi_{s*} j_! \Q_\mc{O} \rightarrow
\pi_s^* \pi_{s*} \Q_X \rightarrow \pi_s^* \pi_{s*} i_{0, \infty *} \Q_{\mc{Q}_0 \cup \mc{Q}_\infty} \xrightarrow{+1}.
\end{equation}
By applying \eqref{equation: additivity} to the triangle \eqref{equation: dt for summing m}, we see that
\begin{equation}
    \label{eq: character sums SL(2,R)}
    m(\pi_s^* \pi_{s*} \Q_X) = m(\pi_s^* \pi_{s*} j_! \Q_\mc{O}) + m(i_{0, \infty *} \Q_{\mc{Q}_0 \cup \mc{Q}_\infty}).
\end{equation}
By Lemma~\ref{lemma: derived pushforward}, 
\begin{equation}
\label{eq: push-pull structure sheaf for SL(2,R)}
\pi_s^* \pi_{s*} \Q_X = \Q_X \oplus \Q_X[-2],
\end{equation}
and 
\begin{equation}
    \label{eq: push-pull 2 points for SL(2,R)}
    \pi_s^* \pi_{s*} i_{0, \infty *} \Q_{\mc{Q}_0 \cup \mc{Q}_\infty} = \Q_X \oplus \Q_X. 
\end{equation}
Combining \eqref{eq: character sums SL(2,R)}, \eqref{eq: push-pull structure sheaf for SL(2,R)}, and \eqref{eq: push-pull 2 points for SL(2,R)}, we have 
\[
m(\pi_s^* \pi_{s*} j_! \Q_\mc{O}) = (q-1)( \Q_{\mc{Q}_0} + \Q_{\mc{Q}_\infty} + \Q_\mc{O}).
\]
It follows that \eqref{equation: push-pull} holds for $(\mc{O}, \Q_{\mc{O}})$. 
\end{example}

In general, let \(s\) be noncompact imaginary reflection of type I with respect to \((KxB/B,\cL)\).
We have
\begin{equation}
\begin{tikzcd}
KxB\times_BP_s/B\arrow[r,"\bar\mu"]\arrow[d,"\bar\pi"]&\widehat\cO\arrow[d, "\pi_sk_{\widehat\cO}"]\\
KxB/B\arrow[r, "\pi_s'"]&KxP_s/P_s
\end{tikzcd}
\end{equation}
such that \(\bar\mu\) and \(\pi_s'\) are isomorphisms, so
\begin{equation}\label{equation: unique type I extension}
\widehat\cL=\bar\mu_!\bar\pi^\ast\cL=(\pi_sk_{\widehat\cO})^\ast\pi_{s!}'\cL
\end{equation}
extends \(\cL\) to \(\widehat\cO\).

By Definition~\ref{def: Ts}, we have
\begin{equation}\label{equation: noncompact imaginary I}
\begin{split}
m(\pi_s^\ast\pi_{s\ast}k_{\cO!}\cL)&=m(\widehat\cL)\\
&=\cL+\widehat\cL\vert_{\cO'}+\widehat\cL\vert_{\cO^{\prime\prime}}\\
&=\cL(T_s+1)
\end{split}
\end{equation}
where \(\cO'\) and \(\cO''\) are the \(K\)-orbits such that \(\widehat\cO=\cO\cup\cO'\cup\cO''\) and \(\dim(\cO)=\dim(\cO'')=\dim(\cO')-1\).
By \eqref{equation: this guy}, \eqref{equation:  push-pull} holds true for all noncompact imaginary reflections of type I.

Now assume \(s\) is real type I with respect to \((\cO,\cL)\), so there exists \(\cL'\) on \(\cO'=KyB/B\) such that 
\begin{equation}
\begin{tikzcd}
KyB\times_BP_s/B\arrow[r,"\bar\mu"]\arrow[d,"\bar\pi"]&\widehat\cO\arrow[d, "\pi_sk_{\widehat\cO}"]\\
KyB/B\arrow[r, "\pi_s'"]&KyP_s/P_s
\end{tikzcd}
\end{equation}
gives rise to the unique local system \(\widehat\cL\) from \eqref{equation: unique type I extension}. 
We have
\begin{equation}
m(\pi_s^\ast\pi_{s\ast}(k_{\widehat\cO!}\widehat\cL))=(q+1)\widehat\cL
\end{equation}
by Lemma~\ref{lemma: derived pushforward} and \eqref{equation: shift}.
By the distinguished triangle \eqref{equation: complex dt} corresponding to the open-closed decomposition \(\widehat\cO=\cO\cup(\cO'\cup\cO'')\) and the type I noncompact imaginary case \eqref{equation: noncompact imaginary I}, we have
\begin{equation}
\begin{split}
m(\pi_s^\ast\pi_{s\ast}(k_{\cO!}\cL))&=m(\pi_s^\ast\pi_{s\ast}(k_{\widehat\cO!}\widehat\cL))-2m(k_{\widehat\cO!}\widehat\cL)\\
&=(q-1)\widehat\cL\\
&=\cL(T_s+1)
\end{split}
\end{equation}
by Definition~\ref{def: Ts}.
It follows that \eqref{equation: push-pull} and \eqref{equation: equivalent push-pull} hold true for every type I real reflection.

\subsubsection*{(e) Real Nonparity}

{\color{black}
\begin{example}\label{example: non-parity} $(G_\R = SL_2(\R))$
We return to the setting of Examples \ref{example: LV module for SL(2,R)} and \ref{example: Type I}. With respect to the M\"obius band local system $\mc{L}$ on the open orbit $\mc{O}$, the reflection $s$ is type (e), real nonparity. In particular, the local system $\mc{L}$ does not extend (as a local system) to $X$, so we have
\[
j_!\mc{L} = j_* \mc{L},
\]
where $j:\mc{O} \rightarrow X$ is inclusion. Hence we can compute $\pi_{s*}j_!\mc{L}=(\pi_s\circ j)_*\mc{L}$ by computing the sheaf cohomology of $\C^\times$ with coefficients in the local system $\mc{L}$. We will show that this cohomology is zero. 

Forgetting equivariance, the local system $\For \mc{L}$ consists of the data of a one-dimensional $\Q$-vector space $V$ and the monodromy map $m=-1:V \rightarrow V$. The cohomology of $\C^\times$ with coefficients in $\For \mc{L}$ is given by the invariants and coinvariants of $m$ \cite[Ex.\ 17.6]{RomWil}:
\[
H^i(\C^\times; \For \mc{L}) = \begin{cases} V^m & i=0; \\
V_m & i=1. 
\end{cases}
\]
Since $V^m=V_m=0$, we see that on the non-equivariant level, $\pi_{s*}j_! \For \mc{L}=0$. Hence the equivariant complex $\pi_{s*}j_!\mc{L}$ must also vanish. 

Alternately, we can see this using the finite Galois covering
\[
f: \C^\times \rightarrow \C^\times; z \mapsto z^2.
\]
Again, we forget equivariance. In the equivalence between local systems and representations of fundamental groups, pushing along a finite covering map corresponds to inducing representations. Hence the local system $f_* \Q_{\C^\times}$ corresponds to the induced representation $\Ind_{2\Z}^\Z \mathrm{triv}$, which is isomorphic to the regular representation of $\Z/2\Z$. We conclude that 
\[
f_* \Q_{\C^\times} = \Q_{\C^\times} \oplus \For \mc{L}. 
\]
Pushing along $\pi_s\circ j$, we see that 
\[
\pi_{s*}j_*\Q_{\C^\times} = \pi_{s*}j_*f_* \Q_{\C^\times} = \pi_{s*}j_*\Q_{\C^\times} \oplus \pi_{s*}j_*\For \mc{L}, 
\]
because $\pi_s \circ j = \pi_s \circ j \circ f$. Hence $\pi_{s*}j_! \For \mc{L} = 0$. 

From either of these two computations, we conclude that 
\begin{equation}
\label{eq: type e}
\pi_s^* \pi_{s*} j_! \mc{L} = 0.
\end{equation}
It follows that \eqref{equation: push-pull} holds for this type (e) reflection. 

\end{example}
}
In general, suppose \(s\) is real nonparity with respect to \((\cO,\cL)\).
We start with a Lemma.

\begin{lemma}\label{lemma: reduction to P^1}
For \(x\in\cO\),
\begin{equation}\label{equation: reduction to P^1}
\widehat\cO=K\times_{K\cap xP_sx^{-1}}L_x^s.
\end{equation}
\begin{proof}
The quotient in \eqref{equation: reduction to P^1} is an algebraic variety since \(L:=L_x^s\) is quasiprojective and \(K\cap xP_sx^{-1}\) is a closed subgroup of \(K\).
Define
\begin{equation}
\eta\,\colon K\times_{K\cap xP_sx^{-1}}xP_s/B\to KxP_s/B
\end{equation}
by multiplication.
We have
\begin{equation}
\eta^{-1}(kxpB/B)=\left\{[k',xp']\mid k'xp'B/B=kxpB/B\right\}=\left\{[k,xp]\right\}
\end{equation}
since \([k',xp']=[k,k^{-1}k'xp']=[k,xp]\) by a routine calculation.
It follows that \(\eta\) is an isomorphism since \(KxP_s/B\) is smooth and \(\eta\) is bijective.
\end{proof}
\end{lemma}

In particular, \eqref{equation: reduction to P^1} shows that \(K\times_{K\cap xP_sx^{-1}}L\) is a Zariski locally trivial fiber bundle over \(KxP_s/P_s=K/(K\cap xP_sx^{-1})\) with fiber \(L\) (even when \(K\cap xP_sx^{-1}\) need not be a parabolic subgroup of \(K\)).
There are finitely many \(K\)-orbits in \(\widehat\cO\), which correspond to \(K\cap xP_sx^{-1}\) orbits in \(L\) via \eqref{equation: reduction to P^1}; this is computed explicitly in \cite[Lem.\ 5.1]{Vogan3}.

We have an equivalence of t-categories 
\begin{equation}\label{equation: key trick}
D_K^b(KxP_s/B)=D^b_{K\cap xP_sx^{-1}}(L)
\end{equation}
by induction.
If \(y\in KxP_s/B\) and \(k\,\colon KyB/B\to KxP_s/B\) is inclusion, then \(\cL=k^\ast k_!\cL\) corresponds to the same representation of \(K_y/K_y^0\) in either category.
Under the equivalence \eqref{equation: key trick},
\begin{equation}
\pi_{s_\ast}\cF=p_\ast\cF
\end{equation}
where \(p\,\colon L\to\pt\) maps to a point.

It follows from Example~\ref{example: non-parity} that for any real nonparity reflection we have
\begin{equation}\label{equation: real nonparity}
\cL(T_s+1)=0.
\end{equation}
This completes the proof of Lemma~\ref{lemma: local systems}.

\subsection{Proof of Theorem~\ref{thm: semisimple categorification}}\label{sec: proof of geometric categorification}
For \(y\in G/B\), let \(k'\,\colon KyP_s/B\to G/B\).
Let \(U\) be the union of the open \(K\)-orbit \(Ky_0B/B\) in \(KyP_s/B\) and the complement of \(\overline{KyP_s/B}\) in \(G/B\).
Let \(j\,\colon U\to G/B\) be the corresponding open embedding.
Let \(Y\) be the complement of \(U\) in \(G/B\), so \(Y\) is the closure in \(G/B\) of a (possibly empty) disjoint union of closed \(K\)-orbits \(Ky_iB/B\) in \(KyP_s/B\).
Let \(i\,\colon Y\to G/B\) the corresponding closed inclusion.
The maps \(i\) and \(j\) give an open-closed decomposition of \(G/B\).
The open-closed distinguished triangle
\begin{equation}\label{equation: open closed dt}
j_!j^\ast\widetilde \cL\to\widetilde\cL\to i_\ast i^\ast\widetilde\cL\overset{+1}{\longrightarrow}
\end{equation}
gives the distinguished triangle
\begin{equation}\label{equation: triangulated functor}
\pi_s^\ast\pi_{s\ast} j_!j^\ast\widetilde\cL\to \pi_s^\ast\pi_{s\ast}\widetilde\cL\to \pi_s^\ast\pi_{s\ast}i_\ast i^\ast\widetilde\cL\overset{+1}{\longrightarrow}
\end{equation}
under an application of the triangulated functor \(\pi_s^\ast\pi_{s\ast}\).
A further application of the triangulated functor \(k'_!k^{\prime\ast}\) yields the distinguished triangle
\begin{equation}\label{equation: restriction is triangulated}
k'_!k^{\prime\ast}\pi_s^\ast\pi_{s\ast}j_!j^\ast\widetilde\cL\to k_!'k^{\prime\ast}\pi_s^\ast\pi_{s\ast}\widetilde\cL\to k_!'k^{\prime\ast}\pi_s^\ast\pi_{s\ast}i_\ast i^\ast\widetilde\cL\overset{+1}{\longrightarrow}.
\end{equation}
Simplifying, we obtain
\begin{equation}\label{equation: simplified dt}
k_!'(\pi_sk')^\ast(\pi_sj)_!j^\ast \widetilde\cL\to k_!'k^{\prime\ast}\pi_s^\ast\pi_{s\ast}\widetilde\cL\to k'_!(\pi_sk')^\ast(\pi_si)_\ast i^\ast\widetilde\cL\overset{+1}{\longrightarrow}.
\end{equation}

The diagram
\begin{equation*}
\begin{tikzcd}
&Ky_0B\times_BP_s/B\arrow[d]\arrow[ddl, bend right=30,swap, "\pi"]\arrow[r, "\mu"]&KyP_s/B\arrow[r, "k'"]\arrow[d, "\pi_s k'"]&G/B\\
&U\arrow[d, "j"]\arrow[r, "\pi_s j"]&G/P_s\\
Ky_0B/B\arrow[r, "j'"]&G/B&&\\
\end{tikzcd}
\end{equation*}
shows 
\begin{equation}
k_!'(\pi_sk')^\ast(\pi_sj)_!j^\ast\widetilde\cL=k'_!\mu_!\pi^\ast j^{\prime\ast}\widetilde\cL,
\end{equation}
since the upper right square is cartesian.
A similar diagram shows
\begin{equation}\label{equation: finite mu}
k_!'(\pi_sk')^\ast(\pi_si)_\ast i^\ast\widetilde\cL=k'_!\mu_!\pi^\ast i^{\prime\ast}\widetilde\cL,
\end{equation}
where \(i'\,\colon \coprod_i Ky_iB/B\to G/B\).
By Lemma~\ref{lemma: local systems}, we have
\begin{equation}
\begin{split}
m(k_!'\mu_!\pi^\ast j^{\prime\ast}\widetilde\cL)&=m(j'_!j^{\prime\ast}\widetilde\cL)(T_s+1)\\
m(k_!'\mu_!\pi^\ast i^{\prime\ast}\widetilde\cL)&=m(i'_!i^{\prime\ast}\widetilde\cL)(T_s+1)
\end{split}
\end{equation}
since restricting to a \(K\)-orbit gives a direct sum of irreducible \(K\)-equivariant local systems.
By \eqref{equation: additivity}, \eqref{equation: simplified dt}, and \eqref{equation: finite mu}, we have
\begin{equation}
m(k'_!k^{\prime\ast}\pi_s^\ast\pi_{s\ast}\widetilde\cL)=(m(j'_!j^{\prime\ast}\widetilde\cL)+m(i'_!i^{\prime\ast}\widetilde\cL))(T_s+1)
\end{equation}
since \(k_!'\mu_!\pi^\ast i^{\prime\ast}\widetilde\cL\) is \(\ast\)-even (q.v.\ Definition~\ref{def: parity}), as seen by observing that \(\mu\) is finite in \eqref{equation: finite mu} so \(\mu_!\) is exact.
It follows from \eqref{equation: additivity} that
\begin{equation}
(m(j'_!j^{\prime\ast}\widetilde\cL)+m(i'_!i^{\prime\ast}\widetilde\cL))(T_s+1)=m(k'_!k^{\prime\ast}\widetilde\cL)(T_s+1)
\end{equation}
by applying \(k_!'\) to the  open-closed decomposition in \(KyP_s/B\) of \(k^{\prime\ast}\widetilde\cL\).
Therefore
\begin{equation}\label{equation: proof of theorem}
m(\pi_s^\ast\pi_{s\ast}\widetilde\cL)=\sum_{k'}m(k'_!k^{\prime\ast}\pi_s^\ast\pi_{s\ast}\widetilde\cL)=m(\widetilde\cL)(T_s+1)
\end{equation}
as desired.

Recall that our initial character map \(\ch\) is given in \eqref{equation: ch}.
We prove in Lemma~\ref{lemma: push pull star even} below that \(\pi_s^\ast\pi_{s\ast}\widetilde\cL\) is \(\ast\)-even, so by \eqref{equation: ch is m}
\begin{equation}
\begin{split}
\ch(\pi_s^\ast\pi_{s\ast}\widetilde\cL)&=m(\pi_s^\ast\pi_{s\ast}\widetilde\cL)\\
&=m(\widetilde\cL)(T_s+1)\\
&=\ch(\widetilde\cL)(T_s+1)
\end{split}
\end{equation}
since \(\widetilde\cL\) is \(\ast\)-even (q.v.\ Remark \ref{remark: purity}).
The theorem then follows from the fact that for any \(\cF\in D_K^b(X)\),
\begin{equation}
\ch(\cF[1])=q^{-1/2}\ch(\cF).
\end{equation}

We note that \(m\) is required to obtain additivity on certain distinguished triangles (q.v.\ \eqref{equation: additivity}) but \(m\) fails to intertwine Hecke actions on odd shifts of \(\widetilde\cL\), so we must move to \(\ch\) to obtain a character map which intertwines Hecke actions on the entire semisimple category.

\subsection{Application: Resolutions of singularities}
\label{sec: resolutions of singularities}

We give a formula to compute the cohomology of  fibers of the resolutions of singularities defined in \cite{Larson2020}.

For a subset \(I\subseteq S\), we denote by \(w_I\) the longest element of \(W_I\), \(P_I=\overline{Bx_IB}\) the corresponding standard parabolic subgroup, and \(p_I\) the Poincar\'e polynomial of \(P_I/B\) with indeterminate \(q^{1/2}\).
Let \((W,\star)\) be the Richardson-Springer monoid associated to any Coxeter group \((W,S)\), and recall the monoid action on \(K\backslash G/B\) (cf.\ \cite{RichardsonSpringer1990}).
For \(x_0\in K\backslash G/B\), for \(x_1,\ldots,x_\ell\in B\backslash G/B\), and for \(1\leq i\leq \ell\), choose
\begin{equation}
J_i\subseteq\tau(x_{i-1})\cap\tau(x_i^{-1}),
\end{equation}
where 
\begin{equation}
\tau(x_i)=\left\{s\in S\mid x_i\star s=x_i\right\}.
\end{equation}
Given a subset \(J\subseteq\tau(x_\ell)\), choose \(I\supseteq J\) and set \(x:=x_0\star x_1\star\cdots\star x_\ell\star x_I\).
Set
\begin{equation}
Z=\overline{Kx_0B}\times_{P_{J_1}}\overline{Bx_1B}\times_{P_{J_2}}\cdots\times_{P_{J_\ell}}\overline{Bx_\ell B}/P_J
\end{equation}
and define \(\xi\,\colon Z\to G/P_I\) by
\begin{equation}\label{equation: resolution}
\xi[g_0,\ldots,g_\ell]=g_0\cdots g_\ell P_I/P_I.
\end{equation}
Then \(\xi\) is a \(K\)-equivariant proper map with image \(\overline{KxB}/P_I\).

Recall that \(m\,\colon D^b_K(G/B)\to M_{LV}\) is given by \eqref{equation: m}, and let \(h\,\colon D^b_B(G/B)\to \bm{H}\) be the corresponding character map for \(D^b_G(G\times_GG/B)=D^b_B(G/B)\) (as in \cite{Soergel2000}).

\begin{theorem}\label{theorem: fibers}
Suppose \(\overline {Kx_0B/B}\) is \(\Q\)-smooth, and for every \(1\leq i\leq \ell\), \(\overline{Bx_iB/B}\) is \(\Q\)-smooth.
Then for every \(y\in \overline{KxP_I/P_I}\), the Poincar\'e polynomial of \(\xi^{-1}(y)\)
with indeterminate \(q^{1/2}\) is computed by summing the coefficients of the local systems on \(KyB/B\) in the expression
\begin{equation}\label{equation: big convolution}
m(\widetilde\Q_{Kx_0B/B})\left(\prod_{i=1}^\ell h(\widetilde\Q_{Bx_iB/B})\ \right) h(\widetilde\Q_{P_I/B})\ p^{-1}
\end{equation}
in \(M_{LV}\), where \(p=p_J\displaystyle\prod_{i=1}^\ell p_{J_i}\).
\end{theorem}

The cohomology of the fiber \(\xi^{-1}(y)\) is given by the stalk at \(y\) of the direct image \(\xi_\ast\Q_Z\), and Theorem~\ref{theorem: fibers} states that these stalks are determined by \(m(\xi_\ast\Q_Z)\).
For every \(\cL\in\cD\) and \(s\in S\) we prove in Lemma~\ref{lemma: push pull star even} that
\(\pi_s^\ast\pi_{s\ast}\widetilde\cL\)
is \(\ast\)-even (q.v.\ Definition~\ref{def: parity}).
It follows that \(\xi_\ast\Q_Z\) is \(\ast\)-even since the pull-back of \(\xi_\ast\Q_Z\) to \(G/B\) is a direct summand of even shifts of 
\begin{equation}
\pi_{s_n}^\ast\pi_{s_n\ast}\cdots\pi_{s_1}^\ast\pi_{s_1\ast}\widetilde\cL
\end{equation}
for some choice of \(\cL\) and \(s_i\).
Stalks of any \(\ast\)-even sheaf can be recovered from its image under \(m\) by \eqref{equation: cohomology sheaves}, so the cohomology of \(\xi^{-1}(y)\) is computed by \(m(\xi_\ast\Q_Z)\).

The explicit formula for \(m(\xi_\ast\Q_Z)\) in \eqref{equation: big convolution} comes from first describing the direct image \(\xi^{\prime\prime}_\ast\Q_Z\) in terms of convolution of \(\cH\) on \(\cM_{LV}\), where
\begin{equation}
\begin{tikzcd}
\overline{Kx_0B}\times_B\overline{Bx_1B}\times_B\cdots\times_B\overline{Bx_\ell B}\times_BP_I/B\arrow[r, "\xi^{\prime\prime}"]\arrow[d, "\zeta"]&G/B\arrow[d, "\id"]\\
\overline{Kx_0B}\times_{P_{J_1}}\overline{Bx_1B}\times_{P_{J_2}}\cdots\times_{P_{J_\ell}}\overline{Bx_\ell B}\times_{P_J}P_I/B\arrow[r,"\xi'"]&G/B
\end{tikzcd}
\end{equation}
is such that \(\xi'\) is pullback of \(\xi\) along \(G/B\to G/P_I\).
The smooth map \(\zeta\) has fiber \(P_{J_1}/B\times\cdots\times P_{J_\ell}/B\times P_J/B\), which explains why we divide by \(p\) in \eqref{equation: big convolution} (cf.\ \cite[Eq.\ (2.9)]{Williamson2011}).
From this discussion, all that remains to prove Theorem~\ref{theorem: fibers} is to establish that \(\pi_s^\ast\pi_{s\ast}\widetilde\cL\) is \(\ast\)-even.
We do this in Lemma~\ref{lemma: push pull star even} below.

For every \((KxB/B,\cL)\in\cD\), define
\begin{equation}
\tau(\cL)=\left\{s\in S\mid x\star s=x,\ \cL\text{ extends to }KxP_s/B\right\} 
\end{equation}
as in \cite[Def.\ 5.6]{Vogan3}, called the \emph{Borho-Jantzen-Duflo \(\tau\)-invariant}.

\begin{lemma}\label{lemma: push pull star even}
For every \(\cL\in\cD\) we have
\begin{equation}\label{equation: push-pull star even}
\pi_s^\ast\pi_{s\ast}\widetilde\cL\in D^b_K(G/B)^{\ast-even}.
\end{equation}
\begin{proof}
If \(s\in\tau(\cL)\), then \(\cL=\pi_s^\ast\cL'\) for some irreducible \(K\)-equivariant local system \(\cL'\) extended by zero to \(G/P_s\).
By \cite[Prop.\ 7.15]{AdamsBarbaschVogan1992}, it follows that \(\widetilde\cL=\pi_s^\ast\widetilde{\cL'}\), so \eqref{equation: push-pull star even} holds true by Lemma~\ref{lemma: derived pushforward}.
If \(s\in\tau(x)\smallsetminus\tau(\cL)\), then \(\pi_{s\ast}\widetilde\cL=0\) by restricting to \(KxP_s/P_s\).
Otherwise \(\pi_s\) restricts to a generically finite map on the support of \(\widetilde\cL\).

Assume from now on that \(\pi_s\) restricts to a generically finite map on the support of \(\widetilde\cL\).
We claim that \(\pi_{s\ast}\IC(\overline\cO,\cL)\) is perverse -- even though the direct image of \(\IC(\overline\cO,\cL)\) under a semi-small map need not be perverse.
This will enable us to provide an upper bound for possible terms occurring in the decomposition theorem, and moreover will force each term to be shifted by zero.

Given \(y\in \overline{KxB}\) let \(k'\,\colon KyP_s/P_s\to G/P_s\) be inclusion.
The cartesian diagram
\begin{equation}\label{equation: cartesian strata}
\begin{tikzcd}
KyP_s/B\cap\overline{KxB/B}\arrow[r, "\pi_s'"]\arrow[d, "k''"]& KyP_s/P_s\arrow[d, "k'"]\\
\overline{KxB/B}\arrow[r, "\pi_s\bar k"]&G/P_s
\end{tikzcd}
\end{equation}
gives 
\begin{equation}\label{equation: restriction}
k^{\prime\ast}\pi_{s\ast}\IC(\overline\cO,\cL)=\pi_{s\ast}'k^{\prime\prime\ast}\IC(\overline\cO,\cL)
\end{equation}
by base change.
Equation \eqref{equation: restriction} allows us to relate cohomology sheaves of the direct image to the direct image over each stratum.

The following picture illustrates \(\widetilde\cL\) on the left and \(\pi_{s\ast}\widetilde\cL\) on the right.
The blue and red symbols indicate possible nonzero cohomology sheaves when restricted to orbits.
Note that \(\widetilde\cL\) has zero odd cohomology sheaves by Remark~\ref{remark: purity}.
\[\includegraphics[width=.4\textwidth]{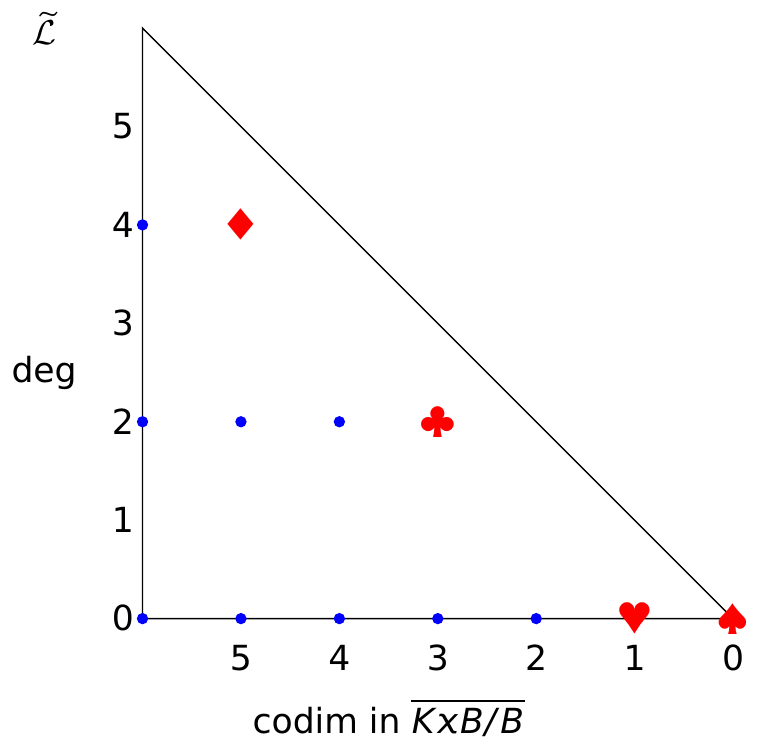}
\includegraphics[width=.4\textwidth]{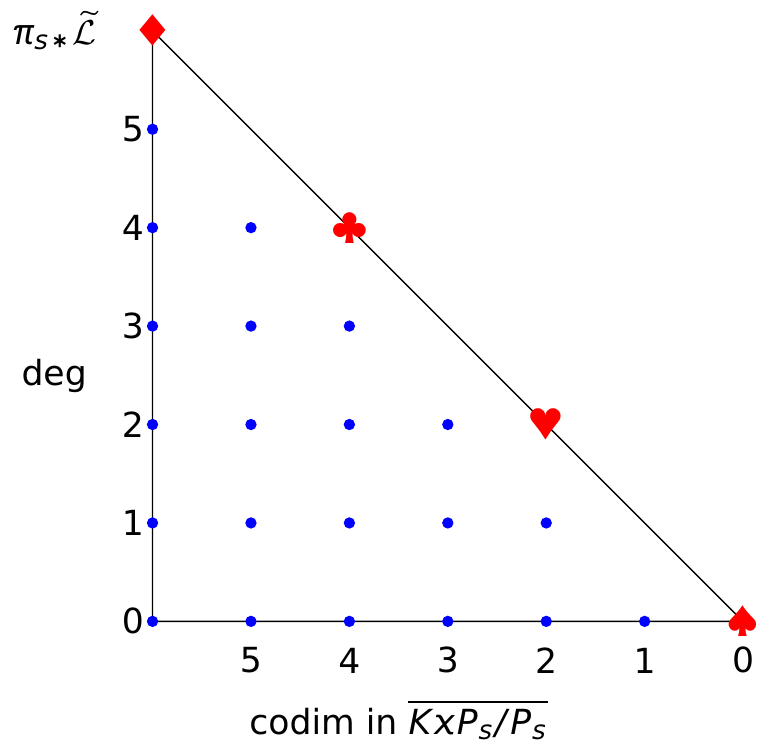}\]
The origin maps to the origin under \(\pi_{s\ast}\) since we are assuming \(\pi_s'\) \eqref{equation: cartesian strata} is finite over the open orbit; i.e., the left spade is always mapped to the right spade.
The other red symbols need not map as depicted, but their locations on the right indicate an upper bound (with respect to the diagonal) for nonzero cohomology sheaves.
The red symbols on the right represent all coordinates on the diagonal of \(\pi_{s\ast}\widetilde\cL\) with possible nonzero cohomology sheaves.
This can occur when the coordinate for \(\widetilde\cL\) is distance \(\sqrt2\) from the diagonal and the map \(\pi_s'\) \eqref{equation: cartesian strata} has one dimensional fiber (which contracts the stratum by one dimension).
In particular, every possible nonzero term on the diagonal has even degree.
There are no nonzero terms above the diagonal, so it follows that \(\pi_{s\ast}\IC(\overline\cO,\cL)\) is perverse such that
\begin{equation}\label{equation: decomposition theorem}
\pi_{s\ast}\IC(\overline\cO,\cL)=\bigoplus_{\cL'}\IC(\overline{\cO'},\cL'),
\end{equation}
where the dimension of all \(\cO'\subseteq\overline{KxP_s/P_s}\) appearing in the sum have same parity.
The claim follows by shifting degree of \eqref{equation: decomposition theorem} by \(-\dim\cO\).
\end{proof}
\end{lemma}

\begin{definition}
\label{def: clean}
An element \((\cO,\cL)\in\cD\) is \emph{clean} if 
\begin{equation}\label{equation: clean}
\widetilde\cL=k_{\cO!}\cL.
\end{equation}
\end{definition}

\begin{remark}\label{remark: purity}
Every block (q.v.\ Definition \ref{def: block equivalence}) is generated by clean local systems, which are \(\ast\)-even since their cohomology sheaves vanish outside of degree zero.
It follows that the proof of Lemma~\ref{lemma: push pull star even} can also be used to show the classic fact \cite[Thm.\ 1.12(a)]{LV} that \(\widetilde\cL\in D_K^b(G/B)^{\ast-even}\).
\end{remark}


\section{A categorification of the trivial block}
\label{sec: a categorification of the trivial block}

In Section \ref{sec: a geometric categorification}, we provide a geometric categorification of the entire Lusztig--Vogan module $M_{LV}$. By considering the subcategory of $\mc{M}_{LV}^{ss}$ generated by trivial $K$-equivariant local systems on closed orbits, we can refine this categorification to give a categorification of the trivial block of $M_{LV}$. The purpose of considering this refinement is that on the subcategory generated by closed orbits, the hypercohomology functor is fully faithful (Theorem~\ref{thm: fully faithful}), which allows us to use Soergel-type techniques to describe the category algebraically. In this section, we recall the notion of blocks and give a geometric categorification of the trivial block. 


\subsection{Blocks}
\label{sec: Blocks}

There is a natural equivalence relation on the set $\mathscr{D}$  in (\ref{eq: D}) which partitions $M_{LV}$ into {\em blocks}. In this section we introduce a geometric formulation of this equivalence relation which is equivalent to the original algebraic formulation in \cite[Def.\ 9.2.1]{Vogan-green-book}, see \cite{Vogan4}.

Let $G, B, K$, and $X$ be as in Section \ref{sec: the geometry of the Lusztig--Vogan module}. For a parameter $(\mc{O}, \mc{L}) \in \mathscr{D}$, the cohomology sheaves $\mc{H}^i(\bm{IC}(\overline{\mc{O}}, \mc{L}))$ are objects in the abelian category of $K$-equivariant constructible sheaves on $X$. We denote by $[\mc{L}': \mc{H}^i(\bm{IC}(\overline{\mc{O}}, \mc{L})]$ the multiplicity of  $\mc{L}'\in \Loc_K(\mc{O}')$ (considered as a $K$-equivariant constructible sheaf on $X$ by extending by zero) in the composition series of $\mc{H}^i(\bm{IC}(\overline{\mc{O}}, \mc{L}))$.

\begin{definition}
\label{def: block equivalence} {\em Block equivalence} is the equivalence relation on $\mathscr{D}$ generated by 
\[
(\mc{O}, \mc{L}) \sim (\mc{O}', \mc{L}') \text{ if } [\mc{L}': \mc{H}^i(\bm{IC}(\overline{\mc{O}}, \mc{L})] \neq 0
\]
for some $i \in \Z$. Equivalence classes under $\sim$ are called {\em blocks}.
\end{definition}

\begin{remark}
\label{rmk: blocks}
We have formulated block equivalence to be the equivalence relation generated by non-zero KLV polynomials \cite[Thm.\ 1.12]{LV}. There are also equivalent formulations using Ext groups of $(\mf{g},K)$-modules \cite[Def.\ 9.2.1]{Vogan-green-book} and composition series of standard modules \cite[Def.\ 1.14]{Vogan4}.   
\end{remark}

The trivial local system on the open $K$-orbit is the Lusztig--Vogan parameter corresponding to the trivial representation of $G_\R$. Because of this, we refer to this block containing the trivial local system on the open orbit as {\em the block of the trivial representation}, and denote it $\mathscr{D}^0$. We focus our attention on this block for several reasons. One basic reason is that it is natural to study representations with the trivial representation as a composition factor. Another reason is that $\mc{D}^0$ is the block that governs birational resolutions of $K$-orbit closures, see Section \ref{sec: resolutions of singularities}. 

The following standard (but difficult) fact is vital to our arguments. This follows from \cite[Thm.\ 9.2.11]{Vogan-green-book} and the algorithm for computing KLV polynomials in \cite{Vogan3, LV}. A geometric proof can also be found in \cite[Prop.\ 3.6]{BV}, and another proof using Vogan duality is given in \cite[Rmk.\ 3.7]{BV}.

\begin{theorem}
\label{lem: closed orbits generate the block} Let $M_{LV}^0$ be the $\bm{H}$-module generated by the action of the operators $T_s, s \in S$ (Definition~\ref{def: Ts}) on $\{(\mc{Q}, \Q_\mc{Q}) \mid \mc{Q}\text{ closed $K$-orbit}\} \subseteq \mathscr{D}$. Then 
\[
M_{LV}^0 = \bigoplus_{\mc{L} \in \mathscr{D}^0} \Z[q^{\pm 1}] \mc{L}.
\]
\end{theorem}

\begin{example}
\label{example: blocks for SL(2,R)} $($Blocks for $SL_2(\R))$ We continue with the setup in Example~\ref{example: LV module for SL(2,R)}. The $\IC$ sheaf corresponding to the parameter $(\mc{O}, \Q_\mc{O}) \in \mathscr{D}$ is $\IC(X, \Q_\mc{O}) = \Q_X[1]$, which has cohomology sheaves
\[
\mc{H}^i(\IC(X, \Q_{\mc{O}})) = \begin{cases}
\Q_X & i=-1; \\
0 & \text{ else}.
\end{cases} 
\]
 Both $\Q_{\mc{Q}_0}$ and $\Q_{\mc{Q}_\infty}$ (extended by $0$ and considered as $K$-equivariant constructible sheaves on $X$) are subsheaves of $\Q_X$, so 
\[
(\mc{Q}_0, \Q_{\mc{Q}_0}) \sim (\mc{Q}_\infty, \Q_{\mc{Q}_\infty}) \sim (\mc{O}, \Q_\mc{O})
\]
in the equivalence relation of Definition~\ref{def: block equivalence}. As the M\"obius band local system $\mc{L}$ does not extend (as a local system) to $X$, the parameter $(\mc{O}, \mc{L})$ is not in the same block as any other Lusztig--Vogan parameter. Hence $\mathscr{D}$ has two blocks, 
\[
    \mathscr{D}^0 = \{(\mc{Q}_0, \Q_{\mc{Q}_0}), (\mc{Q}_\infty, \Q_{\mc{Q}_\infty}), (\mc{O}, \Q_\mc{O})\}, \text{ and }
    \mathscr{D}^1 = \{(\mc{O}, \mc{L})\}. 
\]
We can see from the formulas (\ref{eq: T_s formulas for SL2}) that the submodule of $M_{LV}$ generated by trivial local systems on closed orbits is spanned as a $\Z[q^{\pm 1}]$-module by $\mathscr{D}^0$. 
\end{example}


\subsection{The category generated by closed orbits}
\label{sec: the category generated by closed orbits}

 Let $\mc{Q}$ be a closed $K$-orbit in $X$ and $i_\mc{Q}: \mc{Q} \rightarrow X$ the inclusion map. Denote by
\begin{equation}
\label{eq: IC_Q}
\bm{IC}_\mc{Q}:=\bm{IC}(\mc{Q}, \Q_\mc{Q}) = i_{\mc{Q}!}\Q_\mc{Q} [\dim \mc{Q}]= i_{\mc{Q}*}\Q_\mc{Q} [\dim \mc{Q}] 
\end{equation}
the $\IC$ sheaf corresponding to the trivial $K$-equivariant local system on $\mc{Q}$. 

\begin{definition}
\label{def: geometric LV category}
Let $\mc{M}_{LV}^0$ be the full subcategory of $D^b_K(X)$ generated by $\{\bm{IC}_\mc{Q}\mid \mc{Q} \text{ closed}\}$ under right convolution by objects in $\mc{H}$, finite direct sums ($\oplus$), direct summands ($\ominus$), and grading shifts ($[1]$). That is, 
\[
\mc{M}_{LV}^0:=\left\langle \bm{IC}_\mc{Q} * \mc{H} \mid \mc{Q} \text{ closed} \right\rangle_{\oplus, \ominus, [1]}.
\]
\end{definition}
By construction, the category $\mc{M}^0_{LV}$ is a right $\mc{H}$-module, hence the split Grothendieck group $[\mc{M}_{LV}^0]_{\oplus}$ has the structure of a right $\bm{H}$-module. 
\begin{theorem}
\label{thm: module structure of Mgeom} 
Let $\mc{M}_{LV}^0$ be as in Definition~\ref{def: geometric LV category} and $M_{LV}^0$ as in Theorem~\ref{lem: closed orbits generate the block}. As right $\bm{H}$-modules,
\[
[\mc{M}_{LV}^0]_{\oplus}\simeq M_{LV}^0 
\]
\end{theorem}
\begin{proof}
This follows immediately from Theorem~\ref{lem: closed orbits generate the block} and Theorem~\ref{thm: semisimple categorification}. 
\end{proof} 

\begin{example}
\label{example: category generated by closed orbits SL(2,R)} $($The category $\mc{M}_{\mathrm{LV}}^0$ for $SL_2(\R))$ We return to the setup of Examples \ref{example: LV module for SL(2,R)} and \ref{example: Type I}. To compute the action of $\mc{H}$ we use Lemma~\ref{lem: push pull is convolution with ICs}. Because $\mc{Q}_0$ and $\mc{Q}_\infty$ are closed, we have $\IC(X, \Q_{\mc{Q}_0})=i_{0*}\Q_{\mc{Q}_0}$ and $\IC(X, \Q_{\mc{Q}_\infty})= i_{\infty*}\Q_{\mc{Q}_\infty}$. So by \eqref{eq: type I noncompact imaginary}, we have 
\begin{align}
    \label{eq: Hecke action on closed orbit SL2}
    \IC_{\mc{Q}_0} * \IC_s &= \pi_s^* \pi_{s*} \IC_{\mc{Q}_0}[1] \\
    \nonumber &= \Q_X[1] \\
    \nonumber &= \IC(X, \Q_\mc{O}).
\end{align}
Similarly, 
\begin{equation}
    \label{eq: Hecke action on other closed orbit SL}
    \IC_{\mc{Q}_\infty} * \IC_s = \IC(X, \Q_{\mc{O}}).
\end{equation}
Using Lemma~\ref{lemma: derived pushforward}, we compute
\begin{align}
    \label{eq: Hecke action on open orbit SL2}
    \IC(X, \Q_\mc{O}) * \IC_s &= \pi_s^* \pi_{s*}\Q_X[2] \\
    \nonumber &= \pi_s^*(\Q_\mathrm{pt} \oplus \Q_\mathrm{pt}[2]) \\
    \nonumber &= \Q_X \oplus \Q_X[2] \\
    \nonumber &= \IC(X, \Q_\mc{O})[-1] \oplus \IC(X, \Q_{\mc{O}})[1].
\end{align}
We conclude from these computations that
\[
\mc{M}_\mathrm{LV}^0 = \langle \IC_{\mc{Q}_0}, \IC_{\mc{Q}_\infty}, \IC(X, \Q_{\mc{O}})\rangle_{\oplus, [1]}.
\]
\end{example}

\begin{example}
\label{example: Hecke action on the non-trivial block SL2} $($The non-trivial block for $SL_2(\R))$ In the setting of Example~\ref{example: category generated by closed orbits SL(2,R)}, it is also interesting to compute the action of $\mc{H}$ on the $\IC$ sheaf corresponding to the M\"obius band local system $\mc{L}$. First, note that $\mc{L}$ is clean (q.v.\ Definition~\ref{def: clean}), so 
\[
\IC(X, \mc{L}) = j_{!*}\mc{L} = j_! \mc{L} = j_* \mc{L},
\]
where $j: \mc{O} \hookrightarrow X$ is inclusion. Hence by the computation of $H^*(\C^\times; \mc{L})$ in Example~\ref{example: non-parity}, we have
\begin{equation}
\label{eq: push-pull on mobius band}
\IC(X, \mc{L}) * \IC_s = \pi_s^* \pi_{s*}j_!\mc{L}[2]=0.
\end{equation}
\end{example}


\section{An algebraic categorification}
\label{sec: An algebraic categorification}

In this section, we introduce a category of bimodules which plays the role of Soergel bimodules in the $K$-equivariant setting. Before introducing our category, we recall some facts about torus fixed points in $K$-orbits which will be essential in the arguments of Section \ref{sec: the hypercohomology functor}.  


\subsection{Torus fixed points and closed $K$-orbits}
\label{sec: torus fixed points and closed K-orbits}

For the remainder of the paper, we impose an additional connectedness assumption on the group $K$. Let $G$ be as in Section \ref{sec: introduction}, and set $K$ to be the {\em identity component} of the fixed point group $G^\theta$. With this assumption, the group $K$ has the structure of a connected reductive complex algebraic group. Denote by $\mf{g}$ and $\mf{k}$ the corresponding Lie algebras. The involution $\theta$ induces an involution on $\mf{g}$ which we refer to by the same symbol. Let 
\[
\mf{g} = \mf{k} \oplus \mf{p}
\]
be the decomposition of $\mf{g}$ into eigenspaces of $\theta$. Here $\mf{k}$ is the $+1$-eigenspace and $\mf{p}$ is the $-1$-eigenspace. 

\begin{remark}
\label{rem: connectedness assumption}
The assumption that $K$ is connected is necessary for several of the arguments in Section \ref{sec: the hypercohomology functor}. We suspect that the main results of this paper remain true for disconnected $K$, but to prove them one would need to alter the present arguments. We intend to explore this further in future work. 
\end{remark}



Choose a maximal torus $T_K$ in $K$, and let $W_K:= N_K(T_K)/T_K$ be the Weyl group of $(K,T_K)$. Set
\begin{equation}
    \label{eq:T}
    T:= Z_G(T_K). 
\end{equation}
\begin{lemma}
\label{lem: centralizer is a torus}
The group $T$ is the unique $\theta$-stable maximal torus containing $T_K$.
\end{lemma}
\begin{proof}
Clearly $T$ is $\theta$-stable: for $z \in T$ and $t \in T_K$, 
\[
\theta(z)t=\theta(zt)=\theta(tz)=t\theta(z),
\]
so $\theta(z) \in T$. 

As $T_K$ is a $\theta$-fixed maximal torus in $K$, it is contained in some $\theta$-stable maximal torus $T'$ of $G$. Let $\Sigma = \Sigma(\mf{g}, \mf{t}')$ be the root system of $(\mf{g}, \mf{t}':=\Lie T')$, and 
\[
\mf{g}=\mf{t}' \oplus \bigoplus_{\alpha \in \Sigma} \mf{g}_\alpha
\]
the corresponding root space decomposition of $\mf{g}$. The Lie algebra $\mf{t}'$ decomposes into a direct sum $\mf{t}' = \mf{t}_K \oplus \mf{a}$, where $\mf{t}_K = \Lie T_K$ and $\mf{a} = \mf{t}' \cap \mf{p}$. Recall that a root $\alpha \in \Sigma$ is said to be {\em real} if $\alpha |_{\mf{t}_K} = 0$. If $X \in \mf{g}_\alpha$ for $\alpha$ real, then $[t, x]=\alpha(t)x=0$ for all $t \in \mf{t}_K$. Hence we have a decomposition 
\begin{equation}
    \label{eq: centralizer in general}
Z_\mf{g}(\mf{t}_K) = \mf{t}' \oplus \bigoplus_{\alpha \in \Sigma \atop {\alpha \text{ real}}} \mf{g}_\alpha. 
\end{equation}
Because $T_K$ is a maximal torus in $K$, the Cartan subalgebra $\mf{t}'$ is {\em maximally compact}, meaning that the compact dimension (i.e., the dimension of the $+1$ $\theta$-eigenspace) is maximal among $\theta$-stable Cartan subalgebras. The maximally compact Cartan subalgebras of $\mf{g}$ are precisely those with no real roots \cite[Prop.\ 6.70]{Knapp}, so in our case, (\ref{eq: centralizer in general}) becomes
\[
Z_\mf{g}(\mf{t}_K) = \mf{t}'. 
\]
This implies the group-theoretic statement.
\end{proof}

Let $W:=N_G(T)/T$ be the Weyl group of $(G,T)$. The following lemma shows that $W$ contains $W_K$ as a subgroup. 
\begin{lemma}
\label{lem: Weyl group embedding}
There is an embedding $W_K \hookrightarrow W$. 
\end{lemma}
\begin{proof}
To begin, we claim that 
\begin{equation}
    \label{eq: normalizers agree}
    N_K(T) = N_K(T_K).
\end{equation}
Indeed, let $k \in N_K(T)$ and $t \in T_K$. Then there exists $t' \in T$ such that $kt=t'k$. As both elements $kt$ and $k$ are fixed by $\theta$, we have
\[
t'k=\theta(t'k) = \theta(t')k,
\]
so $t'$ is also fixed by $\theta$. Hence $N_K(T) \subseteq N_K(T_K)$. 

Conversely, consider $kTk^{-1}$ for $k \in N_K(T_K)$. This is a $\theta$-stable maximal torus in $G$ containing $T_K$, so by Lemma~\ref{lem: centralizer is a torus}, $kTk^{-1}=T$. Hence $N_K(T_K) \subseteq N_K(T)$. 

Hence there is a natural injection 
\[
N_K(T_K)/T_K \xhookrightarrow{i} N_G(T)/T_K.
\]
By composing $i$ with the quotient map
\[
N_G(T)/T_K \xrightarrow{q} N_G(T)/T,
\]
we obtain a map 
\begin{equation}
    \label{eq: map between Weyl groups}
    W_K=N_K(T_K)/T_K \xrightarrow{i \circ q} N_G(T)/T=W.
\end{equation}
As $\ker q = T/T_K$ is not contained in the $\im i$, the map (\ref{eq: map between Weyl groups}) is injective.  
\end{proof}

Because the torus $T$ is $\theta$-stable, $\theta$ induces an involution on $W$. Let $W^\theta$ be the group of elements fixed by this involution. After identifying $W_K$ with its image under the injection in Lemma~\ref{lem: Weyl group embedding}, we have  
\begin{equation}
    \label{eq: all the W's}
    W_K \subseteq W^\theta \subseteq W.
\end{equation}

Choose a Borel subgroup $B \subseteq G$ containing $T$, and set $B_K= B \cap K$. The group $B_K$ is a Borel subgroup in $K$ containing $T_K$. Note that $B$ is $\theta$-stable by construction. Let $S \subseteq W$ be the set of simple reflections determined by the choice of $B$ and $\Sigma^+ \subseteq \Sigma$ be the corresponding set of positive roots. 

It is well-known that the $T$-fixed points in the flag variety $X=G/B$ are in bijection with $W$, and each Bruhat cell contains exactly one $T$-fixed point. Specifically, we have \begin{equation}
    \label{eq: T-fixed points}
    X^T=\{x_w:=wB/B\}_{w \in W}, \text{ and } x_w \in X_w = B \cdot wB/B. 
\end{equation}

\begin{lemma}
\label{lem: agreement of fixed points}
All $T_K$-fixed points in $X$ are fixed by $T$. In particular, $X^{T_K}=X^T$. 
\end{lemma}
\begin{proof}
Each $T$-fixed point $x_w$ lies in an affine chart $U_w \simeq \mathbb{A}^{|\Sigma^+|}$ on which $T$ acts linearly by roots. In particular, we can identify
\[
x_w=0 \in U_w = \bigoplus_{\alpha \in -w(\Sigma^+)} \C_{\alpha},
\]
where $T$ acts on $\C_\alpha$ by $t \cdot z = \exp \alpha(t)z$ for $t \in T$, $z \in \C_\alpha$. The charts $\{U_w\}_{w \in W}$ cover $X$. Clearly $x_w$ is the only $T$-fixed point in $U_w$.

By restriction we obtain a linear action of $T_K$ on $U_w$, with weights \[
\{\widetilde{\alpha}:= \alpha|_{\mf{t}_K}\}_{\alpha \in -w(\Sigma^+)}.
\]
The point $x_w \in U_w$ is fixed by $T_K$. Moreover, $x_w$ is the only $T_K$-fixed point in $U_w$. To see this, recall from the proof of Lemma~\ref{lem: centralizer is a torus} that $\mf{t}=\Lie T$ is maximally compact, so there are no real roots in $\Sigma$ \cite[Prop.\ 6.70]{Knapp}. Hence the restriction map $\mf{t}^* \rightarrow \mf{t}_K^*$ is non-zero on all roots $\alpha \in \Sigma$, and thus the weights of the $T_K$ action on $U_w$ are non-zero for all $w \in W$. 

As the charts $U_w$ cover $X$ and each chart contains a single $T_K$-fixed point which is also fixed by $T$, the lemma follows. 
\end{proof}

\begin{lemma}
\label{lem: fixed points in closed orbits} Each closed $K$-orbit $\mc{Q}$ contains $|W_K|$ $T$-fixed points. Moreover, if $x_w \in \mc{Q}$, then $x_{vw} \in \mc{Q}$ for all $v \in W_K$, and all representatives of the coset $W_Kw \in W_K \backslash W$ are in $W^\theta$.  
\end{lemma}
\begin{proof}
Fix a closed $K$-orbit $\mc{Q}$. For any $q \in \mc{Q}$, the stabilizer $K_q:= \stab_K(q)$ is solvable (as it is the intersection of the solvable group $G_q$ with $K$), parabolic (as $\mc{Q} \simeq K/K_q$ is a closed subvariety of a projective variety), and connected (as it is a parabolic subgroup of a connected group), hence it is a Borel subgroup of $K$. This implies that $\mc{Q}$ is isomorphic to the flag variety of $K$, so it contains $|W_K|$ $T_K$-fixed points. By Lemma~\ref{lem: agreement of fixed points}, all of these must also be fixed by $T$.

There is a natural action of $W$ on $X^T$ coming from the conjugation action of $N_G(T)$. For $v \in W$, this action sends $x_w \mapsto x_{vw}$. Similarly, as $\mc{Q}$ is isomorphic to the flag variety of $K$, there is an action of $W_K$ on $\mc{Q}^{T_K}=\mc{Q}^T$ coming from the conjugation action of $N_K(T_K)$. Under the identification of $W_K$ with a subgroup of $W$ (Lemma~\ref{lem: Weyl group embedding}), the action of $W_K$ on $\mc{Q}^T$ is just the restriction of the $W$-action to $W_K$. This proves the first half of the second statement. 

To see that the $T$-fixed points in closed orbits correspond to elements in $W^\theta$, we shift perspectives and identify $X$ with the variety of Borel subalgebras of $\mf{g}$. Under this identification, the $T$-fixed point $x_w$ corresponds to the Borel subalgebra 
\[
\mf{b}^w := \mf{t} \oplus \bigoplus_{\alpha \in \Sigma^+} \mf{g}_{w \alpha}. 
\]
The involution $\theta: G\rightarrow G$ induces involutions on $\mf{g}$, $W$, and $\Sigma$. In what follows, we use the same symbol $\theta$ for each of these involutions. 

The closed $K$-orbits are precisely those consisting of $\theta$-stable Borel subalgebras \cite[Lem.\ 5.8]{penrose}, so the Weyl group elements $w \in W$ corresponding to fixed points in closed orbits are exactly those $w$ such that $\mf{b}^w$ is $\theta$-stable. We have 
\begin{align*}
\theta(\mf{b}^w) &= \theta(\mf{t}) \oplus \bigoplus_{\alpha \in \Sigma^+} \theta(\mf{g}_{w\alpha}) \\
&= \mf{t} \oplus  \bigoplus_{\alpha \in \Sigma^+} \mf{g}_{\theta(w) \theta(\alpha)} \\
&= \mf{t} \oplus  \bigoplus_{\alpha \in \Sigma^+} \mf{g}_{\theta(w) \alpha}.\\
\end{align*}
Here the third equality follows from the fact that $\theta(\Sigma^+)=\Sigma^+$ because $\Sigma$ has no real roots \cite[Ch. VI \S 8]{Knapp}. From this we see that $\theta(\mf{b}^w) = \mf{b}^w$ if and only if $\theta(w)=w$, so all $T$-fixed points in closed orbits correspond to $w \in W^\theta$. This completes the proof of the second statement. 
\end{proof}

Lemma~\ref{lem: fixed points in closed orbits} gives a useful combinatorial description of the closed $K$-orbits in $X$. 
\begin{corollary}
\label{cor: closed Q orbit bijection}
The closed $K$-orbits are in bijection with $W_K \backslash W^\theta$. 
\end{corollary}

\begin{example}
\label{example: torus fixed points for SL(2,R)} $($Torus fixed points for $SL_2(\R))$ We return to our running example (Examples~\ref{example: LV module for SL(2,R)}, \ref{example: Type I}, \ref{example: non-parity}, \ref{example: blocks for SL(2,R)}, \ref{example: category generated by closed orbits SL(2,R)}, \ref{example: Hecke action on the non-trivial block SL2}). In this case, $T_K=T=K$, so $T$-fixed points are the closed orbits $\mc{Q}_0$ and $\mc{Q}_\infty$. We have $W_K = \{1\} \subseteq W=S_2 = W^\theta$. Clearly $W_K \backslash W^\theta$ is in bijection with the set of closed orbits. 
\end{example}


\subsection{A category of bimodules}
\label{sec: a category of bimodules}

We now describe the category which provides an algebraic categorification of the trivial block of the Lusztig--Vogan module. We begin by introducing two rings.

As above, denote by $X(T)_\Q:= X(T) \otimes_\Z \Q$ and $X(T_K)_\Q: = X(T_K) \otimes_\Z \Q$ the $\Q$-span of the character lattices of $T$ and $T_K$. Let $P:=S(X(T_K)_\Q)$ be the symmetric algebra of $X(T_K)_\Q$, graded so that $X(T_K)_\Q$ has degree $2$, and denote by 
\begin{equation}
    \label{eq: P^K}
    P^K:=P^{W_K}=\{ p \in P \mid wp=p \text{ for all } w \in W_K\}.
\end{equation}
As in Section \ref{sec: Soergel bimodules}, set $R=S(X(T)_\Q)$. Let
\begin{equation}
    \label{eq: phi}
    \phi: R \rightarrow P
\end{equation}
be the natural algebra homomorphism extending the restriction map 
\[
X(T) \rightarrow X(T_K).
\] 


Let $(P^K, R)\mathrm{-gbim}$ be the category of finitely generated (as both left $P^K$- and right $R$-modules) graded $(P^K, R)$-bimodules, with morphisms given by graded $(P^K, R)$-bimodule maps which are homogeneous of degree $0$. Denote by $(n):(P^K, R)\mathrm{-gbim}\rightarrow (P^K, R)\mathrm{-gbim}$ the natural shift functor, as defined in Section \ref{sec: Soergel bimodules}. There is a right action of $\SBim$ on $(P^K, R)\mathrm{-gbim}$ given by the tensor product: for $M \in (P^K, R)\mathrm{-gbim}$ and $S \in \SBim$, 
\begin{equation}
    \label{eq: Hecke action} 
    M \otimes_R S \in (P^K, R)\mathrm{-gbim}
\end{equation}

To define our category, we imitate the geometric construction in Section \ref{sec: the category generated by closed orbits}. First we need a class of generating objects to play the role of trivial local systems on closed orbits. 

\begin{definition}
\label{def: standard modules}
For every $x \in W$, define the corresponding {\em standard bimodule} $P_x\in (P^K,R)$-gbim as follows. As a vector space, $P_x=P$. The left action of $f \in P^K$ is given by left multiplication in $P$: 
\[
f \cdot p := fp \text{ for $p \in P_x$}.
\]
The right action of $g \in R$ is given by right multiplication by $\phi(xg)$, where $\phi$ is the algebra homomorphism (\ref{eq: phi}): 
\[
p \cdot_x g := p\phi(xg) \text{ for $p \in P_x$}.
\]
\end{definition}
\begin{notation}
\label{not: W action}
 We write the action of $W$ on $R$ (or $W_K$ on $P$) with no symbol. It is implied from this notation that an element $x \in W$ acts only on the element immediately following it. If we wish for $x \in W$ to act on a product of elements, we use parentheses. For example, $wrs$ denotes the product in $R$ of the elements $wr$ and $s$, whereas $w(rs)$ denotes the image of the element $rs$ under action by $w$. Note that because the action of $W$ on $R=S(\mf{t}^*)$ (and that of $W_K$ on $P=S(\mf{t}^*_K)$) is obtained by extending the natural reflection action of $W$ on $\mf{t}^*$ (resp.\ $W_K$ on $\mf{t}_K^*$) multiplicatively, we have $w(rs)=wrws$.
\end{notation}
\begin{lemma}
\label{lem: standards are parameterized by cosets}
If $w \in W_K$, 
\[
P_{wx} \simeq P_x
\]
as $(P^K, R)$-bimodules.
\end{lemma}
\begin{proof}
Let $\varphi:P_{wx}\rightarrow P_x$ be the linear map sending $p \mapsto w^{-1}p$. We claim that $\varphi$ is an isomorphism of $(P^K, R)$-bimodules. 

Clearly $\varphi$ is a bijection of vector spaces. Let $f \in P^K$, $p \in P_{wx}$. The following computation shows that $\varphi$ is a left $P^K$-module morphism:
\[
f \cdot \varphi(p) = f \cdot w^{-1}p=fw^{-1}p=w^{-1}fw^{-1}p=w^{-1}(fp)=\varphi(fp)=\varphi(f \cdot p). 
\]
Let $g \in R$, $p \in P_{wx}$. The following computation shows that $\varphi$ is a morphism of right $R$-modules:
\begin{align*}
\varphi(p) \cdot_x g = w^{-1}p \cdot_x g = w^{-1}p\phi(xg) &= w^{-1}(pw\phi(xg)) \\
&= \varphi(pw\phi(xg)) = \varphi(p\phi(wxg))= \varphi(p \cdot_{wx} g).
\end{align*}
This completes the proof.
\end{proof}

Recall from Section \ref{sec: torus fixed points and closed K-orbits} that $W_K$ is contained in $W^\theta$, the subgroup of $\theta$-fixed elements of $W$ (\ref{eq: all the W's}). Choose a set $^KW^\theta$ of coset representatives for $W_K\backslash W^\theta$. We have a set of distinguished objects 
\[
\{P_{x} \mid x \in {^KW^\theta} \} \subseteq (P^K, R)\mathrm{-gbim}
\]
parameterized by cosets $W_K \backslash W^\theta$. By Lemma~\ref{lem: standards are parameterized by cosets} and Lemma~\ref{lem: fixed points in closed orbits}, for every $w \in W$ such that the corresponding $T$-fixed point $x_w$ is in a closed $K$-orbit, $P_w$ is isomorphic to an element of this set.

\begin{remark}
If $G_\R$ is equal rank, then the subgroup $W_K$ is a reflection subgroup of $W$ and $W^\theta=W$. In this case, every right coset $\overline{x} \in W_K \backslash W$ has a unique minimal-length representative \cite[Cor.\ 3.4]{Dyer90}. In contrast, for general $K$, it is not the case that each right $W_K$-coset in $W^\theta$ has a minimal length representative. 
\end{remark}

\begin{definition}
\label{def: LV category}
Let $\mc{N}_{LV}^0$ be the full subcategory of $(P^K,R)$-gbim generated by $\{P_x \mid x \in {^KW^\theta}\}$ under action (\ref{eq: Hecke action}) by objects in $\SBim$, finite direct sums ($\oplus$), direct summands ($\ominus$), and grading shifts ($(1)$). That is,
\[
\mc{N}_{LV}^0:= \left\langle P_x \otimes_R \SBim \mid x \in {^KW^\theta} \right\rangle_{\oplus, \ominus, (1)}.
\]
\end{definition}
By construction, $\mc{N}_{LV}^0$ is a right module over the monoidal category $\SBim$, in the sense of \cite[Ch. 7]{TensorCategoriesBook}. Let $[\mc{N}_{LV}^0]_\oplus$ be the split Grothendieck group of $\mc{N}_{LV}^0$. Then by Theorem~\ref{thm: Soergel's categorification theorem}, $[\mc{N}_{LV}^0]_\oplus$ has the structure of a right $\bm{H}$-module. 

\begin{example}
\label{example: bimodule category for SL(2,R)} $($The category $\mc{N}_{\mathrm{LV}}^0$ for $SL_2(\R))$ We return to our running example (Examples~\ref{example: LV module for SL(2,R)}, \ref{example: Type I}, \ref{example: non-parity}, \ref{example: blocks for SL(2,R)}, \ref{example: category generated by closed orbits SL(2,R)}, \ref{example: Hecke action on the non-trivial block SL2}, \ref{example: torus fixed points for SL(2,R)}). As $T_K=T$ and $W_K$ is trivial, 
\[
P^K=P=R=\Q[\alpha]
\]
is a polynomial ring in one variable, graded so that the simple root $\alpha \in \Sigma$ is in degree $2$. It follows that the objects in $\mc{N}_{\mathrm{LV}}^0$ are graded $(R, R)$-bimodules. The standard $(R, R)$-bimodules of Definition~\ref{def: standard modules} are
\begin{align*}
    R_{id} \hspace{3mm} &(R \text{ as a }(R,R)\text{-bimodule under the standard action}), \text{ and }\\
    R_{s} \hspace{4mm} &(R \text{ as a }(R,R)\text{-bimodule with right action twisted by $s$}).
\end{align*}
We compute the action of $\SBim = \langle R, B_s \rangle_{\oplus, \ominus, (-)}$ on these generators. 
\begin{gather}
    \label{eq: SBim action for SL2}
    R_{id} \otimes_R B_s = R \otimes_R R \otimes_{R^s}R(1) = R \otimes_{R^s}R(1) = B_s \\
\label{eq: second SBim action for SL2}
  R_s \otimes_R B_s = R \otimes_R R \otimes_{R^s} R (1) = R_s \otimes_{R^s} R(1) =R \otimes_{R^s} R(1) = B_s
\end{gather}
As the square of $B_s$ is $B_s(1) \oplus B_s(-1)$ (equation (\ref{eq: square of Bs})), we see that 
\[
\mc{N}_{\mathrm{LV}}^0 = \langle R_{\id}, R_s, B_s \rangle_{\oplus, \ominus, (-)}.
\]
(Compare the bimodule actions (\ref{eq: SBim action for SL2}), (\ref{eq: second SBim action for SL2}), (\ref{eq: square of Bs}) to the convolution actions (\ref{eq: Hecke action on closed orbit SL2}), (\ref{eq: Hecke action on other closed orbit SL}), (\ref{eq: Hecke action on open orbit SL2}).)
\end{example}


\section{The hypercohomology functor}
\label{sec: the hypercohomology functor}

The bridge between the geometric category $\mc{M}_{LV}^0$ and the algebraic category $\mc{N}_{LV}^0$ is the $K$-equivariant hypercohomology functor. As in (\ref{eq: equivariant hypercohomology}), we have a functor
\[
    \HH_K^\bullet: D^b_K(X) \rightarrow H^*_K(X)\mathrm{-gmod},
\]
where $H^*_K(X)\mathrm{-gmod}$ is the category of graded modules over the graded ring $H^*_K(X)$. 
\begin{lemma}
\label{lem: K equivariant cohomology of X}
There is an isomorphism of rings
\[
H^*_K(X) \simeq P^K \otimes_{R^W} R,
\]
where $P^K$ is a $R^W$-module via the map $\phi$ in (\ref{eq: phi}).
\end{lemma}
\begin{proof}
We use standard facts about equivariant cohomology, proofs of which can be found in \cite{Brion} in the case of a compact group. After a straightforward reduction from complex reductive algebraic groups to compact groups, these facts also hold in our setting. See also \cite[Prop 1.2.1]{Wyser}.

Using three applications of the induction isomorphism in equivariant cohomology \cite[\S 1, Rmk.\ 3]{Brion} along with \cite[Prop.\ 1]{Brion}, we compute:
\begin{align*}
    H^*_K(G/B)&=H^*_G(G/B \times G/K) \\
    &= H^*_G(G \times_B G/K) \\
    &= H^*_B(G/K) \\
    &= H^*_T(G/K)\\
    &=H^*_G(G/K) \otimes_{R^W} R\\
    &= H^*_K(\mathrm{pt}) \otimes_{R^W} R\\
    &=(H^*_{T_K}(\mathrm{pt}))^{W_K} \otimes_{R^W} R \\
    &=P^K \otimes_{R^W} R.
\end{align*}
\end{proof}
As in Section \ref{sec: Relationship between the geometric Hecke category and Soergel bimodules}, the quotient map $P^K \otimes R \rightarrow P^K \otimes_{R^W} R$ induces a fully faithful functor
\begin{equation}
\label{eq: move to bimodules}
P^K \otimes_{R^W} R \mathrm{-gmod} \rightarrow (P^K, R)\mathrm{-gbim},
\end{equation}
and we identify $P^K \otimes_{R^W} R\mathrm{-gmod}$ with its image in $(P^K, R)\mathrm{-gbim}$. Hence we have a functor 
\begin{equation}
    \label{eq: K-equivariant hypercohomology}
\HH^\bullet_K: D^b_K(X)\rightarrow (P^K, R)\mathrm{-gbim}.
\end{equation}


\subsection{The essential image of the hypercohomology functor}
\label{sec: the essential image of hypercohomology}
\begin{theorem}
\label{thm: image of hypercohomology}
The essential image of the functor  
\[
\HH^\bullet_K:\mc{M}_{LV}^0 \rightarrow (P^K, R)\mathrm{-gbim}
\]
defined in (\ref{eq: K-equivariant hypercohomology}) is $\mc{N}_{LV}^0$. 
\end{theorem}

By construction, $\mc{N}_{LV}^0$ is generated by objects of the form
\[
P_x \otimes_R B_{s_1} \otimes_R B_{s_2} \otimes \cdots \otimes_R B_{s_n}
\]
for $x \in {^KW^\theta}$ and $s_i \in S$. By Lemma~\ref{lem: push pull is convolution with ICs}, $\mc{M}_{LV}^0$ is generated by objects of the form
\[
\bm{IC}_\mc{Q} \Theta_{s_1} \Theta_{s_2} \cdots \Theta_{s_n}
\]
for $\mc{Q}$ a closed $K$-orbit and $s_i \in S$. Hence to prove Theorem~\ref{thm: image of hypercohomology} it suffices to prove the following two lemmas. 
\begin{lemma}
\label{lem: standards are equiv cohomology}
For any closed $K$-orbit $\mc{Q} \subseteq X$,
\[
\HH^\bullet_K(\bm{IC}_\mc{Q}) = P_x
\]
for some $x\in {^KW^\theta}$, where $P_x$ is as in Definition~\ref{def: standard modules}. Moreover, for every $x \in {^KW^\theta}$, there exists a closed $K$-orbit $\mc{Q} \subseteq X$ such that $\HH_{K}(\bm{IC}_\mc{Q}) = P_x$. 
\end{lemma}
\begin{lemma}
\label{lem: push pull is tensoring with Bs}
 For every $s \in S$ and $\mc{F} \in \mc{M}_{LV}^0$, 
\[
\HH^\bullet_K(\mc{F}\Theta_s) = \HH^\bullet_K(\mc{F}) \otimes_R B_s. 
\]
\end{lemma}

\begin{example}
\label{example: Lemmas 7.3 and 7.4 for SL2} $($Lemmas \ref{lem: standards are equiv cohomology} and \ref{lem: push pull is tensoring with Bs} for $SL_2(\R))$ In our running example (Examples~\ref{example: LV module for SL(2,R)}, \ref{example: Type I}, \ref{example: blocks for SL(2,R)}, \ref{example: category generated by closed orbits SL(2,R)}, \ref{example: bimodule category for SL(2,R)}), the closed $K$-orbits are precisely the $T$-fixed points in $X$: in the notation of (\ref{eq: T-fixed points}), $\mc{Q}_0 = x_{id}$, and $\mc{Q}_\infty = x_s$. Hence the $K$-equivariant hypercohomology of $\IC_{\mc{Q}_0}$ and $\IC_{\mc{Q}_\infty}$ are given by the $T$-equivariant cohomology of the corresponding $T$-fixed point: 
\begin{align*}
\HH_K(\IC_{\mc{Q}_0}) &= \Hom^\bullet_{T}(\Q_{x_{id}}, \Q_{x_{id}}) = H_T^*(x_{id}), \text{ and }\\
\HH_K(\IC_{\mc{Q}_\infty}) &= \Hom^\bullet_{T}(\Q_{x_{s}}, \Q_{x_{s}}) = H_T^*(x_{s}).
\end{align*}
For any $T$-fixed point $x_w$, $H^*_T(x_w)\simeq R$ as left $R$-modules via the Borel homomorphism. The right $R$-module structure on $H^*_T(x_w)$ is twisted by the parameter $w \in W$ specifying the $T$-fixed point. (This can be seen by comparing first Chern classes of the line bundles 
\[
ET \times_T \C_\chi^\times \rightarrow ET \times_T x_w \hspace{1mm} \text{ and } \hspace{1mm} ET \times_T wB \times_B \C_\chi^\times \rightarrow ET \times_T x_w
\]
in $H^*(x_w)$.) 

With this computation, Lemma~\ref{lem: push pull is tensoring with Bs} can be seen for $SL_2(\R)$ by comparing the convolution actions (\ref{eq: Hecke action on closed orbit SL2}), (\ref{eq: Hecke action on other closed orbit SL}), and (\ref{eq: Hecke action on open orbit SL2}) with the bimodule actions (\ref{eq: SBim action for SL2}), (\ref{eq: second SBim action for SL2}) and (\ref{eq: square of Bs}).
\end{example}

\noindent
\begin{proof}[Proof of Lemma~\ref{lem: standards are equiv cohomology}]
Fix a closed $K$-orbit $\mc{Q}$. By Lemma~\ref{lem: fixed points in closed orbits}, there exists $w \in {^KW^\theta}$ such that the $T$-fixed point $x_w=wB/B$ lies in $\mc{Q}$. Denote by $K_w:=\stab_{K}x_w$. Recall that $K_w$ is a Borel subgroup in $K$ containing $T_K$ (see proof of Lemma~\ref{lem: fixed points in closed orbits}) and $\mc{Q} \simeq K/K_w \simeq K \times_{K_w} x_w$. We have the following isomorphisms of vector spaces: 
\begin{equation}
\label{eq: v.s. isos}
\mathbb{H}_K^\bullet (\bm{IC}_\mc{Q}) \xrightarrow{\sim} H^*_K(\mc{Q}) \xrightarrow{\sim} H^*_{K_w}(x_w) \xrightarrow{\sim} H^*_{T_K}(x_w).
\end{equation}
The first is given by the adjunction $(i_\mc{Q}^*, i_{\mc{Q}*})$, the second is the induction isomorphism in equivariant cohomology \cite[\S1, Rmk.\ 3]{Brion}, and the third follows from the fact that $K_w/T_K$ is affine. 

Recall that the $T_K$-equivariant cohomology of a point is isomorphic to $P$. Explicitly, the isomorphism 
\begin{equation}
\label{eq: Borel iso}
P \xrightarrow{\sim} H^*_{T_K}(x_w)
\end{equation}
is given by the {\em Borel homomorphism}, which is constructed as follows. Recall that the ring $P$ is generated by $X(T_K)_\Q$. From a character $\chi \in X(T_K)$, we construct a principal $\C^\times$-bundle
\begin{equation}
\label{eq: pi_chi}
ET_K \times_{T_K} \C^\times_\chi \xrightarrow{\pi_\chi} ET_K \times_{T_K} x_w \simeq BT_K,
\end{equation}
where $\C_\chi^\times$ is the $T_K$-space given by the action $t \cdot z = \chi(t)z$ for $t \in T_K$ and $z \in \C^\times$. Because $\pi_\chi$ is a $\C^\times$-bundle, it has a first Chern class $c_1(\pi_\chi) \in H^2(BT_K) = H^2_{T_K}(x_w)$. The homomorphism (\ref{eq: Borel iso}) is defined by extending the assignment $\chi \mapsto c_1(\pi_\chi)$ multiplicatively. 

This establishes that $\mathbb{H}^\bullet_K(\bm{IC}_\mc{Q}) \simeq P$ as graded vector spaces. Moreover, all isomorphisms in (\ref{eq: v.s. isos}) are clearly compatible with the left $H_K^*(\mathrm{pt})$-module structures. As $H_K^*(\mathrm{pt}) \simeq P^K$ \cite[Prop.\ 1(i)]{Brion}, we see that $\mathbb{H}^\bullet_K(\bm{IC}_\mc{Q}) \simeq P$ as left $P^K$-modules.

Each space in (\ref{eq: v.s. isos}) also has a right $P$-action coming from the Borel homomorphism. Specifically, from $\chi \in X(T_K)$, we construct four principal $\C^\times$-bundles:
\[
\begin{tikzcd}
{\scriptstyle EK \times_{T_K} wB \times_B \C^\times_\chi } \arrow[d, "\mu_1^\chi"] \\ {\scriptstyle EK \times_{T_K} wB \times_B \mathrm{pt} }
\end{tikzcd},
\begin{tikzcd}
{\scriptstyle EK \times_{K_w} wB \times_B \C^\times_\chi } \arrow[d, "\mu_2^\chi"] \\ {\scriptstyle EK \times_{K_w} wB \times_B \mathrm{pt} }
\end{tikzcd},
\begin{tikzcd}
{\scriptstyle EK \times_K KwB \times_B \C^\times_\chi }\arrow[d, "\mu_3^\chi"] \\ {\scriptstyle EK \times_K KwB \times_B \mathrm{pt} }
\end{tikzcd}, \mathrm{ \small and}
\begin{tikzcd}
{\scriptstyle EK \times_K G \times_B \C^\times_\chi } \arrow[d, "\mu_4^\chi"] \\ {\scriptstyle EK \times_K G \times_B \mathrm{pt} }
\end{tikzcd},
\]
where the $B$-action on $\C^\times_\chi$ is defined by trivially extending the $T_K$-action. By taking Chern classes, we obtain ring homomorphisms $\varphi_i: \chi \mapsto c_1(\mu_i^\chi)$:
\[
\begin{tikzcd} [column sep=small]
 & & R \arrow[dll, "\varphi_1"'] \arrow[dl, "\varphi_2"] \arrow[dr, "\varphi_3"'] \arrow[drr, "\varphi_4"] & &\\
H^*_{T_K}(x_w) & H^*_{K_w}(x_w) & & H^*_K(\mc{Q}) & H^*_K(X)
\end{tikzcd}
\]
The ring homomorphisms $\varphi_i$ give each space in (\ref{eq: v.s. isos}) a right $P$-module structure. Moreover, because the $\C^\times$-bundles $\mu_i^\chi$ fit into the diagram
\[
\begin{tikzcd}
EK \times_{T_K} wB \times_B \C^\times_\chi \arrow[r, "\mu_1^\chi"] \arrow[d, twoheadrightarrow] & EK \times_{T_K} wB \times_B \mathrm{pt} \arrow[d, twoheadrightarrow] \\
EK \times_{K_w} wB \times_B \C^\times_\chi \arrow[r, "\mu_2^\chi"] \arrow[d, hookrightarrow] & EK \times_{K_w} wB \times_B \mathrm{pt} \arrow[d, hookrightarrow] \\
EK \times_{K_w} KwB \times_B \C^\times_\chi \arrow[r] \arrow[d, twoheadrightarrow] & EK \times_{K_w} KwB \times_B \mathrm{pt} \arrow[d, twoheadrightarrow] \\
EK \times_K KwB \times_B \C^\times_\chi \arrow[r, "\mu_3^\chi"] \arrow[d, hookrightarrow] & EK \times_K KwB \times_B \mathrm{pt} \arrow[d, hookrightarrow] \\
EK \times_K G \times_B \C^\times_\chi \arrow[r, "\mu_4^\chi"]  & EK \times_K G \times_B \mathrm{pt}
\end{tikzcd}
\]
in which every square is Cartesian, the isomorphisms in (\ref{eq: v.s. isos}) are compatible with these right $P$-actions. By precomposing with the map $\phi:R \rightarrow P$ from (\ref{eq: phi}), each of the vector spaces in (\ref{eq: v.s. isos}) obtains a compatible right $R$-module structure. 

It remains to compute the right $R$-action on $H^*_{T_K}(x_w)$ obtained from $\varphi_1 \circ \phi$. First note that the natural map 
\[
H^*_T(x_w) \rightarrow H^*_{T_K}(x_w)
\]
obtained by post- and pre-composing $\phi$ with the Borel isomorphisms $R \xleftarrow{\sim} H^*_T(x_w)$ and $P \xrightarrow{\sim} H^*_{T_K}(x_w)$ sends 
\[
c_1(\pi^T_\lambda) \mapsto c_1(\pi_{\phi(\lambda)})
\]
for $\lambda \in X(T)$, where $\pi_\lambda^T$ is the $\C^\times$-bundle $ET \times_T \C_\lambda^\times \rightarrow ET \times_T x_w$ and $\pi_{\phi(\lambda)}$ is as in (\ref{eq: pi_chi}). Moreover, if $\mu_T^\lambda$ is the $\C^\times$-bundle 
\[
ET \times_T wB \times_B \C^\times_\lambda \xrightarrow{\mu_T^\lambda} ET \times_T wB \times_B \mathrm{pt},
\]
then $\mu_T^\lambda \simeq \pi_{w \lambda}^T$ as $\C^\times$-bundles, so in $H^*_T(x_w)$, we have $c_1(\mu_T^\lambda) = wc_1(\pi_\lambda^T)$. 

Because the square 
\[
\begin{tikzcd}
ET \times_{T_K} wB \times \C^\times_{\phi(\lambda)} \arrow[r, "\mu^{\phi(\lambda)}"] \arrow[d, twoheadrightarrow] & ET \times_{T_K} wB \times_B \mathrm{pt} \arrow[d, twoheadrightarrow] \\
ET \times_T wB \times_B \C^\times_\lambda \arrow[r, "\mu_T ^\lambda"] & ET \times_T wB \times_B \mathrm{pt}
\end{tikzcd}
\]
is Cartesian, we have 
\[
c_1(\mu^{\phi(\lambda)}) = \phi(c_1(\mu_T^\lambda))= \phi(w c_1(\pi_\lambda^T))
\]
in $H^*_{T_K}(x_w)$. Hence the right action of $g \in R$ on $p \in P \simeq H^*_{T_K}(x_w)$ is given by 
\[
p \cdot g = p \phi(wg) = p \cdot_{w} g. 
\]
This completes the proof of the first statement. The second statement follows from the fact that every $T$-fixed point corresponding to $w \in {W^\theta}$ lies in some closed orbit $\mc{Q}$ (q.v.\ proof of Lemma~\ref{lem: fixed points in closed orbits}).  
\end{proof}

\vspace{3mm}

\begin{proof}[Proof of Lemma~\ref{lem: push pull is tensoring with Bs}]
Set $Y:= G/P_s$ and denote by $C_K:=H^*_K(X) = P^K \otimes_{R^W}R$ (Lemma~\ref{lem: K equivariant cohomology of X}). The map $\pi_s^*:\End^\bullet_K(\Q_Y) \rightarrow \End^\bullet_K(\Q_X)=C_K$ is injective with image $C_K^s= P^K \otimes_{R^W}R^s$. We identify $H^*_K(Y)=\End_K^\bullet(\Q_Y)$ with its image under $\pi_s^*$. Consider the diagram of functors
\[
\begin{tikzcd}
D^b_K(X) \arrow[r, "\mathbb{H}^\bullet_K"] \arrow[d, "\pi_{s*}", shift left] & C_K \mathrm{-gmod} \arrow[d, "\mathrm{Res}^s"] \\
D^b_K(Y) \arrow[u, "\pi_s^*", shift left] \arrow[r, "\mathbb{H}^s_K"] & C^s_K\mathrm{-gmod} 
\end{tikzcd}
\]
where $\mathbb{H}^s_K(-):= \Hom^\bullet_{D^b_K(Y)}(\Q_{Y}, -)$ is the hypercohomology functor on $D^b_K(Y)$ and $\mathrm{Res}^s$ is the functor given by restriction to $C^s_K \subseteq C_K$. The adjunction $(\pi_s^*, \pi_{s*})$ gives a natural isomorphism 
\begin{equation}
\label{eq: pushing is restriction}
\mathbb{H}_K^s \pi_{s*} \simeq \mathrm{Res}^s \mathbb{H}_K.
\end{equation}
Lemma~\ref{lem: push pull is tensoring with Bs} follows from the following proposition.

\begin{proposition}
\label{prop: push pull version one}
For $\mc{F} \in D^b_K(X)$ and $s \in S$, 
\[
\mathbb{H}_K^\bullet(\pi_s^* \pi_{s*} \mc{F}) \simeq C_K \otimes_{C_K^s} \mathrm{Res}^s\mathbb{H}^\bullet_K(\mc{F})
\]
in $C_K\mathrm{-gmod}$. 
\end{proposition}
Let $M \in C_K\mathrm{-gmod}$. Under the identification of $C_K\mathrm{-gmod}$ with its image in $(P^K, R)\mathrm{-gbim}$ (see \ref{eq: move to bimodules}), the left $C_K$-module $C_K \otimes_{C_K^s} \mathrm{Res}^sM$ corresponds to the $(P^K, R)$-bimodule $M \otimes_R B_s$, so Lemma~\ref{lem: push pull is tensoring with Bs} immediately follows from Proposition~\ref{prop: push pull version one}. 

\begin{proof}[Proof of Proposition~\ref{prop: push pull version one}] 
Our argument is exactly the argument in \cite[\S 3.2]{Soergel90} with added equivariance. We repeat the argument here for the benefit of readers less comfortable with the German language. 

For any $\mc{F} \in D^b_K(X)$, there is a canonical morphism of $C_K$-modules
\begin{align}
\label{eq: canonical morphism}
C_K \otimes_{C_K^s} \mathbb{H}_K^s(\pi_{s*}\mc{F}) &= \Hom_K^\bullet (\Q_X, \pi_s^*\Q_Y) \otimes_{\End^\bullet_K(\Q_Y)} \Hom_K^\bullet(\Q_Y, \pi_{s*} \mc{F})  \\
& \rightarrow \Hom^\bullet(\Q_X, \pi_s^*\pi_{s*}\mc{F}) = \mathbb{H}_K^\bullet(\pi_s^* \pi_{s*} \mc{F}) \nonumber
\end{align}
given by $\varphi \otimes \psi \mapsto \pi_s^* \psi \circ \varphi$. We will show that this canonical morphism is an isomorphism. 

Because $\pi_s$ is smooth of relative dimension one, we have $\pi_s^* = \pi_s^![-2]$. Using this relation, along with the adjunction $(\pi_{s!}, \pi_s^!)$ and the fact that $\pi_s$ is proper (hence $\pi_{s!} = \pi_{s*}$), we compute
\begin{equation}
\label{eq: first steps}
\mathbb{H}_K^\bullet(\pi_s^* \pi_{s*} \mc{F}) = \Hom^\bullet_K(\Q_X, \pi_s^! \pi_{s*} \mc{F}[-2]) = \Hom^\bullet_K(\pi_{s*}\Q_X[2], \pi_{s*} \mc{F}). 
\end{equation}
We have canonical morphisms
\[
\Q_Y \rightarrow \pi_{s*}\pi_s^* \Q_Y = \pi_{s*}\Q_X = \pi_{s!} \pi_s^! \Q_Y[-2] \rightarrow \Q_Y[-2]
\]
given by the unit of the adjunction $(\pi_s^*, \pi_{s*})$ and the counit of the adjunction $(\pi_{s!}, \pi_s^!)$. Shifted, this is
\[
\Q_Y[2] \rightarrow \pi_{s*} \Q_X[2] \rightarrow \Q_Y. 
\]
Because $Y$ has no negative cohomology groups, the composition is zero. By the decomposition theorem, the sequence splits, and we conclude that 
\begin{equation}
\label{eq: splitting of structure sheaf}
\pi_{s*} \Q_X[2] = \Q_Y \oplus \Q_Y[2]. 
\end{equation}
(Alternatively, \eqref{eq: splitting of structure sheaf} also follows from Lemma~\ref{lemma: derived pushforward}.) Using the splitting (\ref{eq: splitting of structure sheaf}) and the natural isomorphism (\ref{eq: pushing is restriction}), we continue the computation in (\ref{eq: first steps}). 
\begin{align*}
\Hom_K^\bullet(\pi_{s*}\Q_X[2], \pi_{s*}\mc{F}) &= \Hom_K^\bullet (\Q_Y \oplus \Q_Y[2], \pi_{s*}\mc{F}) \\
&= \mathbb{H}_K^s (\pi_{s*}\mc{F}) \oplus \mathbb{H}_K^s(\pi_{s*} \mc{F})[-2] \\
&= (C_K^s \oplus C_K^s[-2]) \otimes_{C_K^s} \mathbb{H}^s_K(\pi_{s*} \mc{F}) \\
&= \Hom^\bullet_K( \Q_Y, \Q_Y \oplus \Q_Y[-2]) \otimes_{C_K^s} \mathbb{H}^s_K(\pi_{s*}\mc{F}) \\
&= \Hom^\bullet_K(\Q_Y, \pi_{s*} \Q_X) \otimes_{C_K^s} \mathbb{H}_K^s(\pi_{s*} \mc{F}) \\
&= \mathrm{Res}^s \mathbb{H}^\bullet_K(\Q_X) \otimes_{C_K^s} \mathbb{H}_K^s (\pi_{s*} \mc{F})\\
&= \mathrm{Res}^s C_K \otimes_{C_K^s} \mathbb{H}_K^s(\pi_{s*} \mc{F})
\end{align*}
This shows that the canonical morphism (\ref{eq: canonical morphism}) is an isomorphism of $C^s_K$-modules, hence it must be an isomorphicm of $C_K$-modules. The proposition then follows from the natural isomorphism (\ref{eq: pushing is restriction}). 
\end{proof}
This completes the proof of Lemma~\ref{lem: push pull is tensoring with Bs}, and in turn, the proof of Theorem~\ref{thm: image of hypercohomology}. 
\end{proof}

\begin{example}
\label{example: hypercohomology vanishes on mobius band} $($Hypercohomology vanishes outside of $\mc{M}_{\mathrm{LV}}^0$ for $SL_2(\R))$ Consider the $\IC$ sheaf corresponding to the M\"obius band local system $\mc{L}$ in Example~\ref{example: LV module for SL(2,R)}, $\IC(X, \mc{L})$, and its non-equivariant version, $\IC(X, \For \mc{L})$. Recall that because $\mc{L}$ is clean, $\IC(X, \For \mc{L}) = j_* \For \mc{L}$, where $j:\mc{O} \hookrightarrow X$ is inclusion of the open orbit. Hence by equation (\ref{eq: hypercohomology agrees with sheaf cohomology}), \[
\HH^\bullet(\IC(X, \For \mc{L})) = H^*(X; j_* \For \mc{L}) = H^*(\C^\times; \For \mc{L}).
\]
We showed (in two ways) in Example~\ref{example: Hecke action on the non-trivial block SL2} that $H^*(\C^\times; \For \mc{L})=0$, so the ordinary hypercohomology functor (and hence the equivariant hypercohomology functor) vanishes on $\IC(X, \For \mc{L})$ (resp.\ $\IC(X, \mc{L})$). This example illustrates that our techniques do not necessarily capture the entire Lusztig--Vogan module in examples with non-trivial (after forgetting equivariance) local systems.   
\end{example}

Theorem~\ref{thm: image of hypercohomology} establishes that the hypercohomology functor 
\[
\mathbb{H}^\bullet_K: \mc{M}_{LV}^0 \rightarrow \mc{N}_{LV}^0
\]
is essentially surjective. In the following section, we show that it is also fully faithful, and hence establishes an equivalence of categories 
\[
\mc{M}_{LV}^0 \simeq \mc{N}_{LV}^0.
\]
\begin{remark}
This equivalence of categories appears (in slightly different language) as Theorem 3.12 in \cite{BV}.
\end{remark}


\subsection{Hypercohomology is fully faithful}
\label{sec: hypercohomology is fully faithful}

\begin{theorem}
\label{thm: fully faithful}
The equivariant hypercohomology functor 
\[
\HH^\bullet_K:\mc{M}^0_{LV}\longrightarrow \mc{N}_{LV}^0
\]
is fully faithful. 
\end{theorem}

To prove Theorem~\ref{thm: fully faithful}, we must show that for any $\mc{F}, \mc{G} \in \mc{M}_{LV}^0$, the map
\begin{equation}
\label{eq: fully faithful}
\Hom_K(\mc{F}, \mc{G}) \xrightarrow{\mathbb{H}_K^\bullet} \Hom_{(P^K, R)\mathrm{-gbim}}(\mathbb{H}_K^\bullet (\mc{F}), \mathbb{H}_K^\bullet(\mc{G}))
\end{equation}
is bijective. We will do so in three steps:
\begin{enumerate}
\item First, we establish that $\HH^\bullet_K$ is compatible with certain adjunctions, allowing us to reduce to the case where $\mc{F} = \bm{IC}_\mc{Q}$. This is Lemma~\ref{lem: hypercohomology commutes with adjunctions} and the preceding material.
\item Second, we establish that $\mathbb{H}_K^\bullet$ is faithful by showing that the map \eqref{eq: fully faithful} is injective. This is Theorem~\ref{thm: faithful}. 
\item Third, we establish that $\mathbb{H}_K^\bullet$ is full by showing that the map \eqref{eq: fully faithful} is surjective. This is Theorem~\ref{thm: full}. 
\end{enumerate}

\begin{lemma}
\label{lem: self-adjoint}
For any $s \in S$, the functor $\Theta_s:\mc{M}_{LV}^0 \rightarrow \mc{M}_{LV}^0$ and the functor $- \otimes_R B_s: \mc{N}_{LV}^0 \rightarrow \mc{N}_{LV}^0$ are self-adjoint.
\end{lemma}  

\begin{proof}
For $\mc{F}, \mc{G} \in D^b_K(X)$, 
\begin{align*}
\Hom(\mc{F}, \pi_s^* \pi_{s*} \mc{G}[1] ) &= \Hom (\mc{F}, \pi_s^!\pi_{s*}\mc{G}[-1]) \\
&= \Hom( \pi_{s!}\mc{F}, \pi_{s*} \mc{G}[-1]) \\ 
&= \Hom(\pi_s^* \pi_{s*} \mc{F}[1], \mc{G}).
\end{align*}
This establishes the first claim. To see that $- \otimes_R B_s$ is self-adjoint, note that the ring inclusion $R^s \hookrightarrow R$ is a degree $1$ Frobenius extension. (See \cite[Ch. 8]{SBim} for a detailed discussion of Frobenius extensions.) Hence we have two adjoint pairs: 
\[
(\Ind_{R^s}^R, \Res_{R^s}^R) \text{ and } (\Res_{R^s}^R(2), \Ind_{R^s}^R).
\]
Each functor can be realized as tensor product: 
\[
\Ind_{R^s}^R(-) = - \otimes_{R^s}R \text{ and } \Res_{R^s}^R(-) = - \otimes_R R.
\]
It follows that for $M,N \in (P^K, R)\mathrm{-gbim}$, 
\begin{align*}
    \Hom(M, N \otimes_R B_s) &= \Hom(M, N \otimes_R R \otimes_{R^s} R(1) ) \\
    &= \Hom(M \otimes_R R, N\otimes_R R(-1)) \\
    &= \Hom(M\otimes_R R \otimes_{R^s} R(1), N) \\ 
    &= \Hom(M \otimes_R B_s, N). 
\end{align*}
\end{proof}

To show that the hypercohomology functor commutes with these adjunctions, we will utilize the categorical module structure of $\mc{M}_{LV}^0$ and $\mc{N}_{LV}^0$. Our arguments depend on some fundamental properties of {\em dual pairs} in monoidal categories, which we recall in Appendix~\ref{app: dual pairs}. We encourage any reader unfamiliar with these tools to read Appendix~\ref{app: dual pairs} before proceding.  

\begin{lemma}
\label{lem: ICs are self-dual}
For $s \in S$, the pair $({\bm{IC}}_s, {\bm{IC}}_s)$ is a dual pair in $\mc{H}$, and the pair $(B_s, B_s)$ is a dual pair in $\SBim$.
\end{lemma}
\begin{proof}
Denote by $\Theta^B_s: D^b_B(X) \rightarrow D^b_B(X)$ the push-pull functor on $D^b_B(X)$:
\[
\Theta^B_s(\mc{F}):= \pi_s^* \pi_{s*} \mc{F}[1],
\]
where $\pi_s:X \rightarrow G/P_s$ is as in (\ref{eq: pi s}) and $\begin{tikzcd} D^b_B(X) \arrow[r, "\pi_{s*}"', shift  right] & D^b_B(G/P_s) \arrow[l, "\pi_s^*"', shift right] \end{tikzcd}$ are the direct and inverse image functors on the $B$-equivariant derived category. By analogous results to Lemma~\ref{lem: self-adjoint} and Lemma~\ref{lem: push pull is convolution with ICs}, $\Theta^B_s$ is self-adjoint and can be realized as convolution with $\bm{IC}_s$. Let $\eta: id \rightarrow \Theta^B_s \Theta^B_s$ (resp.\ $\epsilon: \Theta^B_s \Theta^B_s \rightarrow id$) be the unit (resp.\ counit) of the adjunction $(\Theta^B_s, \Theta^B_s)$ on $\mc{H}$. Here $id$ is the identity functor on $\mc{H}$. We define co-evaluation and evaluation morphisms in $\mc{H}$ by
\begin{align*}
&\Q_B \xrightarrow{\eta_{\Q_B}} \Q_B \Theta^B_s \Theta^B_s \simeq \Q_B * \bm{IC}_s * \bm{IC_s} \simeq \bm{IC}_s * \bm{IC}_s, \text{ and } \\
&\bm{IC}_s * \bm{IC}_s \simeq \Q_B * \bm{IC}_s * \bm{IC}_s \simeq \Q_B \Theta^B_s \Theta^B_s \xrightarrow{\epsilon_{\Q_B}} \Q_B.
\end{align*}
The zig-zag diagrams for the adjunction $(\Theta^B_s, \Theta^B_s, \eta, \epsilon)$ guarantee that the diagrams (\ref{eq: dual pair axiom 1}) and (\ref{eq: dual pair axiom 2}) commute. 

Similarly, the functor $-\otimes_R B_s: \SBim \rightarrow \SBim$ is self-adjoint, and we denote by $\eta': id \rightarrow - \otimes_RB_s \otimes_R B_s$ and $\epsilon': - \otimes_R B_s \otimes_R B_s \rightarrow id$ the unit and counit of the adjunction. We define co-evaluation and evaluation morphisms in $\SBim$ for the pair $(B_s, B_s)$ by 
\begin{align*}
    & R \xrightarrow{\eta'_R} R \otimes_R B_s \otimes_R B_s \simeq B_s \otimes_R B_s, \text{ and }\\
    & B_s \otimes_R B_s \simeq R \otimes_R B_s \otimes_R B_s \xrightarrow{\eta'_R} R.
\end{align*}
Again, the commutativity of the dual pair diagrams (\ref{eq: dual pair axiom 1}) and (\ref{eq: dual pair axiom 2}) follows from the zig-zag diagrams for the adjunction $(- \otimes_R B_s, - \otimes_R B_s, \eta', \epsilon')$.  
\end{proof}

\begin{lemma}
\label{lem: hypercohomology commutes with adjunctions}
For every $s \in S$ and every pair of objects $\mc{F}, \mc{G} \in \mc{M}_{LV}^0$, the diagram 
\[
\begin{tikzcd}
\Hom(\mc{F}, \mc{G}\Theta_s) \arrow[r, "\sim"] \arrow[d, "\mathbb{H}^\bullet_{K}"] & \Hom(\mc{F} \Theta_s, \mc{G}) \arrow[d, "\mathbb{H}_K^\bullet"] \\
\Hom(\mathbb{H}_K^\bullet(\mc{F}), \mathbb{H}_K^\bullet(\mc{G}) \otimes_R B_s) \arrow[r, "\sim"] & \Hom(\mathbb{H}_K^\bullet(\mc{F}) \otimes_R B_s, \mathbb{H}_K^\bullet (\mc{G}))
\end{tikzcd}
\]
commutes. Here the horizontal arrows are adjunction isomorphisms, and we are using the isomorphism in Lemma~\ref{lem: push pull is tensoring with Bs} to replace  $\mathbb{H}_K^\bullet(\mc{G} \Theta_s)$ with $\mathbb{H}_K^\bullet(\mc{G}) \otimes_R B_s$. 
\end{lemma}

\begin{proof} It is enough to show that hypercohomology preserves the unit and counit of each adjunction. In other words, if the unit and counit of the adjoint pair $\begin{tikzcd} \mc{M}_{LV}^0 \arrow[r, shift left, "\Theta_s"] & \mc{M}_{LV}^0 \arrow[l, shift left, "\Theta_s"] \end{tikzcd}$ are $\eta: id \rightarrow \Theta_s \Theta_s$ and $\epsilon: \Theta_s \Theta_s \rightarrow id$, it is enough to show that the collections of morphisms 
\[
\{\mathbb{H}_K^\bullet(\eta_\mc{F})\}_{\mc{F} \in \mc{M}_{LV}^0} \text{ and } \{\mathbb{H}_K^\bullet(\epsilon_\mc{F})\}_{\mc{F} \in \mc{M}_{LV}^0}
\]
define the unit and counit of the adjunction of the adjunction  $\begin{tikzcd} \mc{N}_{LV}^0 \arrow[r, shift left, "-\otimes_R B_s"] & \mc{N}_{LV}^0 \arrow[l, shift left, "- \otimes_R B_s"] \end{tikzcd}$. 

We begin by noting that because $\mathbb{H}_K^\bullet: \mc{M}_{LV}^0 \rightarrow \mc{N}_{LV}^0$ is essentially surjective (Theorem~\ref{thm: image of hypercohomology}), every object $M \in \mc{N}_{LV}^0$ is isomorphic to $\mathbb{H}_K^\bullet(\mc{F})$ for some $\mc{F} \in \mc{M}_{LV}^0$. Hence for each object in $\mc{N}_{LV}^0$, there is a corresponding morphism in $\{\mathbb{H}_K^\bullet(\eta_\mc{F})\}_{\mc{F} \in \mc{M}_{LV}^0}$ and in $\{\mathbb{H}_K^\bullet(\epsilon_\mc{F})\}_{\mc{F} \in \mc{M}_{LV}^0}$.

Because $(\bm{IC}_s, \bm{IC}_s)$ and $(B_s, B_s)$ are dual pairs in $\mc{H}$ and $\SBim$, respectively (Lemma~\ref{lem: ICs are self-dual}), we can apply Lemma~\ref{lem: dual pairs give adjunctions} to $\mc{M}_{LV}^0$ and $\mc{N}_{LV}^0$. We conclude that the unit and counit of the adjunction $(\Theta_s, \Theta_s)$ are
\begin{align*}
&\eta_\mc{F}: \mc{F} \xrightarrow{\mc{F} * \beta} \mc{F} * \bm{IC}_s * \bm{IC}_s \simeq \mc{F} \Theta_s \Theta_s, \text{ and} \\
& \epsilon_\mc{G}: \mc{G} \Theta_s \Theta_s \simeq \mc{G} * \bm{IC}_s * \bm{IC}_s \xrightarrow{\mc{G} * \mc{L}} \mc{G},
\end{align*}
where $\beta: \mathbbm{1} \rightarrow \bm{IC}_s * \bm{IC}_s$ and $\mc{L}: \bm{IC}_s * \bm{IC}_s \rightarrow \mathbbm{1}$ are the co-evaluation and evaluation morphisms for the dual pair $(\bm{IC}_s, \bm{IC}_s)$. Similarly, the unit and counit for the adjoint pair $\begin{tikzcd} \mc{N}_{LV}^0 \arrow[r, shift left, "- \otimes_R B_s"] & \mc{N}_{LV}^0 \arrow[l, shift left, "- \otimes_R B_s"] \end{tikzcd}$ are 
\begin{align*}
    &\eta'_M: M \xrightarrow{M \otimes_R \beta'} M \otimes_R B_s \otimes_R B_s, \text{ and } \\
    &\epsilon_M: N \otimes_R B_s \otimes_R B_s \xrightarrow{N \otimes_R \mc{L}'} N,
\end{align*}
where $\beta': R \rightarrow B_s \otimes_R B_s$ and $\mc{L}': B_s \otimes_R B_s \rightarrow R$ are the co-evaluation and evaluation morphisms for the dual pair $(B_s, B_s)$. 

By Lemma~\ref{lem: push pull is tensoring with Bs}, 
\[
\mathbb{H}_K^\bullet(\eta_\mc{F}) = \mathbb{H}_K^\bullet(\mc{F}) \otimes_R \beta' = \eta'_{\mathbb{H}_K^\bullet(\mc{F})}: \mathbb{H}_K^\bullet(\mc{F}) \rightarrow \mathbb{H}_K^\bullet(\mc{F}) \otimes_R B_s \otimes_R B_s,
\]
and 
\[
\mathbb{H}^\bullet_K(\epsilon_\mc{F}) = \mathbb{H}_K^\bullet(\mc{F}) \otimes \mc{L}' = \epsilon'_{\mathbb{H}^\bullet_K(\mc{F})}: \mathbb{H}_K^\bullet (\mc{F}) \otimes_R B_s \otimes_R B_s \rightarrow \mathbb{H}_K^\bullet(\mc{F}). 
\]
This proves the lemma. 
\end{proof}

This completes the first step in our proof of Theorem~\ref{thm: fully faithful}. By Lemma~\ref{lem: hypercohomology commutes with adjunctions}, proving Theorem~\ref{thm: fully faithful} reduces to showing that the natural map
\begin{equation}
\label{eq: reduction} 
\Hom_K(\bm{IC}_\mc{Q}, \mc{G}) \xrightarrow{\mathbb{H}_K^\bullet} \Hom_{(P^K, R)\mathrm{-gbim}}(\mathbb{H}^\bullet_K(\bm{IC}_\mc{Q}), \mathbb{H}_K^\bullet(\mc{G}))
\end{equation}
is bijective for any $\mc{G} \in \mc{M}_{LV}^0$ and closed $K$-orbit $\mc{Q}$. 

\begin{theorem}
\label{thm: faithful}
The map (\ref{eq: reduction}) is injective for any $\mc{G} \in \mc{M}_{LV}^0$ and closed $K$-orbit $\mc{Q}$. 
\end{theorem}

The proof of Theorem~\ref{thm: faithful} proceeds as follows. First we show that for any $\mc{G} \in D^b_K(X)$, the injectivity of (\ref{eq: reduction}) follows from the injectivity of the canonical map 
\begin{equation}
\label{eq: counit of shriek}
\mathbb{H}_K^\bullet(i_{\mc{Q}!}i_\mc{Q}^! \mc{G}) \rightarrow \mathbb{H}_K^\bullet(\mc{G})
\end{equation}
coming from the adjunction $(i_{\mc{Q}!}, i_\mc{Q}^!)$. This is the content of Lemma~\ref{lem: commuting pentagon}. Then we show that (\ref{eq: counit of shriek}) is injective for objects $\mc{G} \in D^b_K(X)$ satisfying certain parity conditions (Lemma~\ref{lem: injective}). Finally, we show that objects in $\mc{M}_{LV}^0$ satisfy the necessary parity conditions (Corollary \ref{cor: M_LV is parity}).  

We start by recalling some basic facts in $D^b_K(X)$. Denote by $i:=i_\mc{Q}:\mc{Q} \hookrightarrow X$, and let $\mc{G} \in D^b_K(X)$ be arbitrary. Denote the hypercohomology functor on $D^b_K(\mc{Q})$ by $\mathbb{H}_K^\mc{Q} := \Hom_{D^b_K(\mc{Q})}^\bullet(\Q_\mc{Q}, -): D^b_K(\mc{Q}) \rightarrow H^*_K(\mc{Q})\mathrm{-gmod}$. The adjunction $(i^*, i_*)$ gives isomorphisms
\begin{equation}
\label{eq: one}
\mathbb{H}_K^\bullet(\bm{IC}_\mc{Q}) \simeq \mathbb{H}_K^\mc{Q}(\Q_\mc{Q})\simeq H^*_K(\mc{Q}).
\end{equation}
The adjunction $(i_!, i^!)$ gives an injection 
\begin{equation}
\label{eq: two}
\Hom_K(\bm{IC}_\mc{Q}, \mc{G}) \simeq \Hom_{D_K^b(\mc{Q})}(\Q_\mc{Q}, i^! \mc{G}) \hookrightarrow \Hom^\bullet_{D^b_K(\mc{Q})}(\Q_\mc{Q}, i^!\mc{G}) = \mathbb{H}^\mc{Q}_K(i^! \mc{G}). 
\end{equation}
Because $i^*\Q_X = \Q_\mc{Q}$, we have 
\begin{equation}
\label{eq: three}
\Hom^\bullet_{D_K^b(\mc{Q})}(\Q_\mc{Q}, i^! \mc{G}) \simeq \Hom^\bullet_K(\Q_X, i_* i^! \mc{G}) \simeq \mathbb{H}_K^\bullet (i_! i^! \mc{G}).
\end{equation}
And finally, because $H^*_K(\mc{Q})$ is cyclic as a $H^*_K(X)$-module (Lemma~\ref{lem: standards are equiv cohomology}), there is a natural injecton 
\begin{equation}
\label{eq: four}
\Hom_K(H^*_K(\mc{Q}), \mathbb{H}_K(\mc{G})) \hookrightarrow \mathbb{H}_K(\mc{G})
\end{equation}
given by $\varphi \mapsto \varphi(1)$. 

For $M,N \in (P^K, R)\mathrm{-gbim}$, denote by 
\[
\Hom_\mathrm{bim}(M,N):= \Hom_{(P^K, R)\mathrm{-gbim}}(M,N).
\]
\begin{lemma}
\label{lem: commuting pentagon}
The diagram 
\begin{equation}
\label{eq: commuting pentagon}
\begin{tikzcd}
\Hom_K(\bm{IC}_\mc{Q}, \mc{G}) \arrow[d, hookrightarrow, "a"'] \arrow[rr, "\mathbb{H}_K^\bullet"]  & & \Hom_{\mathrm{bim}}(H^*_K(\mc{Q}), \mathbb{H}^\bullet_K(\mc{G})) \arrow[d, hookrightarrow, "d"]\\
\mathbb{H}_K^\mc{Q}(i^!\mc{G}) \arrow[dr, "b", "\sim"']  & & \mathbb{H}^\bullet_K(\mc{G}) \\
& \mathbb{H}_K^\bullet(i_!i^!\mc{G}) \arrow[ur, "c"] &
\end{tikzcd}
\end{equation}
obtained from (\ref{eq: one})-(\ref{eq: four}) commutes for any $\mc{G} \in D^b_K(X)$. 
\end{lemma}
\begin{proof}
Denote by $ \overset{*}{\eta}:\mathbbm{1} \rightarrow i_*i^*$ the unit of the adjunction $(i^*, i_*)$, and by $\overset{!}{\eta}: \mathbbm{1} \rightarrow i^! i_!$ (resp.\ $\overset{!}{\epsilon}: i_! i^! \rightarrow \mathbbm{1}$) the unit (resp.\ counit) of the adjunction $(i_!, i^!)$. 

Let $(\bm{IC}_\mc{Q} \xrightarrow{\varphi} \mc{G}) \in \Hom_K(\bm{IC}_\mc{Q}, \mc{G})$. We start by computing the image of $\varphi$ as we move counterclockwise around the diagram. Under the injection $a$, $\varphi$ maps to the composition 
\[
\Q_\mc{Q} \xrightarrow{\overset{!}{\eta}_{\Q_\mc{Q}}} i^! i_! \Q_\mc{Q} \xrightarrow{i^! \varphi} i^! \mc{G}.
\]
Applying the isomorphism $b$ to this and using the fact that $\Q_\mc{Q}=i^*\Q_X$, we obtain
\[
\Q_X \xrightarrow{\overset{*}{\eta}_{\Q_X}} i_* i^* \Q_X \xrightarrow{i_*\overset{!}{\eta}_{\Q_\mc{Q}}} i_* i^! i_! \Q_\mc{Q} \xrightarrow{ i_* i^! \varphi} i_* i^! \mc{G}.
\]
Finally, applying the map $c$ and using the fact that $i_!=i_*$, we obtain the homomorphism
\begin{equation}
\label{eq: counterclockwise}
\Q_X \xrightarrow{\overset{*}{\eta}_{\Q_X}} i_! i^* \Q_X \xrightarrow{i_! \overset{!}{\eta}_{\Q_\mc{Q}}} i_! i^! i_! \Q_\mc{Q} \xrightarrow{i_! i^! \varphi} i_! i^! \mc{G} \xrightarrow{\overset{!}{\epsilon}_\mc{G}} \mc{G}
\end{equation}
in $\mathbb{H}_K^\bullet(\mc{G})$.

Now we compute the image of $\varphi$ as we move clockwise around the diagram. Under $\mathbb{H}_K^\bullet$, $\varphi$ maps to 
\[
H^*_K(\mc{Q}) \simeq \Hom_{D^b_K(\mc{Q})}^\bullet(i^*\Q_X, \Q_\mc{Q}) \xrightarrow{\sim} \Hom_K^\bullet(\Q_X, i_*\Q_\mc{Q}) \xrightarrow{\mathbb{H}_K^\bullet(\varphi)} \Hom_K^\bullet (\Q_X, \mc{G}) = \mathbb{H}^\bullet_K(\mc{G}),
\]
where the first arrow is the adjunction isomorphism sending 
\[
(\Q_\mc{Q} = i^*\Q_X \xrightarrow{f} \Q_\mc{Q}) \mapsto (\Q_X \xrightarrow{\overset{*}{\eta}_{\Q_X}} i_* i^* \Q_X \xrightarrow{i_* f} i_* \Q_\mc{Q} )
\]
and the second arrow is post-composition with $\varphi$. Under $d$, this maps to the image of the identity map:
\begin{equation}
\label{eq: clockwise}
\Q_X \xrightarrow{\overset{*}{\eta}_{\Q_X}} i_* i^* \Q_X = i_* \Q_\mc{Q} \xrightarrow{\varphi} \mc{G}.
\end{equation}
The sequences (\ref{eq: clockwise}) and (\ref{eq: counterclockwise}) fit into the following diagram:
\[
\begin{tikzcd}
\Q_X \arrow[r, "\overset{*}{\eta}_{\Q_X}"] & i_! i^* \Q_X \arrow[r, "i_! \overset{!}{\eta}_{\Q_\mc{Q}}"] \arrow[rd, "id"'] & i_! i^! i_! \Q_\mc{Q} \arrow[r, "i_! i^! \varphi"] \arrow[d, "\overset{!}{\epsilon}_{i_!\Q_\mc{Q}}"] & i_! i^! \mc{G} \arrow[r, "\overset{!}{\epsilon}_\mc{G}"] & \mc{G} \\
& & i_* \Q_\mc{Q} \arrow[rru, "\varphi"'] & &
\end{tikzcd}
\]
The zig-zag diagrams for the adjunction $(i_!, i^!)$ guarantee that the left triangle commutes. The right triangle commuted by the naturality of $\overset{!}{\epsilon}$. Hence the images agree, and the diagram (\ref{eq: commuting pentagon}) commutes. 
\end{proof}

Next we establish a condition on complexes $\mc{G} \in D^b_K(X)$ which guarantees that the map (\ref{eq: counit of shriek}) is injective. We will do this by applying parity results from \cite{ParitySheaves}. To begin, we need to establish that our geometric setting satisfies the necessary assumptions of loc.\ cit. 

Let $\Lambda$ be the stratification of $X$ by $K$-orbits. We need to show that the pair $(X, \Lambda)$ satisfies conditions (2.1) and (2.2) of \cite{ParitySheaves}. Because we are taking coefficients in a field, these assumptions reduce to establishing the following lemma \cite[Rmk.\ 2.2]{ParitySheaves}.

\begin{lemma}
\label{lem: odd cohomology of K-orbits vanishes}
For any $K$-orbit $\mc{O}$ in $X$ and $\mc{L} \in \Loc_K(\mc{O})$, $H^n_K(\mc{O}; \mc{L})=0$ for $n$ odd. 
\end{lemma}
\begin{proof}
Let $x \in \mc{O}$, $K_x = \stab_Kx$, and $K_x^\circ \subseteq K_x$ the connected component of the identity. There is a fibration
\[
\widetilde{\mc{O}}:=K/K_x^\circ \xrightarrow{\pi} \mc{O} \simeq K/K_x
\]
given by $gK_x^\circ \mapsto g \cdot x$. The fibres of $\pi$ are isomorphic to $K_x/K_x^\circ$, which has exponent $2$ \cite{LV}. The category $\Loc_K(\mc{O})$ is equivalent to the category of $\C[K_x/K_x^\circ]$-modules. In particular, $\Loc_K(\mc{O})$ is semisimple because $K_x/K_x^\circ$ is abelian. Under this equivalence, $\pi_*\Q_{\widetilde{\mc{O}}}$ corresponds to the regular representation. Hence all simple $\mc{L} \in \Loc_K(\mc{O})$ appear as summands of $\pi_*\Q_{\widetilde{\mc{O}}}$, so it is enough to show that $H^*_K(\mc{O}; \pi_*\Q_{\widetilde{\mc{O}}})$ vanishes in odd degree. 

Let $(K_x^\circ)^\mathrm{red}$ be the maximal reductive quotient of $K_x^\circ$. We have
\[
H^*_K(\mc{O}; \pi_*\Q_{\widetilde{\mc{O}}}) = H^*_K(\widetilde{\mc{O}}) = H^*_{K_x^\circ}(\mathrm{pt}) = H^*_{(K_x^\circ)^\mathrm{red}}(\mathrm{pt}) = \left( H^*_{T'}(\mathrm{pt})\right)^{W'},
\]
where $T'$ is a maximal torus in $(K_x^\circ)^\mathrm{red}$ and $W'$ is the Weyl group of $(K_x^\circ)^\mathrm{red}$. It is a classical result that $H_{T'}^*(\mathrm{pt})$ is a polynomial ring with generators in even degree (see Borel isomorphism in proof of Lemma~\ref{lem: standards are equiv cohomology}), which establishes the result. 
\end{proof}

Lemma~\ref{lem: odd cohomology of K-orbits vanishes} guarantees that we can apply results in \cite{ParitySheaves} in our setting. 

\begin{definition}
\label{def: parity} Let $\mc{G} \in D^b_K(X)$, and denote by $\mc{H}^i(\mc{G})$ its cohomology sheaves. For an orbit $\mc{O}$, denote by $k_\mc{O}: \mc{O} \rightarrow X$ the locally closed inclusion map. 
\begin{enumerate}
\item The complex $\mc{G}$ is {\em $*$-even} (resp.\ {\em $!$-even}) if for all $K$-orbits $\mc{O}$ in $X$, $\mc{H}^n(k_\mc{O}^*\mc{G})$ (resp.\ $\mc{H}^n(k_\mc{O}^! \mc{G})$) vanishes for $n$ odd. 
\item The complex $\mc{G}$ is {\em $*$-odd} (resp.\ {\em $!$-odd}) if $\mc{G}[1]$ is $*$-even (resp.\ $!$-even). 
\item The complex $\mc{G}$  is {\em even} (resp.\ {\em odd}) if it is both $*$-even (resp.\ $*$-odd) and $!$-even (resp.\ $!$-odd). 
\item The complex $\mc{G}$ is {\em parity} if it splits into the direct sum of an even and an odd complex. 
\end{enumerate}
\end{definition}

\begin{lemma}
\label{lem: hom vanishing}
\cite[Cor.\ 2.8]{ParitySheaves} If $\mc{F}$ is $*$-even and $\mc{G}$ is $!$-odd, then \[
\Hom_K(\mc{F}, \mc{G})=0.
\] 
\end{lemma}

Using Lemma~\ref{lem: hom vanishing}, we are able to prove our desired result. Fix a closed orbit $\mc{Q}$ and set $i:=i_\mc{Q}:\mc{Q} \hookrightarrow X$. 
\begin{lemma}
\label{lem: injective}
If $\mc{G} \in D^b_K(X)$ is parity, then the natural map $\mathbb{H}_K^\bullet(i_! i^! \mc{G}) \rightarrow \mathbb{H}_K^\bullet(\mc{G})$ is injective. 
\end{lemma}

\begin{proof}
Let $U=X \smallsetminus \mc{Q}$. We  have open and closed immersions
\[
U \xhookrightarrow{j} X \xhookleftarrow{i} \mc{Q},
\]
so for any $\mc{G} \in D^b_K(X)$, there is a distinguished triangle
\begin{equation}
\label{eq: open closed triangle}
i_!i^! \mc{G} \rightarrow \mc{G} \rightarrow j_* j^* \mc{G} \xrightarrow{[1]}. 
\end{equation}
Because the functor $\Hom_K(\Q_X, -)$ is cohomological, we obtain from (\ref{eq: open closed triangle}) a long exact sequence
\begin{equation}
\label{eq: long exact sequence}
\begin{split}
\cdots \rightarrow \Hom^0_K(\Q_X, i_! i^! \mc{G}) &\rightarrow \Hom^0_K(\Q_X, \mc{G}) \\ &\rightarrow \Hom^0(\Q_X, j_* j^* \mc{G}) \rightarrow 
\Hom^1(\Q_X, i_! i^! \mc{G}) \rightarrow \cdots
\end{split}
\end{equation}

Assume that $\mc{G}$ is $!$-even. Then by Lemma~\ref{lem: hom vanishing}, $\Hom_K^n(\Q_X, \mc{G}) = \Hom_K(\Q_X, \mc{G}[n])=0$ for $n$ odd, so every seventh term in (\ref{eq: long exact sequence}) vanishes. Morover, we claim that if $\mc{G}$ is $!$-even, then $j_*j^*\mc{G}$ is $!$-even as well, so every 8th term in (\ref{eq: long exact sequence}) also vanishes. To see this, first note that $i^!j_*=0$, so $\mc{H}^n(i^!j_*j^*\mc{G})=0$ for all $n$. Any other orbit $\mc{O} \neq \mc{Q}$ is contained in $U$, and the inclusion map $k_\mc{O}$ factors 
\[
\begin{tikzcd}
\mc{O} \arrow[r, "\tilde{j}", hookrightarrow] \arrow[rr, bend right, "k_\mc{O}", hookrightarrow] & U \arrow[r, "j", hookrightarrow] & X.
\end{tikzcd}
\]
Hence
\[
k_\mc{O}^! j_* j^* \mc{G} = \tilde{j}^! j^! j_* j^* \mc{G} = \tilde{j}^! j^* j_* j^* \mc{G} = \tilde{j}^! j^* \mc{G} = k_\mc{O}^! \mc{G}.
\]
Here $j^!=j^*$ because $j$ is an open immersion. This shows that $j_*j^* \mc{G}$ is also $!$-even, so the long exact sequence (\ref{eq: long exact sequence}) splits into exact sequences 
\begin{equation}
\label{eq: short exact sequence}
\begin{split}
0 \rightarrow \Hom^n_K(\Q_X, i_! i^! \mc{G}) &\xrightarrow{f^n} \Hom_K^n(\Q_X, \mc{G}) \\
&\rightarrow \Hom^n_K(\Q_X, j_* j^* \mc{G}) 
\rightarrow \Hom_K^{n+1}(\Q_X, i_! i^! \mc{G}) \rightarrow 0
\end{split}
\end{equation} 
for $n$ even. In particular, for $!$-even $\mc{G}$, the map $f^n$ is injective for all $n$ even (and zero for all $n$ odd). Analogously, if $\mc{G}$ is $!$-odd, then (\ref{eq: long exact sequence}) splits into (\ref{eq: short exact sequence}) for $n$ odd, and $f^n$ is injective for $n$ odd (and zero for $n$ even). Hence if $\mc{G}$ is parity, $f^n$ is injective for all $n$. As the natural map (\ref{eq: counit of shriek}) is given by $\bigoplus_{n \in \Z} f^n$, this shows that if $\mc{G}$ is parity, (\ref{eq: counit of shriek}) is injective.  
\end{proof}

Together, Lemma~\ref{lem: injective} and Lemma~\ref{lem: commuting pentagon} imply that (\ref{eq: reduction}) is injective if $\mc{G}$ is parity. Our final step is to show that every object in $\mc{M}_{LV}^0$ is parity. 

\begin{corollary}
\label{cor: M_LV is parity} 
All complexes in $\mc{M}_{LV}^0$ are parity. 
\end{corollary}

\begin{proof}
Let $k_\mc{O}: \mc{O} \rightarrow X$ be inclusion. Because the functors $k_\mc{O}^*$ are exact for all $K$-orbits $\mc{O}$, Remark~\ref{remark: purity} implies that $\mc{H}^n(k_{\mc{O}'}^*\bm{IC}(\overline{\mc{O}}, \mc{L}) )=0 $ for any $\mc{O}' \in K \backslash X$ and $n$ odd. Hence $\bm{IC}(\overline{\mc{O}}, \mc{L})$ is $*$-even. Because $\bm{IC}(\overline{\mc{O}}, \mc{L})$ is self-dual, it must also be $!$-even \cite[Rmk.\ 2.5(3)]{ParitySheaves}. By the decomposition theorem, $\pi_s^* \pi_{s*} \bm{IC}(\overline{\mc{O}}, \mc{L})$ is a sum of shifts of $\IC$ sheaves. Because $\mc{M}_{LV}^0$ is generated by the action of push-pull functors, any object in $\mc{M}_{LV}^0$ decomposes into a sum of shifts of $\IC$ sheaves. Hence any object in $\mc{M}_{LV}^0$ can be split into a sum of an even complex and an odd complex. 
\end{proof}

This completes the second step in the proof of Theorem~\ref{thm: fully faithful}. We have now established that the functor $\mathbb{H}_K^\bullet$ is  faithful on $\mc{M}_{LV}^0$. It remains to show that $\mathbb{H}_K^\bullet$ is full. 

\begin{theorem}
\label{thm: full}
The map (\ref{eq: reduction}) is surjective for $\mc{G} \in \mc{M}_{LV}^0$. 
\end{theorem}

The structure of the proof of Theorem~\ref{thm: full} is as follows. By Lemma~\ref{lem: commuting pentagon} and Theorem~\ref{thm: faithful}, for $\mc{G} \in \mc{M}_{LV}^0$ we have injections 
\begin{equation}
\label{eq: F and G}
\begin{tikzcd}[row sep=tiny]
\Hom_K(\bm{IC}_\mc{Q}, \mc{G}) \arrow[rd, hookrightarrow, "F", start anchor={[yshift=1ex]}] & \\
& \mathbb{H}_K^\bullet(\mc{G}) \\
\Hom_\mathrm{bim}(\mathbb{H}_K^\bullet(\bm{IC}_\mc{Q}), \mathbb{H}_K^\bullet(\mc{G})) \arrow[ru, hookrightarrow, "G"',  start anchor={[yshift=-2ex]} ] \hspace{10mm}& 
\end{tikzcd}
\end{equation}
given by $F=c \circ b \circ a$ and $G=d \circ \mathbb{H}_K^\bullet$ in the pentagon diagram (\ref{eq: commuting pentagon}). Moreover, because $\mc{G}$ is parity (Corollary \ref{cor: M_LV is parity}), the natural restriction map (\ref{eq: restriction})
\[
\mathbb{H}_K^\bullet(\mc{G}) \xrightarrow{\rres} \mathbb{H}_{T_K}^\bullet(\rres\mc{G})
\]
and the localization map 
\begin{equation}
\label{eq: localization}
\mathbb{H}_{T_K}^\bullet(\rres\mc{G}) \xrightarrow{i_{T_K}^*} \mathbb{H}_{T_K}^\bullet(i_{T_K}^*\rres\mc{G}),
\end{equation}
where $i_{T_K}: X^{T_K} \hookrightarrow X$ is the inclusion of $T_K$-fixed points, are also injective. This is shown in Lemma~\ref{lem: res is injective} and Lemma~\ref{lem: loc is injective}. In Lemma~\ref{lem: equating images} we show that the image of $\Hom_K(\bm{IC}_\mc{Q}, \mc{G})$ and the image of $\Hom_\mathrm{bim}(\mathbb{H}_K^\bullet(\bm{IC}_\mc{Q}), \mathbb{H}_K^\bullet(\mc{G}))$ in $\mathbb{H}_{T_K}^\bullet(i_{T_K}^* \rres \mc{G})$ agree under this series of injections. Because the diagram (\ref{eq: commuting pentagon}) in Lemma~\ref{lem: commuting pentagon} commutes and (\ref{eq: reduction}) is injective (Theorem~\ref{thm: faithful}), this implies the surjectivity of (\ref{eq: reduction}). 

Let $\rres:D^b_K(X) \rightarrow D^b_{T_K}(X)$ be the restriction functor sending $\mc{G} = (\mc{G}_X, \overline{\mc{G}}, \beta)$ to $\rres\mc{G}:= (\mc{G}_X, f^*\overline{\mc{G}}, \beta)$, where $f$ is the projection 
\[
X \times_{T_K} EK \xrightarrow{f} X \times_K EK. 
\]
Denote by $\mathbb{H}_K^\bullet(\mc{G}) \xrightarrow{\rres} \mathbb{H}_{T_K}^\bullet(\mc{G})$ the induced map on hypercohomology:
\begin{equation}
\label{eq: restriction}
\varphi = (\varphi_X, \overline{\varphi}) \xmapsto{\rres} \rres \varphi:=(\varphi_X, f^* \overline{\varphi}). 
\end{equation}
\begin{lemma}
\label{lem: res is injective}
If $\mc{G} \in D^b_K(X)$ is parity, then $\mathbb{H}_K^\bullet(\mc{G}) \xrightarrow{\rres} \mathbb{H}^\bullet_{T_K}(\rres\mc{G})$ is injective. 
\end{lemma}

\begin{proof}
Let $\mc{G} = (\mc{G}_X, \overline{\mc{G}}, \beta) \in D^b_K(X)$. The Leray spectral sequence for $\rres \mc{G}=(\mc{G}_X, f^*\overline{\mc{G}}, \beta) \in D^b_{T_K}(X)$ associated to the fibration 
\[
K/{T_K} \rightarrow X \times_{T_K} EK \xrightarrow{f} X \times_K EK
\]
is 
\[
E^{pq}_2=H^p(X \times_K EK; R^q f_* f^*\overline{\mc{G}}) \implies H^{p+q}(X \times_{T_K} EK; f^*\overline{\mc{G}}). 
\]
Because $f$ is a fibration with fiber $K/{T_K}$, we have 
\[
R^qf_* f^* \overline{\mc{G}} \simeq \overline{\mc{G}} \otimes H^q(K/{T_K}). 
\]
Hence $\mathbb{H}^\bullet_{T_K}(\rres\mc{G}) = H^*(X \times_T EK; f^*\overline{\mc{G}})$ can be computed with the spectral sequence
\begin{equation}
\label{eq: s.s. for H_T in terms of H_K}
H^p(X \times_K EK; \overline{\mc{G}}) \otimes H^q(K/{T_K}) \implies H^{p+q}(X \times_{T_K} EK; f^* \overline{\mc{G}}). 
\end{equation}
Let $B_K = K \cap B$, so $K/B_K$ is the flag variety of $K$. We have a fibration 
\[
B_K/{T_K} \rightarrow K/{T_K} \rightarrow K/B_K
\]
with affine fibres, so 
\[
H^*(K/{T_K})=H^*(K/B_K) \simeq P/P_+^K,
\]
where $P_+ \subseteq P$ is the augmentation ideal and $P_+^K:=P_+^{W_K}$ is the $W_K$ invariants in $P_+$. In particular, $H^*(K/{T_K})$ vanishes in odd degrees. 

By Lemma~\ref{lem: hom vanishing}, if $\mc{G}$ is $!$-even (resp.\ $!$-odd), then $\mathbb{H}_K^\bullet(\mc{G}) = H^*(X \times_KEK; \overline{\mc{G}})$ vanishes in odd (resp.\ even) degree. Hence if $\mc{G}$ is $!$-parity (i.e., either $!$-even or $!$-odd), the spectral sequence (\ref{eq: s.s. for H_T in terms of H_K}) has a chessboard pattern, hence degenerates on the $E_2$ page. In this case, 
\begin{equation}
\label{H_T}
\mathbb{H}_{T_K}^\bullet(\rres \mc{G}) \simeq P \otimes_{P^K} \mathbb{H}_K^\bullet(\mc{G}).
\end{equation}
If $\mc{G}$ is parity, then $\mc{G}$ is the direct sum of a $!$-even and a $!$-odd complex, hence (\ref{H_T}) also holds. This implies the result. 
\end{proof}

Next we will show that the localization map (\ref{eq: localization}) is also injective. 
\begin{lemma}
\label{lem: loc is injective}
Let $\mc{G} \in \mc{M}_{LV}^0$ and $i_{T_K}:X^{T_K} \hookrightarrow X$ be the inclusion of $T_K$-fixed points. The localization map $\mathbb{H}_{T_K}^\bullet(\rres \mc{G}) \xrightarrow{i_{T_K}^*} \mathbb{H}_{T_K}^\bullet(i_{T_K}^* \rres \mc{G})$ is injective. 
\end{lemma}
\begin{proof}
Let $\For:D^b_K(X) \rightarrow D^b(X)$ be the forgetful functor, as in (\ref{eq: forgetful functor}). Clearly $\For$ factors through $D^b_K(X) \xrightarrow{\rres} D^b_{T_K}(X)$. 

By the localization theorem in equivariant hypercohomology \cite[Thm.\ 1.6.2]{GKM}, $i_{T_K}^*$ is injective if $\rres\mc{G}$ is {\em equivariantly formal}; that is, if the spectral sequence for its equivariant cohomology   
\begin{equation}
\label{eq: s.s. for H_T}
E_2^{pq}=H^p_{T_K}(\mathrm{pt}) \otimes \mathbb{H}^\bullet( \mc{G}) \implies \mathbb{H}^{p+q}_{T_K}(\rres\mc{G})
\end{equation}
degenerates at the $E_2$ page. Here $\mathbb{H}^\bullet(\mc{G}) = \Hom^\bullet_{D^b_c(X)}(\Q_X, \For \mc{G})$ as in (\ref{eq: forgetful functor}). Because $H^*_{T_K}(\mathrm{pt})$ vanishes in odd degree, to show that $\rres \mc{G}$ is equivariantly formal, it is enough to show that $\mathbb{H}^\bullet(\mc{G})$ vanishes in all degrees of one parity (as this implies the spectral sequence (\ref{eq: s.s. for H_T}) has a checkerboard pattern, hence degenerates at $E_2$). 

Let $C:=H^*(X) = R/R_+^W$ be the cohomology ring of $X$. Note that $C$ is a graded ring that vanishes in odd degree. Since $\mc{G} \in \mc{M}_{LV}^0$, $\mc{G}$ is a direct summand of a finite direct sum of shifts of objects of the form $\bm{IC}_\mc{Q} * \bm{IC}_{s_1} * \cdots * \bm{IC}_{s_n}$ for some closed orbit $\mc{Q}$ and sequence of simple reflections $(s_1, \ldots, s_n)$. If $\mathbb{H}^\bullet(\bm{IC}_\mc{Q} * \bm{IC}_{s_1} * \cdots * \bm{IC}_{s_n})$ vanishes in odd degree for all closed orbits $\mc{Q}$ and all subsets of simple roots, then $\mathbb{H}^\bullet(\mc{G})$ does as well, so without loss of generality, we may assume that $\mc{G} = \bm{IC}_\mc{Q} * \bm{IC}_{s_1} * \cdots * \bm{IC}_{s_n}$. By \cite[Lem.\ 13, Kor. 2]{Soergel90}, we have
\[
\mathbb{H}^\bullet(\mc{G}) = C \otimes_{C^{s_n}}  \cdots \otimes_{C^{s_2}}C\otimes_{C^{s_1}} H^*(\mc{Q})[\dim Q]. 
\]
As $\mc{Q} \simeq K/B_K$ is a flag variety, $H^*(\mc{Q})$ vanishes in odd degrees, hence $\mathbb{H}^\bullet( \mc{G})$ vanishes in one parity. 
\end{proof}
Our final step is to compare the images of $\Hom_\mathrm{bim}(H^*_K(\mc{Q}), \mathbb{H}_K^\bullet(\mc{G}))$ and $\Hom_K(\bm{IC}_\mc{Q}, \mc{G})$ in $\mathbb{H}_{T_K}^\bullet( i_{T_K}^* \rres \mc{G})$. To do so, we will use the $R$-bimodule structure of $\mathbb{H}^\bullet_{T_K}(i^*_{T_K} \rres \mc{G})$. 

The space $\mathbb{H}_{T_K}^\bullet(\rres \mc{G})$ is naturally a module over $H^*_{T_K}(X) = P \otimes_{R^W} R$, so it also has the structure of an $(P,R)$-bimodule. Via $i_{T_K}^*$, $\mathbb{H}_{T_K}^\bullet(i_{T_K}^*\rres \mc{G})$ inherits a $P \otimes_{R^W} R$-module structure from $\mathbb{H}_{T_K}^\bullet(\rres \mc{G})$: for $\varphi \in \mathbb{H}_{T_K} ^\bullet(i_{T_K}^* \rres \mc{G})$ and $f \otimes g \in P \otimes_{R^W} R$, 
\[
f \otimes g \cdot \varphi = \varphi \circ i_{T_K}^*(f \otimes g). 
\]
With respect to this action, $i_{T_K}^*$ is a homomorphism of $P \otimes_{R^W}R$-modules. 

Recall that the $T_K$-fixed points in $X$ align with the $T$ fixed points, so they are isolated and in bijection with $W$ (Lemma~\ref{lem: agreement of fixed points}). Denote the skyscraper sheaf on $x_w =wB/B \in X^{T_K}$ by $\Q_w$. Analogously to the notation defined in (\ref{eq: polynomial powers}) for raising $R$-bimodules to polynomial powers, we set the following notation for objects in derived categories: for every $q = \sum q_i v^i \in \Z[v^{\pm 1}]$ and $\mc{G}$ a complex in a derived category, we denote by 
\[
\mc{G}^{\oplus q}:= \bigoplus \mc{G}[i]^{\oplus q_i}.
\]

\begin{lemma}
\label{lem: bimodule structure on fixed points}
As graded $(P,R)$-bimodules, 
\[
\mathbb{H}_{T_K}^\bullet (i_{T_K}^* \rres \mc{G}) \simeq \bigoplus_{w \in W} (P_w)^{\oplus q_w}
\]
where $q_w \in \Z[v^{\pm 1}]$ are given by the decomposition $i_{T_K}^* \rres \mc{G} = \bigoplus_{w \in W} (\Q_w)^{\oplus q_w}$ and $P_w \in (P,R)\mathrm{-gbim}$ is the standard module corresponding to $w \in W$ (given by the obvious adaptation to Definition~\ref{def: standard modules}). 
\end{lemma}
\begin{proof}
To begin, note that since $\Q_{X^{T_K}} = \bigoplus_{w \in W} \Q_w$, we clearly have 
\[
\mathbb{H}_{T_K}^\bullet(i_{T_K}^* \rres \mc{G}) = \Hom_{T_K}^\bullet(\bigoplus_{w \in W} \Q_w, \bigoplus_{w \in W} (\Q_w)^{\oplus q_w}) = \bigoplus_{w \in W}(P_w)^{\oplus q_w}
\]
as vector spaces, so the content of the lemma is that the $R$-bimodule structures align. 

The $(P,R)$-bimodule structure on $\mathbb{H}_{T_K}^\bullet(i_{T_K}^* \rres \mc{G})$ comes from the algebra homomorphism 
\begin{equation}
    \label{eq: fixed point pullback on rings}
P \otimes_{R^W} R = H^*_{T_K}(X) \xrightarrow{i_{T_K}^*} H^*_{T_K}(X^{T_K}) = \bigoplus_{w \in W} P_w,
\end{equation}
so to prove the lemma, we need to determine the image of an element $f \otimes g \in P \otimes_{R^W} R$ under $i_{T_K}^*$. The algebra $P \otimes_{R^W} R$ is generated by $P \otimes 1$ and $1 \otimes X(T)_\Q$, so it suffices to determine the image of $p \otimes 1$ and $1 \otimes \chi$ for $p \in R$ and $\chi \in X(T)_\Q$.

The map (\ref{eq: fixed point pullback on rings}) is compatible with the left $P=H^*_{T_K}(\mathrm{pt})$-module structures on either side because the commutativity of the the diagram 
\[
\begin{tikzcd}
X \arrow[rd, "\pi"'] \arrow[rr, hookleftarrow, "i_{T_K}"]& & X^{T_K}  \arrow[dl, "\pi_{T_K}"] \\
& \mathrm{pt} &
\end{tikzcd}
\]
implies the commutativity of the diagram 
\[
\begin{tikzcd}
H^*_{T_K}(X) \arrow[rr, "i_{T_K}^*"] & & H^*_{T_K}(X^{T_K}) \\
& H^*_{T_K}(\mathrm{pt}) \arrow[ul, "\pi^*"] \arrow[ur, "\pi_{T_K}^*"'] &
\end{tikzcd}
\]
Hence 
\[
i_{T_K}^*(p \otimes 1) = p \cdot i_{T_K}^*(1 \otimes 1) = p \cdot (1, \ldots, 1) = (p, \ldots, p).
\]

It remains to determine the image of $1 \otimes \chi$. First, note that we have a natural ring homomorphism $P \rightarrow H_{T_K}^*(X^{T_K})$ given by the Borel homomorphism. The construction is similar to that in the proof of Lemma~\ref{lem: standards are equiv cohomology}: for $\chi \in X(T_K)$ and $w \in W$, construct $\C^\times$-bundles
\begin{align*}
    ET_{K} \times_{T_K} G \times_B \C^\times_\chi &\xrightarrow{\mu^\chi} ET_K \times_{T_K} G \times_B \mathrm{pt} \simeq ET_K \times_{T_K} X\\
    ET_K \times_{T_K} wB \times_B \C^\times_\chi &\xrightarrow{\mu_w^\chi} ET_K \times_{T_K} wB \times_B \mathrm{pt} \simeq BT_K,
\end{align*}
where $\chi$ is inflated trivially to a character of $B$. By taking first Chern classses, we obtain ring homomorphisms 
\begin{align*}
P &\xrightarrow{\varphi}H^*_{T_K}(X); \chi \mapsto c_1(\mu^\chi), \text{ and } \\
P &\xrightarrow{\varphi_w} H^*_{T_K}(x_w); \chi \mapsto c_1(\mu_w^\chi).
\end{align*}
Because $\mu^\chi$ and $\mu^\chi_w$ fit into the Cartesian square
\[
\begin{tikzcd}
ET_K \times_{T_K} G \times_B \C^\times_\chi \arrow[r, "\mu^\chi"] & ET_K \times_{T_K} G \times_B \mathrm{pt} \simeq ET_K \times_{T_K} X \\
ET_K \times_{T_K} wB \times_B \C^\times_\chi \arrow[u, hookrightarrow] \arrow[r, "\mu_w^\chi"] & ET_K \times_{T_K} wB \times_B \mathrm{pt} \simeq BT_K \arrow[u, hookrightarrow], 
\end{tikzcd}
\]
the diagram 
\[
\begin{tikzcd}
& H^*_{T_K}(X) \arrow[dd, "i^*_{T_K}"] \\
P  \arrow[ur, "\varphi", end anchor={[yshift=1ex]}] \arrow[dr, "\displaystyle{\bigoplus_{w \in W}\varphi_w}"', end anchor={[yshift=-1ex]}, end anchor = {[xshift=-1ex]}] &  \\
& \displaystyle{\bigoplus_{w \in W} H^*_{T_K}(x_w)}
\end{tikzcd}
\]
commutes. This gives a natural ring homomorphism $\bigoplus_{w \in W} \varphi_w: P \rightarrow H^*_{T_K}(X^{T_K})$. The image of $1 \otimes \chi \in P \otimes_{R^W} R$ is then given by pre-composing this ring homomorphism with $\phi:R \rightarrow P$:
\[
1 \otimes \chi \mapsto \bigoplus_{w \in W} 1 \otimes \varphi_w \circ \phi(\chi). 
\]
In the proof of Lemma~\ref{lem: standards are equiv cohomology}, we showed that the right action of $g \in R$ on $p \in P \simeq H^*_{T_K}(x_w)$ obtained from $\varphi_w \circ \phi$ is given by 
\[
p \cdot g = p\varphi(wg), 
\]
so (\ref{eq: fixed point pullback on rings}) sends 
\[
1 \otimes \chi \mapsto (\phi(wx))_{w \in W}. 
\]
Hence for $f \otimes g \in P \otimes_{R^W} R$, 
\[
i_{T_K}^*: f \otimes g \mapsto (f\phi(wg))_{w \in W}.
\]
This completes the proof. \end{proof}

Using the $(P,R)$-bimodule structure of $\mathbb{H}^\bullet_{T_K}(i_{T_K}^*\rres \mc{G})$, we can now show that the images of $i_{T_K}^* \circ \rres \circ F$ and $i_{T_K}^* \circ \rres \circ G$ agree. As before, we set $i:=i_\mc{Q}: \mc{Q} \hookrightarrow X$. 

\begin{lemma}
\label{lem: equating images}
Let $F, G, \rres, i_{T_K}^*$ be as in (\ref{eq: F and G}), (\ref{eq: restriction}), and (\ref{eq: localization}). Then the images of $i_{T_K}^* \circ \rres \circ F$ and $i_{T_K} \circ \rres \circ G$ agree. 
\end{lemma}
\begin{proof}
In the proof of Lemma~\ref{lem: commuting pentagon}, we showed that the image of $F$ in $\mathbb{H}_K^\bullet(\mc{G})$ is 
\[
 \left\{ \Q_X \xrightarrow{\varphi} \mc{G} \mid \varphi \text{ factors as } \begin{tikzcd} \Q_X \arrow[r, "\overset{*}{\eta}_{\Q_X}"']  \arrow[rr, bend left, "\varphi"] & \bm{IC}_\mc{Q} \arrow[r, "\psi"'] & \mc{G} \end{tikzcd} \text{ for some }\psi \in \Hom_K(\bm{IC}_\mc{Q}, \mc{G}) \right\}, 
\]
where $\overset{*}{\eta}:\mathbbm{1} \rightarrow i_* i^*$ is the unit of the adjunction $(i^*, i_*)$. 

We can also identify the image of $G$ in $\mathbb{H}_K^\bullet(\mc{G})$. As elements of $\im G$ are of the form $\psi(1)$ for some $\psi \in \Hom_\mathrm{bim}(H^*_K(\mc{Q}), \mathbb{H}_K^\bullet(\mc{G}))$, and $H^*_K(\mc{Q}) \simeq P_v$ as $P^K \otimes_{R^W}R$-modules for some $v \in {^KW^\theta}$ (Lemma~\ref{lem: standards are equiv cohomology}), any element $g = \psi(1) \in \im G$ has the property that for any $p \in P^K$ and $r \in \varphi^{-1}(p)$, 
\begin{align}
\label{eq: r-action}
p \otimes 1 \cdot g = p \otimes 1 \cdot \psi(1) &= \psi(p \otimes 1 \cdot 1) \\
&= \psi(p) \nonumber\\
&= \psi(1 \otimes v^{-1} r \cdot 1) \nonumber\\
&= 1 \otimes v^{-1} r \cdot \psi(1) = 1 \otimes v^{-1}r \cdot g. \nonumber
\end{align}
Moreover, if $g \in \mathbb{H}_K^\bullet(\mc{G})$ satisfies (\ref{eq: r-action}), there exists $\psi \in \Hom_{\mathrm{bim}}(H^*_K(\mc{Q}), \mathbb{H}_K^\bullet(\mc{G}))$ defined by $\psi(1) = g$ such that $G(\psi)=g$, so 
\[
\im G = \left\{ \varphi \in \mathbb{H}_K^\bullet(\mc{G}) \mid (p \otimes 1 - 1 \otimes v^{-1}r) \cdot \varphi = 0 \text{ for all } p \in P^K  \text{ and } r \in \phi^{-1}(p)\right\}.
\]

We will show that under $i_{T_K}^* \circ \rres$, $\im F$ and $\im G$ map to the same subset of $\mathbb{H}_{T_K}^\bullet(i^*_{T_K} \rres \mc{G})$. Let $\varphi  \in \im F$. Under $\rres$ and $i^*_{T_K}$, $\varphi$ maps to 
\[
\Q_{X^{T_K}} = \bigoplus_{w \in W} \Q_w \xrightarrow{i_{T_K}^* \rres \varphi} i_{T_K}^* \rres \mc{G} = \bigoplus_{w \in W} (\Q_w)^{\oplus q_w}.
\]
Because $\varphi$ factors through $\bm{IC}_\mc{Q}$, $i_{T_K}^* \rres \varphi$ must factor through $i_{T_K}^* \bm{IC}_\mc{Q}$. The $T_K$-fixed points in $\mc{Q}$ correspond to Weyl group elements in the same $W_K$-coset as $v$ (Lemma~\ref{lem: fixed points in closed orbits}), so 
\[
i_{T_K}^* \bm{IC}_\mc{Q} = \bigoplus_{w \in W_K v} \Q_w.
\]
 Hence $i_{T_K}^* \rres \varphi$ vanishes outside of $W_Kv$, and we have an inclusion 
\[
\im i_{T_K}^* \circ \rres \circ F \subseteq \im i_{T_K}^* \bigcap \bigoplus_{w \in W_K v} (R_w)^{\oplus q_w}.
\]
We claim that this is actually an equality
\begin{equation}
\label{eq: imF}
\im i_{T_K}^* \circ \rres \circ F = \im i_{T_K}^* \bigcap \bigoplus_{w \in W_K v} (R_w)^{\oplus q_w}.
\end{equation}
Indeed, recall that if we denote by $i$ the inclusion of $\mc{Q}$ into $X$ and $j$ the inclusion of the complement, we have 
\[
\Hom_K(\IC_\mc{Q}, \mc{G}) \simeq \Hom_{D^b_K(\mc{Q})} (\Q_\mc{Q}, i^! \mc{G})
\]
(see \eqref{eq: two}) and the short exact sequence
\begin{equation}
    \label{eq: useful ses}
    0 \rightarrow \mathbb{H}^\bullet_{T_K}(i_!i^! \mc{G}) \rightarrow \mathbb{H}^\bullet_{T_K}(\mc{G}) \rightarrow \mathbb{H}_{T_K}^\bullet(j_* j^* \mc{G}) \rightarrow 0
\end{equation}
given by standard parity sheaf arguments. Applying the localization theorem to every term in \eqref{eq: useful ses} implies that an element in $\mathbb{H}^\bullet_{T_K}(\mc{G})$ (represented by a tuple $(f_w)_{w \in W}$) is in the image of $\mathbb{H}_{T_K}^\bullet(i_!i^!\mc{G})$ in \eqref{eq: useful ses} if and only if $f_w = 0$ for all $w$ outside of $W_Kv$, so \eqref{eq: imF} holds.

Now let $\varphi \in \im G$. As both $\rres$ and $i_{T_K}^*$ are compatible with $(P^K, R)$-bimodule structures, we have that $(p \otimes 1 - 1 \otimes v^{-1}r)\in P^K \otimes_{R^W} R$ annihilates $i_{T_K}^* \rres \varphi = (p_w^{q_w}) _{w \in W} \in \bigoplus_{w \in W} (P_w)^{q_w}$ for all $p \in P^K$ and $r \in \phi^{-1}(p)$. By Lemma~\ref{lem: bimodule structure on fixed points}, for $p \in P^K$ and $r \in \varphi^{-1}(p)$, 
\begin{align}
\label{eq: left action}
p \otimes 1 \cdot i_T^* \rres \varphi &= (p p_w^{q_w})_{w \in W}, \text{ and }\\
\label{eq: right action}
1 \otimes {v^{-1} r} \cdot i_{T_K}^* \rres \varphi &= (\phi(wv^{-1}r) p_w^{q_w})_{w \in W}.
\end{align}
The actions (\ref{eq: left action}) and (\ref{eq: right action}) agree if and only if $w \in W_K v$. Hence
\begin{equation}
\label{eq: imG}
\im i_T^* \circ \rres \circ G = \im i_T^* \bigcap \bigoplus_{w \in W_K v}(P_w) ^{\oplus q_w}.
\end{equation}
As (\ref{eq: imF}) and (\ref{eq: imG}) align, we are done.
\end{proof}

This completes the proof of Theorem~\ref{thm: fully faithful}, and establishes an equivalence of categories between $\mc{M}_{LV}^0$ and $\mc{N}_{LV}^0$. By Theorem~\ref{thm: geometric categorification}, we have established the main result of our paper. 

\begin{corollary}
\label{cor: main theorem}
Let $\mc{N}_{LV}^0$ be the category in Definition~\ref{def: LV category} and $M_{LV}^0$ the $\bm{H}$-module in Theorem~\ref{lem: closed orbits generate the block}. As right $\bm{H}$-modules,
\[
[\mc{N}_{LV}^0]_\oplus \simeq M_{LV}^0.
\]
\end{corollary}

\begin{example}
\label{example: final formulas for SL(2,R)} $($An alternate basis of $M_{LV}$ for $SL_2(\R))$ The formulas in Example~\ref{example: bimodule category for SL(2,R)} describe the Lusztig--Vogan action on an alternate basis of $M_{LV}$ (the basis corresponding to $\IC$ sheaves). In particular, by passing to the split Grothendieck group in equations (\ref{eq: SBim action for SL2}), (\ref{eq: second SBim action for SL2}) and (\ref{eq: square of Bs}), we obtain the following equations in $M_{LV}$: 
\begin{align}
    \Q_{\mc{Q}_0} \cdot b_s &= m_s, \\
    \Q_{\mc{Q}_\infty} \cdot b_s &= m_s, \\ 
    m_s \cdot b_s &= vm_s + v^{-1} m_s,
\end{align}
where $m_s = v\Q_{\mc{Q}_0} + v \Q_{\mc{Q}_\infty} + \Q_\mc{O} \in M_{LV}$ corresponds to $[B_s] \in [\mc{N}_{\mathrm{LV}}^0]_\oplus$ (resp.\ $[\IC(X, \Q_\mc{O})] \in [\mc{M}_\mathrm{LV}^0]_\oplus$) under the isomorphism $[\mc{N}_\mathrm{LV}^0]_\oplus \simeq M_\mathrm{LV}$ (resp.\ $[\mc{M}_\mathrm{LV}^0]_\oplus \simeq M_\mathrm{LV}$). We note that in this example, the formulas describing the Lusztig--Vogan action are considerably cleaner in this basis. (Compare to equations (\ref{eq: T_s formulas for SL2}).) 
\end{example}


\appendix
\section{Dual pairs in monoidal categories}
\label{app: dual pairs}

In this appendix we recall the definition of a dual pair in a monoidal category and describe a procedure for constructing adjunctions from dual pairs. These results provide the categorical framework for the proof of Lemma~\ref{lem: hypercohomology commutes with adjunctions}. 

Recall that a {\em monoidal category} consists of the data of a category $\mc{A}$, along with a bifunctor $\otimes:\mc{A} \times \mc{A} \rightarrow \mc{A}$, a unit object $\mathbbm{1}$, and natural isomorphisms 
\begin{align*}
\alpha_{X, Y, Z}: (X \otimes Y) \otimes Z &\xrightarrow{\sim} X \otimes (Y \otimes Z) \hspace{3mm} \text{(associator)}, \\
\lambda_X: \mathbbm{1} \otimes X &\xrightarrow{\sim} X \hspace{3mm} \text{(left unitor), and } \\
\rho_X: X \otimes \mathbbm{1} &\xrightarrow{\sim} X \hspace{3mm} \text{(right unitor)}
\end{align*}
satisfying the pentagon and unit axiom \cite[Ch. 1]{TensorCategoriesBook}.

\begin{definition}
\label{def: dual pair}
A {\em dual pair} in a monoidal category $(\mc{A}, \otimes, \alpha, \mathbbm{1}, \rho, \lambda)$ is a pair $(X, X^*)$ of objects in $\mc{A}$ such that there exist morphisms 
\[
\beta: \mathbbm{1} \rightarrow X \otimes X^* \text{ (co-evaluation)}, \hspace{3mm} \gamma:X^* \otimes X \rightarrow \mathbbm{1} \text{ (evaluation)}
\]
making the diagrams 
\begin{equation}
  \label{eq: dual pair axiom 1}
\begin{tikzpicture}[commutative diagrams/every diagram]
\node (P0) at (90:2.3cm) {$\mathbbm{1} \otimes X$};
\node (P1) at (90+72:2cm) {$(X \otimes X^*) \otimes X$} ;
\node (P2) at (90+2*72:2cm) {\makebox[5ex][r]{$X \otimes (X^* \otimes X)$}};
\node (P3) at (90+3*72:2cm) {\makebox[5ex][l]{$X \otimes \mathbbm{1}$}};
\node (P4) at (90+4*72:2cm) {$X$};
\path[commutative diagrams/.cd, every arrow, every label]
(P0) edge node[swap] {$\beta \otimes X$} (P1)
(P1) edge node[swap] {$\alpha_{X, X^*, X}$} (P2)
(P2) edge node {$X \otimes \gamma$} (P3)
(P3) edge node[swap] {$\rho_X$} (P4)
(P0) edge node {$\lambda_X$} (P4);
\end{tikzpicture}
\end{equation}
and 
\begin{equation}
  \label{eq: dual pair axiom 2}
\begin{tikzpicture}[commutative diagrams/every diagram]
\node (P0) at (90:2.3cm) {$X^* \otimes \mathbbm{1}$};
\node (P1) at (90+72:2cm) {$X^* \otimes (X \otimes X^*) $} ;
\node (P2) at (90+2*72:2cm) {\makebox[5ex][r]{$(X^* \otimes X ) \otimes X^* $}};
\node (P3) at (90+3*72:2cm) {\makebox[5ex][l]{$ \mathbbm{1}\otimes X^*$}};
\node (P4) at (90+4*72:2cm) {$X^*$};
\path[commutative diagrams/.cd, every arrow, every label]
(P0) edge node[swap] {$X^* \otimes \beta$} (P1)
(P1) edge node[swap] {$\alpha_{X^*, X, X^*}$} (P2)
(P2) edge node {$\gamma \otimes X^*$} (P3)
(P3) edge node[swap] {$\lambda_{X^*}$} (P4)
(P0) edge node {$\rho_{X^*}$} (P4);
\end{tikzpicture}
\end{equation}
commute.
\end{definition}

Let $(\mc{A}, \otimes, \alpha, \mathbbm{1}, \rho, \lambda)$ be a monoidal category. Recall that a right $\mc{A}$-module consists of the data of a category $\mc{M}$, equipped with a bifunctor $\otimes: \mc{M} \times \mc{A} \rightarrow \mc{M}$ and natural isomorphisms 
\begin{align*}
    \alpha_{M, X, Y}: (M \otimes X) \otimes Y &\xrightarrow{\sim} M \otimes (X \otimes Y) \text{ (associator)}\\
    \rho_M : M \otimes \mathbbm{1} &\xrightarrow{\sim} M \text{ (right unitor) }
\end{align*}
satisfying the pentagon axiom and the unit axoim \cite[Ch. 7]{TensorCategoriesBook}. 
\begin{lemma}
\label{lem: dual pairs give adjunctions}
Let $(X, X^*, \beta: \mathbbm{1} \rightarrow X \otimes X^*, \gamma: X^* \otimes X \rightarrow \mathbbm{1})$ be a dual pair in a monoidal category $(\mc{A}, \otimes, \alpha, \mathbbm{1}, \rho, \lambda)$, and let $\mc{M}$ be a right $\mc{A}$-module. Define functors 
\[
\begin{tikzcd}
\mc{M} \arrow[r, shift left, "- \otimes X"] & \mc{M} \arrow[l, shift left,  " - \otimes X^*"]
\end{tikzcd}
\]
and morphisms 
\begin{align*}
    &\eta_M: M \xrightarrow{\rho_M^{-1}} M \otimes \mathbbm{1} \xrightarrow{M \otimes \beta} M \otimes X \otimes X^* \\
    &\epsilon_M: M \otimes X^* \otimes X \xrightarrow{M \otimes \gamma}  M \otimes \mathbbm{1} \xrightarrow{\rho_M} M 
\end{align*}
for all objects $M$ in $\mc{M}$. Then $(- \otimes X, - \otimes X^*)$ is an adjoint pair of functors, with unit $\eta: id \rightarrow - \otimes X \otimes X^*$ and counit $\epsilon: - \otimes X^* \otimes X \rightarrow id$.
\end{lemma}
\begin{proof}
Naturality of $\rho$ guarantees that the collection of morphisms $\{\eta_M\}_{M \in \mc{M}}$ and $\{ \epsilon_M\}_{M \in \mc{M}}$ form natural transformations. To see that the zig-zag diagrams commute, we factor them as follows (ignoring associativity isomorphisms):
\begin{equation}
    \label{eq:zig-zag 1}
\begin{tikzcd}
M \otimes X \arrow[r, "\rho_M^{-1} \otimes X"] \arrow[dr, "id"'] & M \otimes \mathbbm{1} \otimes X \arrow[r, "M \otimes \beta \otimes X"] \arrow[d, "M \otimes \lambda_X"] & M \otimes X \otimes X^* \otimes X \arrow[d, "M \otimes X \otimes \gamma"] \\
& M \otimes X \arrow[dr, "id"'] & M \otimes X \otimes \mathbbm{1} \arrow[l, "M \otimes \rho_X"'] \arrow[d, "\rho_{M \otimes X}"] \\
& & M \otimes X 
\end{tikzcd}
\end{equation}
\begin{equation}
    \label{eq:zig-zag 2}
\begin{tikzcd}
M \otimes X^* \arrow[r, "\rho_{M \otimes X^*}^{-1}"] \arrow[dr, "id"'] & M \otimes X^* \otimes \mathbbm{1} \arrow[r, "M \otimes X^* \otimes \beta"] \arrow[d] & M \otimes X^* \otimes X \otimes X^* \arrow[d, "M \otimes \gamma \otimes X^*"] \\
& M \otimes X^* \arrow[dr, "id"'] & M \otimes \mathbbm{1} \otimes X^* \arrow[l] \arrow[d, "\rho_M \otimes X^*"] \\
& & M \otimes X^* 
\end{tikzcd}
\end{equation}

In (\ref{eq:zig-zag 1}), the commutativity of the top left triangle follows from the unit axiom in $\mc{M}$, commutativity of the top right square follows from the first dual pair axiom (\ref{eq: dual pair axiom 1}), and commutativity of the bottom right triangle follows from the pentagon axiom and unit axiom of $\mc{M}$, as in \cite[Prop.\ 2.2.4]{TensorCategoriesBook}. 

In (\ref{eq:zig-zag 2}), commutativity of the top left triangle follows from the pentagon axiom and unit axiom in $\mc{M}$, as in \cite[Prop.\ 2.2.4]{TensorCategoriesBook}, commutativity of the top right square follows from the second dual pair axiom (\ref{eq: dual pair axiom 2}), and commutativity of the bottom right triangle follows from the unit axiom in $\mc{M}$. 
\end{proof}

\bibliographystyle{alpha}
\bibliography{LV}

\end{document}